\documentclass{amsart}

\usepackage{graphicx}
\usepackage[utf8]{inputenc}
\usepackage[authoryear,round]{natbib}

\input{commands.tex}

\pgfplotsset{compat=1.18}

\begin{document}

\title[Wasserstein Least Squares]{Wasserstein Least Squares: A Canonical Regression Method for
  Probability Distributions}

\author{Uriel Mart\'inez Le\'on}
\address{Courant Institute of Mathematics, Computing, and Data Science,
  New York University, New York, NY 10011, USA}
\email{uriel.leon@nyu.edu}

\author{Jonathan Niles-Weed}
\address{Courant Institute of Mathematics, Computing, and Data Science,
  New York University, New York, NY 10011, USA}
\email{jnw@cims.nyu.edu}


\begin{abstract}
We perform a mathematical and statistical analysis of the Wasserstein least squares problem, a regression method for vector-valued covariates and distribution-valued responses.
Our proposal contrasts with other distributional regression methods by having a direct interpretation in terms of random variables, as a nonparametric analogue of the classic random-effects model.
On the mathematical side, we use a strategy of \citet{lavenantLiftingFunctionalsDefined2023} to show that Wasserstein least squares is the \textit{canonical} extension of Euclidean least squares to the space of probability distributions from the perspective of convex analysis; this viewpoint gives rise to multimarginal and dual formulations of the Wasserstein least squares problem, extending a similar theory for Wasserstein barycenters.
We perform a statistical analysis of the Wasserstein least squares problem under the template deformation model, showing, surprisingly, that estimation is possible at the $n^{-1/2}$ rate.
As a special case, we obtain improved rates of estimation for Wasserstein barycenters, which are an exponential improvement over those established by~\cite{ahidar-coutrixConvergenceRatesEmpirical2020}.
Finally, we propose a heuristic particle method for Wasserstein least squares and use it to conduct a novel analysis of large-scale demographic data from the RAND Health and Retirement Study.
\end{abstract}

\maketitle

\section{Introduction}
Legendre's 1805 publication of the method of least squares is a turning point in the history of mathematical statistics.
In contrast to later justifications of the subject---particularly that of Gauss, which appeared in 1809---Legendre motivates his method not by a probabilistic model but by practical considerations.
(See \citealp{stigler}, for a historical review of the development and popularization of the least squares method.)
For Legendre, the usefulness of least squares was justified by several interrelated facts:
\begin{enumerate}
	\item The naturalness and interpretability of linear relationships between covariates and responses;
	\item The simplicity of the first-order optimality conditions for the minimization problem (the \textit{normal equations}); and
	\item The consistency of the least squares method with the method of averaging: in the case of repeated measurements of the same object, least-squares minimization reduces to computation of the arithmetic mean.
\end{enumerate}
Later, Gauss's probabilistic developments connected the least-squares objective to statistical estimation in the presence of random errors whose law bears his name.

The power and ubiquity of linear least squares has led to the growth of a large literature that seeks to extend its reach beyond the Euclidean setting in which it was initially proposed, for instance, by allowing the independent variables (\textit{covariates}) and dependent variables (\textit{responses}) to take values in infinite-dimensional Hilbert spaces, on manifolds, or in general metric spaces \citep{ramsay_functional_2005,hsing_theoretical_2015,PetersenMuller2019}.
From a statistical perspective, generalizing beyond the Euclidean setting requires specifying both a method (analogous to Legendre's least squares proposal) and a statistical model (analogous to Gauss's theory of random errors) under which the method returns sensible results.

The goal of this work is to propose and study a version of the least squares method, with vector-valued covariates but responses which take values in the space of probability distributions.
We are far from the first to propose such a method (see \cref{sec:related_work} for an in-depth summary), but we take a different approach to most existing statistical work on the problem.
Our philosophy can be summarized by saying that rather than view probability distributions as elements of an arbitrary metric space, we \textit{lift} standard linear least squares to the space of probability distributions by allowing our regression coefficients to be random variables.
(This is analogous to the construction of the Wasserstein space by lifting a given base metric to the space of random variables, \citealp{savareSimpleRelaxationApproach}).

Given covariates $\bx \in \RR^p$ and a response $y \in \RR$,
Legendre's proposal begins with the structural assumption that
\begin{equation}\label{eq:linear_1d}
	y \approx \bm{\beta}^\top \bx\,,
\end{equation}
for an unknown vector of parameters $\bm{\beta} \in \RR^p$.
More generally, for responses in $\RR^d$, \eqref{eq:linear_1d} takes the form
\begin{equation}\label{eq:linear_ansatz}
	\by \approx \bB^\top \bx
\end{equation}
for an unknown $\bB \in \RR^{p \times d}$.

In its simplest form, our proposal is based on the following analogous structural assumption: given covariates $\bx \in \RR^p$ and a \textit{distribution-valued response} $\nu \in \cP(\RR^d)$, we suppose
\begin{equation}\label{eq:measure_ansatz}
	\nu \approx \mathrm{Law}(\bB^\top \bx)\,,
\end{equation}
where $\bB$ is a random variable.\footnote{Our statistical analysis adopts the fixed design assumption under which $\bx$ is deterministic, but our results also give valid conditional guarantees under random design as well if we suppose that $\bB$ is independent of $\bx$.
	In this context,~\eqref{eq:measure_ansatz} should be interpreted as $\nu \approx \mathrm{Law}(\bB^\top \bx \mid \bx)$.}

From the perspective of applications,~\eqref{eq:measure_ansatz} is a natural model in a number of domains.
In econometrics, in the context of repeated cross-sectional data, an analyst observes covariate--response pairs for individuals stratified into groups (for example, survey waves, age bands, regions), but the individuals sampled vary across groups, so that within-individual variation cannot be tracked.
Following~\citet{Deaton1985}, the standard econometric approach is to aggregate respondents into cohorts and to fit a linear regression to the resulting cohort-level \emph{averages}, yielding the so-called pseudo-panel estimator (see, e.g., \citealp{Verbeek}).
In~\eqref{eq:measure_ansatz}, we enrich this approach by modeling the full conditional distribution $\nu$ rather than only its first moment.

In computational biology for single-cell data, \citet{Bunne} propose an optimal-transport method for examining the impact of treatment interventions and genetic perturbations on cell populations.
They model the baseline state of a population of cells by a probability measure $\mu \in \cP(\RR^d)$, and consider the perturbed states $\nu_1, \dots, \nu_n \in \cP(\RR^d)$ for a population of cells that obtains after applying $n$ different perturbations of the cells' environment (for example, introducing a drug or using CRISPR to perform a genetic ``knockout'').
\citet{Bunne} posit the existence of a parametrized family of optimal transport maps $\{T_c\}_{c \in C}$ indexed by ``contexts'' $C$ corresponding to the medical or genetic interventions in question, which satisfy $\nu_i = (T_{c_i})_\# \mu$.
Supposing that context $c_i$ is identified with a vector $\bx_i$ in $\RR^p$ and that $T_{c_i} = \sum_{j=1}^p x_{ij} T_j$ for a dictionary $T_1, \dots, T_p$ of potential transport maps, we obtain our model by setting $\mathrm{Law}(\bB) = (T_1, \dots, T_p)_\# \mu$.

Finally, in demography, the relation between observed cohort distributions of outcomes such as mortality, morbidity, or fertility, and the unobserved heterogeneity of the underlying population is a long-standing concern.
In their seminal paper, \citet{VaupelMantonStallard1979} introduced the \emph{frailty model}, in which each individual is endowed with an unobserved random scalar $z$ that multiplicatively scales their mortality hazard, so that the cohort-level survival distribution arises by mixing over the population-level distribution of $z$.
They argued that ignoring this heterogeneity systematically biases standard estimates of life expectancy and of rates of individual aging.
Writing $\bx \in \RR^p$ for cohort-level covariates (year of birth, sex, country, education, and the like) and $\bm{\beta}$ for a random vector of individual-level traits, the principle that cohort distributions reflect a population-level distribution over individual parameters takes its simplest linear form in model~\eqref{eq:measure_ansatz}, with $\nu = \mathrm{Law}(\bm{\beta}^\top \bx)$.

More generally, the relation~\eqref{eq:measure_ansatz} may be viewed as a random-effects model~\citep{Laird1982}.
Such models are typically written in the form
\begin{equation}\label{eq:random_effects}
	\by_{ij} = \bm{\beta}^\top \bm{w}_{ij} + \bm{b}_i^\top \bz_{ij} + \eps_{ij}\,,
\end{equation}
where $\bm{\beta}$ are fixed (deterministic) effects, $\bm{b}_i \sim \cN(0, \Sigma)$ are random effects, and $\eps_{ij} \sim \cN(0, \sigma^2)$ is independent noise.
We observe that~\eqref{eq:random_effects} is a particular case of~\eqref{eq:measure_ansatz}, when $\bB$ has a (possibly non-centered) Gaussian distribution and the approximation symbol in~\eqref{eq:measure_ansatz} reflects the presence of additive noise.

However, our model~\eqref{eq:measure_ansatz} goes beyond~\eqref{eq:random_effects} in several ways.
First, we make \textit{no} parametric assumptions on $\bB$.
Our model is therefore closer to the fully nonparametric random coefficient model whose study was initiated by~\citet{BeranHall1992} and developed in the multivariate setting by~\citet{Hoderlein2010}.
Second, our precise statistical setting (see \cref{subsec:intro_template}) adopts the ``template-deformation model'' \citep{Boissard2015Distribution}, in which the approximation symbol in~\eqref{eq:measure_ansatz} reflects possibly nonlinear distortions of $\bB^\top \bx$.
We recover classical additive noise as a very special case, but our statistical theory is more general.

Paralleling Legendre's least squares method for~\eqref{eq:linear_ansatz}, we propose to fit covariate--response data $(\bx_1, \nu_1), \dots, (\bx_n, \nu_n)$ under~\eqref{eq:measure_ansatz} by solving
\begin{equation}\label{eq:basic_wls}
		\min_{\bB} \sum_{i=1}^n W_2^2(\nu_i, \mathrm{Law}(\bB^\top \bx_i))
\end{equation}
Here, the minimization is taken over all $\RR^{p \times d}$-valued random variables $\bB$.
This procedure extends an approach studied by~\citet{KarimiRipaniGeorgiou2021} and~\citet{KarimiGeorgiou2021} for one-dimensional covariates.
Following~\citet{graph-structured-cost}, who studied algorithmic aspects of that approach, we call~\eqref{eq:basic_wls} \textit{Wasserstein least squares}.

From the perspective of practical usefulness, Wasserstein least squares captures the important benefits of its Euclidean analogue:
\begin{enumerate}
	\item Via~\eqref{eq:measure_ansatz}, the solution to~\eqref{eq:basic_wls} has a natural interpretation \textit{at the level of random variables}.
	This distinguishes our proposal from many other variants of distributional regression and is crucial to applications.
	\item Convex duality implies that the optimal solution to~\eqref{eq:basic_wls} satisfies a simple first-order optimality condition directly parallel to the normal equations of Euclidean least squares.
	As in the Euclidean case, these normal equations show that $\bB$ can be interpreted as a projection of the data onto the span of the design matrix.
	\item In the case of repeated measurements (i.e., when the covariates are duplicated), Wasserstein least squares reduces to the well-studied Wasserstein barycenter problem.
\end{enumerate}

Our first main contribution is to give a new mathematical argument that Wasserstein least squares is a principled formulation of regression for probability distributions.
While~\eqref{eq:basic_wls} may seem like little more than a formal analogue of the Euclidean least squares problem, we show that this is far from the case: we give an axiomatic justification for~\eqref{eq:basic_wls}, establishing it as the canonical \textit{lifting} of ordinary least squares to the space of probability measures, in the sense of~\cite{lavenantLiftingFunctionalsDefined2023}, by showing it to be the unique functional of the marginal measures $\nu_1, \dots, \nu_n$ satisfying certain desirable properties.
In addition to providing evidence that~\eqref{eq:basic_wls} is a fundamental question, this perspective allows us to develop a convex duality theory for~\eqref{eq:basic_wls}; this theory gives us both algorithmic and statistical insights into properties of optimizers of~\eqref{eq:basic_wls}.

To perform a statistical analysis of~\eqref{eq:basic_wls}, we formalize a model for~\eqref{eq:measure_ansatz} with two sources of noise.
First, to capture the assumption that $\nu$ only approximately matches the law of $\bB^\top \bx$, we adopt the template-deformation model proposed by~\cite{Boissard2015Distribution}.
The template-deformation model has emerged as a standard probabilistic framework for distribution-valued regression and the Wasserstein barycenter problem.
In this model, the measure corresponding to the law of $\bB^\top \bx_i$ is corrupted by the application of a random nonlinear warping of $\RR^d$; formally, we have $\nu_i = (T_i)_\# \mathrm{Law}(\bB^\top \bx_i)$ for an unobserved random cyclically monotone function $T_i$.
Second, to capture the assumption that $\nu$ is only partially observed, we assume that the statistician only has access to i.i.d.\ samples from $\nu$, and therefore that the responses $\nu_1, \dots, \nu_n$ in~\eqref{eq:basic_wls} are replaced with nonparametric estimators $\widehat \nu_1, \dots, \widehat \nu_n$ constructed on the basis of samples.

Our full statistical model incorporates both template-deformation noise and sampling noise.
We obtain finite-sample estimation guarantees for Wasserstein least squares under this model.
In fact, since the Wasserstein barycenter problem is a special case of our model, our bounds imply new rates of estimation for Wasserstein barycenters as well.
These results complement the rates obtained by~\citet{LeGouic2022FastCE}, which hold only under very restrictive conditions on the deformations, and provide an exponential improvement on the best available general results for the template deformation model~\citep{ahidar-coutrixConvergenceRatesEmpirical2020}.

We give extensive empirical evidence for the utility of our method.
We validate our approach on a number of simulated examples and conduct a new, in-depth analysis of a large public health dataset obtained from the RAND Health and Retirement Study~\citep{RANDHRS2022V1}.
Our method extracts insights from this study that other methods for distribution regression do not.

\subsection{Our results and approach}
\subsubsection{A canonical regression method for probability distributions}
Given covariate-response pairs $(\bx_1, \nu_1), \dots, (\bx_n, \nu_n)$, most other approaches to regression adopt the following roadmap: 1.\ Identify $\nu_1, \dots, \nu_n$ with elements of a metric space $\cM$. 2.\ Develop an analogue of the least-squares problem $\min_{\bB} \|\by_i - \bB^\top \bx_i\|^2$ by a) replacing the Euclidean distance with the metric on $\cM$ and b) replacing the linear function $\bx \mapsto \bB^\top \bx$ with a suitable generalization to $\cM$.
We review these approaches in more detail in \cref{sec:related_work}.

The statistical properties and implications of this model depend to a great extent on the design choices in the above description. (See \citealp{Song2026-wi,petersen_econometrics} for two recent review articles.)
Moreover, some approaches (such as those based on mean embeddings in reproducing kernel Hilbert spaces) require choosing additional tuning parameters.
While this flexibility can be appealing in some applications, it raises the question of whether there exists a regression method for this setting that does not involve any additional choices.

We take a different approach, inspired by~\cite{lavenantLiftingFunctionalsDefined2023}, which identifies a canonical method through an axiomatic argument.
The standard Euclidean least squares problem is based on evaluating the function
\[E(\by_1, \dots, \by_n) = \min_{\bB \in \RR^{p \times d}} \frac 1n \sum_{i=1}^n \|\by_i - \bB^\top \bx_i\|^2.\]
We propose to define a new functional $\cE(\nu_1, \dots, \nu_n)$ which is uniquely characterized as the \textit{convex extension}~\citep{Rockafellar1970,Hiriart-Urruty2011-dv} of $E$ to the space of probability measures.
The construction, also called the convex envelope or convex hull, is the analyst's standard tool for extending convex functions to larger spaces.
Though this definition is not explicit, we show that in fact  $\cE(\nu_1, \dots, \nu_n)$ is \textit{exactly} equal to~\eqref{eq:basic_wls}.
This provides a new, canonical construction of linear regression on Wasserstein space, and reveals Wasserstein least squares to be, from the viewpoint of convex analysis, the correct analogue of Euclidean least squares for distribution-valued responses.

Our axiomatic justification of Wasserstein least squares extends beyond the linear regression case and applies to general nonparametric regression problems involving distributions, see \cref{sec:lifting}.

\subsubsection{Multimarginal formulations and convex duality}
A benefit of the formulation of Wasserstein least squares that we have adopted is its close connection to convex analysis and the theory of optimal transportation.
We obtain several results generalizing the theory of Wasserstein barycenters developed by~\cite{aguehBarycentersWassersteinSpace2011}.
First, we show that~\eqref{eq:basic_wls} possesses an equivalent ``multimarginal'' formulation as an optimization problem over the space of couplings of the $n$ measures $\nu_1, \dots, \nu_n$.
This parallels the analogous multimarginal formulation of the Wasserstein barycenter problem, and also clarifies the interpretation of our procedure: as we show in \cref{prop:linear_explicit}, Wasserstein least squares finds a coupling between the response measures that maximizes the explained variance of the resulting linear model.

Like the optimal transportation problem itself, we show that the Wasserstein least squares problem also possesses a dual formulation (\cref{thm:basic_dual}).
This duality theory provides a criterion under which~\eqref{eq:basic_wls} has a unique solution and a natural first-order optimality condition for~\eqref{eq:basic_wls}: writing $\nabla \varphi_i$ for the Brenier map from $\mathrm{Law}(\bB^\top \bx_i)$ to $\nu_i$, \cref{theorem:W-normal-equations} shows that the optimal $\bB$ satisfies
\begin{equation}\label{eq:intro_normal_equations}
	\sum_{i=1}^n \bx_i (\nabla \varphi_i(\bB^\top \bx_i))^\top =  \sum_{i=1}^n \bx_i \bx_i^\top \bB \quad \text{a.s.}
\end{equation}
Note that $\nabla \varphi_i(\bB^\top \bx_i) \sim \nu_i$.
This equation is the exact Wasserstein analogue of the celebrated normal equations for linear regression:
\begin{equation*}
	\bX^\top \bY = \bX^\top \bX \bB\,.
\end{equation*}

As in the case of Euclidean least squares, the normal equations offer a powerful geometric interpretation of Wasserstein least squares: the optimal $\bB$ is the projection of the random variables $(\nabla \varphi_i(\bB^\top \bx_i)^\top)_{i=1}^n$ onto the column space of $\bX$.

\subsubsection{Algorithms}
\begin{sloppypar}
Like many optimization problems involving measures, the Wasserstein least squares problem is computationally challenging.
We propose a gradient flow algorithm whose stationary points are precisely solutions to~\eqref{eq:intro_normal_equations}.
This algorithm has a concrete interpretation via~\eqref{eq:intro_normal_equations}: at each step, it updates the law of the random variable $\bB$ to bring it closer to the projection of $(\nabla \varphi_i(\bB^\top \bx_i)^\top)_{i=1}^n$.
As with many optimization problems in Wasserstein space~\citep{ChewiNilesWeedRigollet2025,chewiGradientDescentAlgorithms2020c,altschulerAveragingBuresWassersteinManifold2021}, the objective function in~\eqref{eq:basic_wls} is not geodesically convex, and we cannot rule out the possibility of local minima; nevertheless, our approach is successful in practice.
We view obtaining rigorous convergence guarantees for our algorithm as an important open problem.
\end{sloppypar}

This algorithm is particularly tractable in two special cases: when each response $\nu_i$ is a Gaussian distribution (with arbitrary mean and covariance), or when each response is a univariate distribution.
In the former case, we show that the optimal $\bB$ may be chosen to be Gaussian as well.
That fact implies that the optimization problem may be formulated on the finite-dimensional Bures--Wasserstein manifold.
The benefit of this formulation is that the gradient updates may be computed in closed form.
In the latter case, when each $\nu_i$ is univariate, the special structure of $\cP(\RR)$ makes the Brenier maps in~\eqref{eq:intro_normal_equations} easy to compute.
As a result, the resulting gradient descent steps are easy to implement via particle methods, at a cost of $O(n M \log M)$ per iteration, where $M$ is the number of particles.

\subsubsection{Rates of estimation under the template deformation model}\label{subsec:intro_template}
Having developed the Wasserstein least squares approach, we analyze its performance under the template deformation model~\citep{Boissard2015Distribution,zemelFrechetMeansProcrustes2018a,panaretos_invitation_2020}.
We suppose that for $i = 1, \dots, n$ and $j = 1, \dots, m$, we observe samples
\begin{equation}\label{eq:intro_model}
	Y_{i, j} = \nabla \phi_i (\bB_{i, j}^\top \bx_i)\,,
\end{equation}
where $\bB_{i, j}$ are independent copies of an unobserved random variable $\bB$, and $\nabla \phi_i$ are independent copies of the gradient of an unobserved random convex function, satisfying $\E[\nabla \phi_i(\by)] = \by$ for all $\by \in \RR^d$.
When $\nabla \phi_i$ is the function $\by \mapsto \by + \eps_i$ for mean-zero $\eps_i$, \eqref{eq:intro_model} is a standard random-coefficient linear model~\cite{longford_random_1993}.
In general, however, \eqref{eq:intro_model} allows for noise in the form of nonlinear ``warpings'' of $\RR^d$.

Under the assumption that $\phi_i$ is almost surely smooth and strongly convex, we show that the Wasserstein least squares solution $\widehat \bB$  satisfies
\begin{equation}\label{eq:intro_error}
	\E \left[\frac 1n \sum_{i=1}^n W_2^2(\mathrm{Law}(\bB^\top \bx_i), \mathrm{Law}(\widehat{\bB}^\top \bx_i))\right] \lesssim n^{-1/2} + m^{-2/d}\,,
\end{equation}
for $d > 4$, where the implicit constant depends on $p$, $d$, and the smoothness and strong convexity bounds for the random deformations.
The error bound in~\eqref{eq:intro_error} reflects the two sources of noise in the model: the first term reflects the problem of estimating the law of $\bB$ in the presence of the noise arising from the template deformations $\nabla \phi$, and the second reflects the sampling noise inherent in~\eqref{eq:intro_model}.
When $m \to \infty$ (which corresponds to the case where $\nu_i = (\nabla \phi_i)_\# \mathrm{Law}(\bB^\top \bx_i)$ is fully observed), \eqref{eq:intro_error} shows that the error decays at the $n^{-1/2}$ rate.

As a special case, taking $p = 1$ and $\bx_1 = \dots = \bx_n$,~\eqref{eq:intro_error} yields new rates for the problem of estimating Wasserstein barycenters in the template deformation model.
Under our assumptions, the best general results for estimation of Wasserstein barycenters, due to~\citet{ahidar-coutrixConvergenceRatesEmpirical2020}, prove a rate of convergence at the rate $n^{-1/d}$, exhibiting a steep curse of dimensionality.
In the very special case that each $\phi_i$ is $\alpha$-strongly convex and $\beta$-smooth where $\beta - \alpha < 1$, \citet{LeGouic2022FastCE} prove the remarkable fact that the error decreases at the rate $n^{-1}$, with no dependence on the dimension.
However, their bounds become vacuous when $\beta - \alpha \geq 1$.
Our bound is an exponential improvement over the general state of the art; though we fail to recover the sharp bound of~\citet{LeGouic2022FastCE} when $\beta - \alpha <1$, our $n^{-1/2}$ rate applies to strongly convex and smooth deformations for any $\alpha, \beta  \in (0, \infty)$.

\subsubsection{Empirical results}
We apply this framework to Body Mass Index data from the RAND Health and Retirement Study~\citep{RANDHRS2022V1}, a nationally representative longitudinal survey of approximately $45{,}000$ U.S.\ adults observed across 16 biennial waves (1992--2022)~\citep{gutin2018bmi,life-course,Ward2019Projected}.
Treating each demographic cell (birth cohort $\times$ survey wave $\times$ gender) as a distributional observation yields $n = 164$ BMI distributions, which we model with the quadratic design
$\bx_i = (1, \widetilde a_i, \widetilde a_i^2, \widetilde c_i, \widetilde g_i)^\top$
encoding normalized age, birth cohort, and gender.
The template deformation noise model~\eqref{eq:model} makes no distributional assumption on the individual-level heterogeneity beyond the regularity conditions C1--C3, so the analysis is free of the Gaussian random-effect assumption that underpins classical mixed-effects analyses of these data~\citep{Laird1982,life-course}.

We estimate $Q^\star$ using Algorithm~\ref{alg:wls-gd-particle} with $M = 20{,}000$ particles.
The estimator returns the full joint distribution over the coefficient vector
\[\bbeta = (\beta_0, \beta_{\rm age}, \beta_{{\rm age}^2}, \beta_{\rm cohort}, \beta_{\rm gender})^\top,\]
recovering the coupling structure of the random effects without parametric assumptions.
This joint distribution enables a form of conditional inference that is unavailable in prior-distributional regression methods.
Since each particle $\bbeta_m \sim \widehat{Q}$ represents a plausible individual trajectory under the fitted model, conditioning on an observed BMI value amounts to asking which trajectories in the estimated population are consistent with that observation; the retained subset then predicts the future distribution of BMI for that subpopulation.
This conditional forecast is meaningful precisely because $\widehat{Q}$ captures the true population heterogeneity, a claim that we assess directly by comparing predicted trajectory bands against the actual observed HRS cohort paths in \cref{sec:experiments}.
We compare our approach to Global Fréchet regression~\citep{PetersenMuller2019} and show that Wasserstein least squares substantially outperforms it in coverage under sequential conditioning, and is the only method capable of representing all clinically relevant outcome scenarios; see \cref{sec:experiments} for the full analysis.

We further validate Wasserstein least squares on two synthetic settings (\cref{appendix:synthetic}) in which the Wasserstein least squares model class strictly contains the Fréchet family as a special case, demonstrating that the empirical advantage is a structural consequence of the broader scope of our model.

\subsection{Related work}\label{sec:related_work}
There is a significant literature on designing regression methods for distribution-valued data.
One initial challenge is a lack of linear structure on the space of probability measures, which necessitates either mapping distributions into a linear space (as is typical in functional data analysis, \citealp{hsing_theoretical_2015}), or simply treating them as ``random objects'' in a metric space~\citep{muller_peter_2016}.
There are challenges with either approach.
In the former, transformation-based, approach, the mapping is rarely invertible, which makes interpreting the output of estimation procedures in the linear space challenging; in the latter, metric space, approach, the loss of linear structure can cause difficulties for computation and theoretical analysis.

As \citet{petersen_econometrics} make clear, regression methods under either approach may be unified using the \textit{global Fr\'echet regression} framework of \citet{PetersenMuller2019}: given a metric $d$ defined between probability measures and data $(\bx_1, \nu_1), \dots, (\bx_n, \nu_n)$, the predicted output for a new covariate $\bx$ is given by
\begin{equation}\label{eq:frechet_regression_intro}
	\argmin_{\nu} \frac 1n \sum_{i=1}^n w_i(\bx) d^2(\nu_i, \nu)\,,
\end{equation}
where $w_i$ are (possibly negative) weights designed so that, if $\nu_i$ and $\nu$ are replaced by vectors and $d$ is the Euclidean metric, the solution to~\eqref{eq:frechet_regression_intro} coincides with the solution to standard Euclidean least squares.
Different choices of the metric $d$ give rise to different generalizations of this approach.

When $d$ is the Wasserstein metric on the space of univariate measures, we obtain the Wasserstein regression method studied by \citet{PetersenMuller2019} and \citet{petersen_wasserstein-f-tests_2021}.
These works propose a statistical model under which they can study estimation and testing problems on the Wasserstein space; the focus on one dimensional measures plays an essential role in~\citet{petersen_wasserstein-f-tests_2021} in particular, since on $\RR$ the Wasserstein metric agrees with an $L^2$ metric on quantile functions.

\begin{sloppypar}
When the covariates are also distributions, one obtains a distribution-to-distribution regression problem.
Such problems have been studied in detail by \citet{Panaretos_regression} and \citet{Chen03042023}.
Once again, a key challenge lies in defining a suitable class of maps between covariates and responses.
\citet{Chen03042023} do so by employing the linear structure of the tangent space to a measure in the Wasserstein geometry, whereas \citet{Panaretos_regression} study a ``shape constraint'' in the form of a monotonicity assumption.
A detailed comparison between these two models is provided in Section 3.3 of \cite{Panaretos_regression}.
\end{sloppypar}

What these methods have in common is their focus on \textit{probability measures} as the key objects of interest, either as covariates or responses.
As explained above, our approach is philosophically different: the primary interpretation of our methods is in terms of \textit{random variables}.
One benefit of our perspective is that it leverages a unique feature of the Wasserstein distance as opposed to other metrics, namely, that couplings between random variables play a central role.
As we show in \cref{sec:experiments}, it is possible to interpret the approach of \citet{PetersenMuller2019} in terms of random variables in special cases, but their model corresponds to a very rigid constraint on these variables' joint distribution.
For example, in the univariate Gaussian case with a single covariate, global Fr\'echet regression corresponds to the model $Q_x=\mathcal{N}(\mu_0+\beta x,\,(\sigma_0+\gamma x)^2)$ \citep[§6.2]{PetersenMuller2019}, whereas our model corresponds to $Q_x=\mathcal{N}(\mu_0+\beta x,\,\sigma_0^2+\gamma^2 x^2+2\rho x)$ for any $|\rho|\leq\sigma_0\gamma$, which is strictly more flexible.
This additional freedom has important benefits in applications.

An important precedent for our work is the proposals of \citet{KarimiRipaniGeorgiou2021} and \cite{KarimiGeorgiou2021}.
Their focus was on \textquote{utilizing general curves in a Euclidean setting and lifting them to corresponding measure-valued curves in Wasserstein space}; accordingly, they viewed their method as a way to generalize the ``straight-line geometry'' of simple linear regression.
When $p =1$ (i.e., when there is only a single covariate), our model reduces to theirs.
Their theoretical results include deriving a multimarginal formulation (akin to our \cref{prop:linear_explicit}) and matrix formulation in the Gaussian case (akin to our \cref{lemma:marginal_cov_gaussian}).
Our full axiomatic justification, duality theory, and statistical analysis are new.

Our model also extends statistical work on the Wasserstein barycenter problem under the template deformation model~\citep{ahidar-coutrixConvergenceRatesEmpirical2020,LeGouic2022FastCE,zemelFrechetMeansProcrustes2018a}.
Outside of special cases (for example, one-dimensional measures), the best general rates are due to \citet{ahidar-coutrixConvergenceRatesEmpirical2020}, who show that, under our \cref{cond:alphabeta}, barycenters on $\cP_2(\RR^d)$ can be estimated at the rate $n^{-1/d'}$ for any $d' > d$.
A striking improvement of this result is due to \cite{LeGouic2022FastCE}: under \cref{cond:alphabeta} with $\beta - \alpha < 1$, barycenters can be estimated at the parametric rate $1/n$, with no dependence on the ambient dimension.
Our results are in some sense intermediate between these regimes: we show that estimation at the rate $\sqrt{d/n}$ is possible for any positive $\alpha$ and $\beta$.
While we fail to recover the sharp rate obtained by \citet{LeGouic2022FastCE} when $\beta$ and $\alpha$ are very close, our result applies in far greater generality and is an exponential improvement over the general bounds of \citet{ahidar-coutrixConvergenceRatesEmpirical2020}.

\section{Wasserstein least squares as a lifting of Euclidean least squares}\label{sec:lifting}

\subsection{Lifting via convex extension}
The goal of this section is to describe the canonical lifting procedure that allows us to extend the least squares problem to probability measures, following a strategy developed by~\cite{lavenantLiftingFunctionalsDefined2023}.
As described in the introduction, we first view the method of least squares as a purely deterministic optimization problem, an analogue of which we will develop for the Wasserstein space.
We then describe and analyze a statistical model under which our lifted least squares objective serves as a natural estimation procedure.

Consider a generic nonparametric regression problem with fixed nonzero covariates $\{\bx_i\}_{i=1}^n \subseteq \RR^p$.
Given a candidate class $\cF$ of regression functions from $\RR^p$ to $\RR^d$ and responses $\by_1, \dots, \by_n$, the method of least squares amounts to evaluating  $\EEuc:(\RR^d)^{n}\to \RR$ given by
\begin{equation}\label{eq:NPR}
    \EEuc(\by_1, \dots, \by_n)=\min_{f\in\mathcal{F}}\frac{1}{n}\sum_{i=1}^n \|\by_i-f(\bx_i)\|^2_2\,.
\end{equation}
Throughout, we will adopt one of two main assumptions on $\cF$.
\begin{assume}\label{assume:compact}
	$\cF$ is a non-empty, convex, compact subset of $C(\RR^p; \RR^d)$.
\end{assume}
\begin{assume}\label{assume:linear}
	$\cF$ is the set of all linear functions from $\RR^p$ to $\RR^d$.
\end{assume}
Under either \cref{assume:compact} or \cref{assume:linear}, minimizers for the least squares problem exist in $\cF$, which justifies the use of $\min$ in~\eqref{eq:NPR}.
We call $\EEuc$ the \textit{Euclidean least squares functional}, to indicate that it is defined on the Euclidean space $(\RR^d)^{\otimes n}$.

Our goal is to define a new Wasserstein least squares functional $\EWas: \cP_2(\RR^d)^{n} \to \RR$, which is an analogue of~\eqref{eq:NPR} defined on the Wasserstein space.
Inspired by the properties of $\EEuc$, we shall require that $\EWas$ satisfy two requirements:
\begin{enumerate}
		\renewcommand{\theenumi}{R\arabic{enumi}}
	\item $\EWas(\delta_{\by_1}, \dots, \delta_{\by_n}) = \EEuc(\by_1, \dots, \by_n)$.
	That is, in the special case where each of the measures $\nu_i$ is a Dirac at a single point $\by_i$, the lifted functional should reduce to its Euclidean counterpart~\eqref{eq:NPR}.
		\item $\EWas(\nu_1, \dots, \nu_n)$ is a convex, lower-semicontinuous function of $(\nu_1, \dots, \nu_n)$.\footnote{We stress that this convexity is with respect to the \textit{linear} structure of the space $\cP(\RR^d)$ (equipped with the topology of narrow convergence), that is, given $\nu_1, \dots, \nu_n, \nu_1', \dots, \nu_n' \in \cP_2(\RR^d)$, we require $$\EWas(\lambda \nu_1 + (1-\lambda) \nu_1', \dots, \lambda \nu_n  + (1-\lambda) \nu_n') \leq \lambda \EWas(\nu_1, \dots, \nu_n) + (1-\lambda)\EWas (\nu_1', \dots, \nu_n') \quad \forall \lambda \in [0,1]\,.$$
		We discuss geodesic convexity (with respect to the Wasserstein geometry) in \cref{rem:non-convex-G}, below.}
\end{enumerate}

These two requirements are natural in light of our statistical goals.
The first ensures that $\EWas$ genuinely captures the features of the original Euclidean least squares functional: any proposed lifting that failed to satisfy this requirement would not be a plausible method of extending $\EEuc$ to the Wasserstein space.
The second, convexity and continuity, is a fundamental property of $\EEuc$, and is clearly desirable for mathematical analysis of the regression problem.
As we shall see, this requirement also plays a crucial role in the development of convex duality for the Wasserstein least squares functional, see \cref{sec:duality}.

These two requirements alone are not sufficient to uniquely specify $\EWas$, nor to make it useful; for example, on a compact set, the definition $\EWas(\nu_1, \dots, \nu_n) = \EEuc\left(\int x \dd \nu_1, \dots, \int x \dd \nu_n\right)$ which simply performs Euclidean least squares on the \textit{means} of the measures satisfies R1 and R2, but it throws away too much information, since it fails to distinguish between different measures with the same centering.
This example shows that
$\EWas$ can fail to be meaningful if it does not usefully discriminate between different inputs.
Motivated by this observation, we define $\EWas$ to be ``as discriminating as possible'' by setting
\begin{equation}\label{eq:ewas_def}
	\EWas(\nu_1, \dots, \nu_n) = \sup_{\Phi: \text{$\Phi$ satisfies R1 \& R2}} \Phi(\nu_1, \dots, \nu_n)\,.
\end{equation}
Equivalently, $\EWas$ is the \textit{convex extension} or \textit{convex hull} of the Euclidean functional $\EEuc$, which is the canonical way of constructing convex functions taking prescribed values at certain points (\citealp[IV.2.5]{Hiriart-Urruty2011-dv}, see also \citealp{BinBunNil25}).
Crucially, since both R1 and R2 are preserved under pointwise suprema, $\EWas$ defined by~\eqref{eq:ewas_def} itself satisfies our main requirements.
Though this is good evidence that $\EWas$ is a natural object, unfortunately,
the definition in~\eqref{eq:ewas_def} is not explicit.
Our main theorem in this section gives two equivalent characterizations for $\EWas$ in terms of optimal transport problems, which provides further justification for it as our primary object of study.

To state these, we need two pieces of notation. First, given $\nu_1, \dots, \nu_n \in \cP_2(\RR^d)$, write $\Pi(\nu_1, \dots, \nu_n) \in \cP_2((\RR^d)^n)$ for the set of multi-marginal couplings: joint distributions whose $d$-dimensional marginals agree with $\nu_1, \dots, \nu_n$, respectively.
Second, given a Borel probability measure $Q$ on $\cF$ and $\bx \in \RR^p$, denote by $Q_{\bx}$ the probability measure on $\RR^d$ obtained as the pushforward of $Q$ by the map $f \mapsto f(\bx)$.\footnote{We equip $\cF$ with the compact-open topology, for which the pointwise evaluation maps are measurable.}

The following result shows that~\eqref{eq:ewas_def} can be equivalently expressed either as a multi-marginal optimal transport problem over $\Pi(\nu_1, \dots, \nu_n)$, or as an optimization problem over probability measures on $\cF$.
\begin{theorem}\label{thm:lifted_main}
	Adopt either \cref{assume:compact} or \cref{assume:linear}.
	Let $\EWas$ be the convex extension of $\EEuc$ defined via~\eqref{eq:ewas_def}, and let $\nu_1, \dots, \nu_n \in \cP_2(\RR^d)$.
	Then
	\begin{align}
		\EWas(\nu_1, \dots, \nu_n) & = \min_{P \in \Pi(\nu_1, \dots, \nu_n)} \int \EEuc(\by_1, \dots, \by_n) \dd P(\by_1, \dots, \by_n)\label{eq:mult_lift} \\
		& = \min_{Q \in \cP(\cF)} \frac 1n \sum_{i=1}^n W_2^2(\nu_i, Q_{\bx_i}) \label{eq:marg_lift}\,.
	\end{align}
\end{theorem}
Part of the content of \cref{thm:lifted_main} is that the minima in both~\eqref{eq:mult_lift} and~\eqref{eq:marg_lift} are attained.
In the linear case (Assumption~\ref{assume:linear}), $\cP(\cF)$ is exactly the set of laws of $\RR^{p\times d}$-valued random variables, recovering the basic Wasserstein least squares problem from the introduction.

The expression given by~\eqref{eq:marg_lift} is the Wasserstein least squares function promised in the introduction.
Remarkably, despite the abstractness of the definition~\eqref{eq:ewas_def}, the equivalent version in~\eqref{eq:marg_lift} is a direct Wasserstein space analogue of $\EEuc$, with the $\ell^2$ distance replaced by the $W_2$ distance and the optimization set $\cF$ replaced by the set $\cP(\cF)$ of probability measures over $\cF$.
The axiomatic justification of $\EWas$ via~\eqref{eq:ewas_def} and its more interpretable expressions given by \cref{thm:lifted_main} together make the case that~\eqref{eq:marg_lift} is a canonical formulation of least squares regression on the Wasserstein space.

Note that, unlike the Euclidean least squares problem, the solution to~\eqref{eq:marg_lift} is not unique in general, since the objective function involves only the marginal distributions $Q_{\bx_1}, \dots, Q_{\bx_n}$.
However, in \cref{prop:unique} we will give a natural condition under which these marginal distributions are uniquely identified.

\subsection{Convex duality}\label{sec:duality}
In this section, we establish that the functional $\EWas$, which is given by a convex minimization problem, also possesses a dual formulation as a maximization problem.
This formulation is crucial to the development of the algorithms we present in \cref{sec:alg} as well as our statistical results in \cref{sec:estimation}.

Our proof is based on an extension of the strategy pioneered by \cite{aguehBarycentersWassersteinSpace2011} for showing a similar dual formulation for the Wasserstein barycenter problem.

Given $\psi: \cF \to \RR$ and $i \in [n]$, we denote by $S_i \psi: \RR^d \to \RR$ the function
\begin{equation}\label{eq:s_def}
	S_i \psi(\by) = \inf_{f \in \cF} \|\by - f(\bx_i)\|^2 - \psi(f)\,.
\end{equation}
\begin{theorem}\label{thm:basic_dual}
	Adopt either \cref{assume:compact} or \cref{assume:linear}.
	If $\nu_1, \dots, \nu_n \in \cP_2(\RR^d)$, then
	\begin{equation}\label{eq:dual}
		\EWas(\nu_1, \dots, \nu_n) = \sup \left\{\frac 1n \sum_{i=1}^n \int S_i \psi_i \dd \nu_i : \psi_1, \dots, \psi_n \in C(\cF), \sum_{i=1}^n \psi_i = 0 \right\}\,.
	\end{equation}
\end{theorem}
Note that for any $\psi \in C(\cF)$, the function $S_i \psi$ satisfies $S_i \psi(\by) \leq C(1+ \|\by\|^2)$ for some $C > 0$, and since $\nu_i \in \cP_2(\RR^d)$, this ensures that the integral appearing in~\eqref{eq:dual} is well defined in $[-\infty, \infty)$.

    \subsection{Wasserstein linear least squares}\label{section:linear case}
    In this section, we work under \cref{assume:linear},  where the set $\mathcal{F}$ consists of \textit{linear} functions from $\RR^{p}$ to $\RR^d$, in which case the Wasserstein least squares problem furnishes a canonical linear regression methodology for the Wasserstein space.

    We identify the set $\cF$ with the set of real $p \times d$ matrices, with $\bB \in \RR^{p \times d}$ giving rise to the linear function $\bx \mapsto \bB^\top \bx$.

    For convenience, we restate \cref{thm:lifted_main} explicitly for the linear case.
    \begin{theorem}\label{thm:main_linear}
    	The convex extension of the Euclidean least squares functional
    	\[\EEuc(\by_1, \dots, \by_n) = \min_{\bB \in \RR^{p \times d}} \frac 1n \sum_{i=1}^n \|\by_i - \bB^\top \bx_i\|^2\]
    	is given by
    	\begin{align}
    		\EWas(\nu_1, \dots, \nu_n) & =  \min_{P\in \Pi(\nu_1,...,\nu_n) }  \int_{(\mathbb{R}^{d})^n} \min_{\bB\in\RR^{p\times d}}\frac{1}{n}\sum_{i=1}^n \|\by_i-\bB^\top \bx_i\|^2 \dd P(\by_1, \dots, \by_n) \label{eq:lin_mult}\\
    		& = \min_{Q \in \cP(\RR^{p \times d})} \frac 1n \sum_{i=1}^n \min_{\substack{\bY_i \sim \nu_i \\ \bB \sim Q}} \E \|\bY_i - \bB^\top \bx_i\|^2 \label{eq:lin_wls}\,.
    	\end{align}
    \end{theorem}

    Since the inner minimization in~\eqref{eq:lin_mult} is the standard linear regression problem, it can be solved explicitly.
    Collecting the covariates into the design matrix $\bX \in \RR^{n \times p}$ and writing $\bY = (\by_1 | \dots | \by_n)^\top \in \RR^{n \times d}$, we recall that the minimal norm solution for Euclidean least squares is
     \begin{equation}\label{eq:LR-normal_equations}
        \bB^*=(\bX^\top \bX)^+ \bX^\top \bY \in\argmin_{\bB\in\RR^{p\times d}}\frac{1}{n}\sum_{i=1}^n \|\by_i- \bB^\top \bx_i\|^2.
    \end{equation}
    When $(\by_1, \dots, \by_n) \sim P$, this induces a corresponding distribution for $\bY$ and thereby for $\bB^*$ via~\eqref{eq:LR-normal_equations}.
    We therefore have the following alternative expression for~\eqref{eq:lin_mult}.
    \begin{proposition}\label{prop:linear_explicit}
    	In the same setting as \cref{thm:main_linear},
    	\begin{align}
    			\EWas(\nu_1, \dots, \nu_n) & = \min_{P\in \Pi(\nu_1,...,\nu_n) }  \int_{(\mathbb{R}^{d})^n} \frac 1n \sum_{i=1}^n \|\by_i-(\bB^*)^\top \bx_i\|^2 \dd P(\by_1, \dots, \by_n) \\
    			& = \min_{P\in \Pi(\nu_1,...,\nu_n) }  \int_{(\mathbb{R}^{d})^n} \frac 1n \|\mathbf{Y}\|_F^2 - \frac 1n \|((\bX^\top \bX)^+)^{1/2}\bX^\top \bY\|^2_F \dd P
    	\end{align}
    	In particular, since $\int\|\mathbf{Y}\|_F^2 \dd P$ depends only on the marginals $\nu_1, \dots, \nu_n$, solutions to~\eqref{eq:lin_mult} agree with those of
    	\begin{equation}\label{eq:max-linear-mot}
    		       \max_{P\in \Pi(\nu_1,...,\nu_n) }  \int_{(\mathbb{R}^{d})^n}  \|((\bX^\top \bX)^+)^{1/2}\bX^\top \bY\|^2_F\, \dd P\,.
    	\end{equation}
    \end{proposition}

We stress again that, in~\eqref{eq:max-linear-mot}, the coupling $P$ determines the distribution of $\bY$.
Recall that in Euclidean linear least squares, the quantity $ \|((\bX^\top \bX)^+)^{1/2}\bX^\top \bY\|^2_F$ measures the \emph{explained variance}, the variance in the responses attributable to the covariates.
This expression gives a natural geometric interpretation to the Wasserstein linear least squares problem: it seeks a coupling between $\nu_1, \dots, \nu_n$ that maximizes the predictive power of the linear model as measured by explained variance.

	We can also exploit the linear structure to obtain stronger duality results under \cref{assume:linear}.
	\begin{theorem}\label{thm:lin_dual}
		Adopt \cref{assume:linear}.
		If $\nu_1, \dots, \nu_n \in \cP_2(\RR^d)$, then the supremum in \eqref{eq:dual} is attained.
	\end{theorem}

	Examining the optimality conditions of~\eqref{eq:dual}, we obtain a characterization of optimal primal and dual solutions for Wasserstein least squares.
	\begin{theorem}\label{thm:opt_duals}
		Let $\nu_1, \dots, \nu_n \in \cP_2(\RR^d)$, and let $Q$ be an optimal solution to~\eqref{eq:lin_wls}.
		Then there exist convex functions $\varphi_1, \dots, \varphi_n$ such that:
		\begin{enumerate}
			\item For each $i \in [n]$, the pair $(\varphi_i, \varphi_i^*)$ are optimal Brenier potentials for $(Q_{\bx_i}, \nu_i)$; i.e., $\varphi_i$ satisfies
			\begin{equation}
				W_2^2(\nu_i, Q_{\bx_i}) = \int (\|\cdot\|^2 - 2 \varphi_i^*) \dd \nu_i + \int (\|\cdot\|^2 - 2 \varphi_i) \dd Q_{\bx_i}
			\end{equation}
			\item The functions satisfy
			\begin{equation}\label{eq:dual_ineq}
				\frac 1n \sum_{i=1}^n \varphi_i(\bB^\top \bx_i) \leq \frac 1n \sum_{i=1}^n \frac{\|\bB^\top \bx_i\|^2}{2}\,,
			\end{equation}
			with equality $Q$-a.s.
		\end{enumerate}
		Conversely, if $Q$ is such that there exist $\varphi_1, \dots, \varphi_n$ satisfying the above conditions, then $Q$ is an optimal solution to~\eqref{eq:marg_lift}.
	\end{theorem}

	In light of \cref{thm:opt_duals}, we call any collection $(\varphi_1, \dots, \varphi_n)$ satisfying conditions 1 and 2 above \textit{optimal dual solutions} to the Wasserstein least squares problem.

	\Cref{thm:opt_duals} also gives a criterion under which the marginal distributions of an optimal solution are uniquely identified.
	\begin{proposition}\label{prop:unique}
		Suppose that $\nu_1, \dots, \nu_n$ are absolutely continuous with respect to the Lebesgue measure.
		If $Q$ and $Q'$ are any two optimal solutions to~\eqref{eq:marg_lift}, then $Q_{\bx_i} = Q'_{\bx_i}$ for each $i \in [n]$.
	\end{proposition}
	\begin{proof}
		Let $(\psi_1, \dots, \psi_n)$ solve~\eqref{eq:dual}, and let $\varphi_1, \dots, \varphi_n$ be constructed by~\eqref{eq:phi_def}.
		Then \cref{thm:opt_duals} shows that $(\varphi_i, \varphi_i^*)$ are optimal Brenier potentials for both $(Q_{\bx_i}, \nu_i)$ and $(Q'_{\bx_i}, \nu_i)$ for each $i = 1, \dots, n$.
		In particular, since $\nu_i$ is absolutely continuous, this fact implies that both $Q_{\bx_i}$ and $Q'_{\bx_i}$ are equal to $(\nabla \varphi_i^*)_\# \nu_i$, hence they agree.
	\end{proof}

	Finally, \cref{thm:opt_duals} gives a characterization of optimality exactly equivalent to the normal equations in linear regression.
	\begin{theorem}\label{theorem:W-normal-equations}
		Let $\nu_1, \dots, \nu_n \in \cP_2(\RR^d)$ and let $Q$ be an optimal solution to~\eqref{eq:lin_wls}.
		Then there exist convex functions $\varphi_1, \dots, \varphi_n$ such that $(\varphi_i, \varphi_i^*)$ are optimal Brenier potentials for $(Q_{\bx_i}, \nu_i)$ for each $i \in [n]$, and
		\begin{equation}\label{eq:W-normal-equations}
			 \frac{1}{n}\sum^n_{i=1}\bx_i(\nabla\varphi_i(\bB^\top\bx_i))^\top=\Sigma_{XX}\bB \quad\forall \, \bB \in \operatorname{supp}(Q)\,.
		\end{equation}
		In particular, $\varphi_i$ is differentiable on $\operatorname{supp}(Q_{\bx_i})$.
	\end{theorem}

	Since $\varphi_i$ is differentiable on $\operatorname{supp}(Q_{\bx_i})$ and $\varphi_i$ is an optimal Brenier potential, if we write $\nabla\varphi_i(\bB^\top\bx_i)=\bY_i$, then $\bY_i \sim \nu_i$ when $\bB \sim Q$.
	Writing $\nabla\varphi_i(\bB^\top\bx_i)=\bY_i$ and defining $\Sigma_{YX}$ and $\Sigma_{XX}$ as above, we observe that $\bY_i \sim \nu_i$ and that~\eqref{eq:W-normal-equations} is equivalent to
	\begin{equation}\label{eq:new-normal}
		 \Sigma_{YX}=\Sigma_{XX}\bB \quad \text{a.s.}
	\end{equation}
	 However, in contrast to the Euclidean case,~\eqref{eq:new-normal} alone does not specify the law of $\bB$, since the law of $\Sigma_{XY}$ depends on $P$, the coupling that solves~\eqref{eq:lin_mult}.

	The proof of \cref{theorem:W-normal-equations} reveals that the gradient of the Wasserstein least squares objective vanishes precisely when the normal equations hold.
	The following lemma (proved in \cref{appendix:FOG}) makes this explicit and is the key ingredient for the algorithms of \cref{sec:alg}.

	\begin{lemma}\label{lemma:W-gradient-of-G}
    \begin{sloppypar}Let $Q$ be absolutely continuous.
		The $W_2$-gradient of $G(Q)=\frac{1}{n}\sum_{i=1}^n W_2^2(Q_{\bx_i},\nu_i)$ with respect to $Q$ is the map $\WGrad_Q G(Q)\in L^2(\RR^{p\times d},\RR^{p\times d};Q)$ with assignment rule\end{sloppypar}
		\begin{equation}\label{eq:W-gradient-of-G}
			\bB\mapsto -\frac{1}{n}\sum_{i=1}^n\bigl(\bx_i(\nabla\varphi_i(\bB^\top\bx_i))^\top -\bx_i\bx_i^\top\bB\bigr)
		\end{equation}
		for every $\bB\in\supp(Q)$, where $\varphi_i$ are the optimal Brenier potentials for $(Q_{\bx_i},\nu_i)$.
        In particular, $\WGrad_Q G(Q)=0$ $Q$-a.s.\ if and only if the normal equations~\eqref{eq:W-normal-equations} hold.
	\end{lemma}

	 \begin{remark}\label{remark:regression_barycneter}
	 	In the case where all the covariates are equal, \eqref{eq:lin_wls} is a barycenter problem in Wasserstein space \citep{aguehBarycentersWassersteinSpace2011}.
	 	In this case, \eqref{eq:lin_mult} is the well known multimarginal formulation of the barycenter problem, and \cref{thm:lin_dual,thm:opt_duals} recover the duality results for barycenters obtained by Agueh and Carlier~\cite[]{aguehBarycentersWassersteinSpace2011}.
	 	Our statistical results for Wasserstein least squares regression will supply new convergence results for barycenters as corollaries, see \cref{section:barycenter}.
	 \end{remark}

\section{Gradient descent algorithms}\label{sec:alg}
A canonical approach for designing algorithms to solve optimization problems in Wasserstein space is to discretize the gradient flow of the associated objective function along Wasserstein geodesics; see Chapter~6 of \cite{ChewiNilesWeedRigollet2025} for a sampler of applications.
We apply this strategy to the Wasserstein least squares objective introduced in \cref{section:linear case}; the full first-order geometry of $G$ is developed in \cref{appendix:FOG}.
The two algorithms derived below are the computational workhorses of the experiments in \cref{sec:experiments}.

Recall that the $W_2$-gradient flow of $G$, the $\cP_{2,{\text{a.c}}}(\RR^{p\times d})$-variable objective function of~\eqref{eq:lin_wls}, with starting measure $Q$ is the curve of measures $({Q}^t)_{t\geq0}$ that solves the system
\begin{equation}\label{eq:W-Flow-linear}
    \begin{cases}
        \dot{Q}^t= - \WGrad_{{Q}^t}G(Q^t) \\
        Q^0=Q,
    \end{cases}
\end{equation}

The Wasserstein gradient descent scheme for $G$ with starting measure $Q\in\cP_2(\mathbb{R}^{p\times d})$, the corresponding discretization of \eqref{eq:W-Flow-linear}, is
\begin{equation}\label{eq:gradient_descent_G}
\begin{cases}
    Q^{k+1}:=\exp_{Q^k}(-\tau\WGrad_{Q^k}G(Q^k)) \quad \forall k>0 \text{ and } k\in \mathbb{N},  \\
    Q^0:= Q.
\end{cases}
\end{equation}
where $\exp_{Q^k}:T_{Q^k}\cP_2(\RR^{p\times d})\to\cP_2(\RR^{p\times d})$ is the exponential map with assignment rule
\begin{equation}
    \exp_{Q^k}(v):=(\Id+v)_{\sharp}Q^k \quad \forall v\in T_{Q^k}\cP(\RR^{p\times d}).
\end{equation}
\Cref{lemma:W-gradient-of-G} gives
\begin{equation}\label{eq:W-Flow-linear'}
    \WGrad_{\dot{Q}^t}G(Q^t)=-\frac{1}{n}\sum^n_{i=1}(\bx_i(\nabla\varphi^t_i(\bB^\top\bx_i))^\top -\bx_i\bx_i^\top \bB ),
\end{equation}
which means that the main difficulty of solving \eqref{eq:lin_wls} via a gradient descent scheme boils down to having access to $\{\nabla\varphi^t_i\}_{i=1}^n$, the collection of Brenier maps from $Q^t_{\bx_i}$ to $\nu_i$ at each $t\geq0$.

We study two cases where the optimal transport between $Q_{\bx_i}$ and $\nu_i$ is explicit: the Gaussian (\cref{alg:bw_gd}) and the 1-d responses (\cref{alg:wls-gd-particle}) cases.

\begin{remark}[Non-geodesic convexity of $G$]\label{rem:non-convex-G}
The non-geodesic convexity of $G$\footnote{A function $f$ on a metric space $(X,d)$ is \emph{geodesically convex} if $f(\gamma(t))\le(1-t)f(\gamma(0))+tf(\gamma(1))$ for every constant-speed geodesic $\gamma:[0,1]\to X$ and $t\in[0,1]$  \cite[see][Section~5.2]{ChewiNilesWeedRigollet2025}.} is the central difficulty for gradient descent: without convexity, standard descent arguments do not guarantee convergence to a global minimum.

Indeed, Wasserstein barycenters are a special case of Wasserstein least squares ---obtained by taking a constant design $\bx_i\equiv c$, so that every response is fitted against the same marginal $Q_c$ (see \cref{section:barycenter})---and \citet[Section~B.2]{chewiGradientDescentAlgorithms2020c} construct explicit counterexamples showing that the Wasserstein barycenter functional fails to be geodesically convex.

\begin{sloppypar}
This non-convexity is precisely the motivation for \citet{chewiGradientDescentAlgorithms2020c} and \citet{altschulerAveragingBuresWassersteinManifold2021} to replace convexity by a Polyak--\L{}ojasiewicz condition~\citep{karimi2016linear} in order to recover convergence rates for barycenter gradient descent. We leave a parallel analysis for the full Wasserstein least squares functional to future work.
\end{sloppypar}

At a deeper level, the non-convexity stems from the fact that $(\cP_2(\RR^{p\times d}),W_2)$ is a CBB(0) space---a metric space of non-negative Alexandrov curvature \citep{ambrosioGradientFlowsMetric2008,sturm2006geometry1,alexander2023alexandrovgeometryfoundations}---in which squared-distance functions are geodesically concave, so each summand of $G$ is generally geodesically concave rather than convex.
\end{remark}

\subsection{Bures-Wasserstein gradient descent for the linear case}\label{sec:BW_GD}
In \cref{appendix:geometry_of_gaussians}, we show that when each of the responses is Gaussian, then the solution to the Wasserstein least squares problem can be taken to be Gaussian as well. The reduction rests on parametrizing Gaussian measures $Q\in\cP_2(\RR^{p\times d})$ by their vectorized mean $m_Q\in\RR^{dp}$ and covariance $\Sigma_Q\in\mathbf{S}^{dp}_{++}$, where the $\vect(\cdot)$ operator identifies coefficient matrices $\bB^\top$ with vectors in $\RR^{dp}$. In these coordinates, \cref{appendix:geometry_of_gaussians} derives the closed-form marginal covariances $\Sigma_{Q_{\bx_i}}$, Brenier maps between $Q_{\bx_i}$ and $\nu_i$, and the Bures-Wasserstein gradient of $G$ (all collected in \cref{lemma:marginal_cov_gaussian}), along with first-order optimality conditions for the Gaussian problem.
Following the view adopted by \cite{altschulerAveragingBuresWassersteinManifold2021} and \cite{10.5555/3600270.3601319}, we may therefore exploit this fact to study the Gaussian Wasserstein least squares problem as an optimization over the Bures-Wasserstein manifold.

With these tools in hand, we specialize the Wasserstein gradient flow \eqref{eq:W-Flow-linear} to Gaussian responses $\nu_i\in\cP_2(\RR^d)$ with mean zero and covariance $\Sigma_{\nu_i}$. The flow takes the form
\begin{equation}\label{eq:W-flow-gaussian-mean-zero}
    \begin{cases}
        \dot{Q^t}=\frac{1}{n}\sum^n_{i=1}\bx_i\bx_i^\top(\cdot)(\Sigma_{Q^t_{i}}\#\Sigma_{\nu_i}- I_d)\\
        Q^0= \mathcal{N}_{dp}(0,\Sigma_{Q^0}).
    \end{cases}
\end{equation}
We show in \cref{appendix:geometry_of_gaussians} that both \eqref{eq:W-flow-gaussian-mean-zero} and its discrete counterpart \eqref{eq:gradient_descent_G} propagate within the class of Gaussian measures, and admit explicit expressions in terms of the covariance $(\Sigma_{Q^t})_{t\geq0}$ alone.

\begin{proposition}\label{prop:full_gaussian_gf}
The following two properties hold for the Wasserstein least squares functional with mean-zero Gaussian data:
\begin{enumerate}
    \item The Bures-Wasserstein  gradient flow that solves \eqref{eq:W-flow-gaussian-mean-zero} is the curve of Gaussian measures $\left(Q^t=\mathcal{N}_{dp}(0,\Sigma_{Q^t})\right)_{t\geq0}$, where $(\Sigma_t)_{t\geq 0}$ is a solution of the Lyapunov equation
    \begin{equation}
        \dot{\Sigma}_{Q^t}=M_t\Sigma_{Q^t}+\Sigma_{Q^t}M_t
\end{equation}
and $M_t:=(\frac{1}{n}\sum^n_{i=1}\bx_i\bx_i^\top \otimes(\Sigma_{Q^t_{\bx_i}}\#\Sigma_{\nu_i}- I_d))$.
    \item A sequence $(Q^k)_{k\in\mathbb{N}}$ produced by \eqref{eq:gradient_descent_G} at each iteration, has covariances $(\Sigma_{Q^k})_{k\in\mathbb{N}}$ given by each iteration of Algorithm \ref{alg:bw_gd}.
\end{enumerate}

\end{proposition}

\begin{algorithm}[htbp]
\caption{Bures--Wasserstein Gradient Descent}
\label{alg:bw_gd}
\begin{algorithmic}[1]
\renewcommand{\algorithmicrequire}{\textbf{Input:}}
\renewcommand{\algorithmicensure}{\textbf{Output:}}

\Require Step size $\tau > 0$; design points $\{\bx_i\}_{i=1}^n$; initial covariance $\Sigma_{Q^0}$; target covariances $\{\Sigma_{\nu_i}\}_{i=1}^n$.
\Ensure Final covariance matrix $\Sigma_{Q^K}$.

\State \textbf{Initialize:} $\Sigma_{Q^0}$

\For{$k = 0, 1, \ldots, K-1$}
    \State $\displaystyle M_k \leftarrow \tau \left( \sum_{i=1}^n \frac{1}{n} \bx_i \bx_i^\top \otimes (\Sigma_{Q^k,i} \# \Sigma_{\nu_i} - I_d) \right) + I_{dp}$
    \State $\Sigma_{Q^{k+1}} \leftarrow M_k \Sigma_{Q^k} M_k$
\EndFor

\State \Return $\Sigma_{Q^K}$
\end{algorithmic}
\end{algorithm}

\subsection{Wasserstein least squares when $d=1$}
When $d=1$ we can find a solution to the Wasserstein least squares problem via a gradient descent (GD) algorithm since at each step we have a concrete expression for the Brenier maps involved.
Specifically,
\begin{equation}
    \nabla\varphi_{Q_{\bx_i}\to\nu_i}:=f_{i}^{-1}\circ g_{\bx_i}
\end{equation}
with $f_i$ and $g_{\bx_i}$ being respectively the cumulative distribution functions of $\nu_i$ and $Q_{\bx_i}$.

We begin the operationalization of the gradient descent scheme by looking at the particle interpretation of \eqref{eq:W-Flow-linear}, meaning, any $\bbeta\sim Q^0\in \cP_2(\RR^p)$ induces a solution $(\bbeta^t)_{t\geq 0}$ of the following ODE
\begin{equation}\label{eq:particle-W-Flow}
\begin{cases}
    \dot\bbeta^t&=\frac{1}{n}\sum^n_{i=1}(f_{i}^{-1}\circ g^t_{\bx_i}(\bx_i^\top \bbeta^t) -\bx_i^\top \bbeta^t )\bx_i\\
    \bbeta^0&=\bbeta
 \end{cases}
\end{equation}
where $\bbeta_t\sim Q_t$ for every $t\geq0$.
This is a McKean--Vlasov ODE (which depends on the distribution of the solution via the cumulative distribution function $g^t_{\bx_i}$).

To implement \eqref{eq:particle-W-Flow}, we employ the corresponding mean-field interacting particle system.
Let $\{\bbeta_j^0\}_{j=1}^M$ be a collection of $M$ i.i.d. random variables coming from $Q^0$. Set $\widehat{Q}^0:=\sum_{l=1}^M\delta_{\bbeta^0_l}$ and $\widehat g^0_{\bx_i}$ to be the cdf of $(\cdot^\top\bx_i)_{\sharp}\widehat{Q}^0$. Then we define $(\widehat{Q}^t)_{t\geq0}$,
the particle gradient flow estimate of $(Q^t)_{t\geq 0}$, to be the curve of empirical measures formed with the solutions $\{(\bbeta^t_j)_{t\geq0}\}_{j=1}^M$ of the  $M$-particle system
\begin{equation}\label{eq:particle-gf}
\begin{cases}
    \dot\bbeta^t_j&=\frac{1}{n}\sum^n_{i=1}(f_{i}^{-1}\circ \widehat g^t_{\bx_i}(\bx_i^\top \bbeta^t_j) -\bx_i^\top \bbeta^t_j )\bx_i,\\
    \widehat Q^t&=\frac{1}{M}\sum_{l=1}^M\delta_{\bbeta^t_l}, \\
    \bbeta^0_j&=\bbeta_j,
 \end{cases}
\end{equation}
where $\widehat{g}^t_{\bx_i}$ is the cdf of $\widehat{Q}^t_{\bx_i}$.

\begin{algorithm}[H]
\caption{Particle gradient descent for Wasserstein least squares ($d=1$)}
\label{alg:wls-gd-particle}
\begin{algorithmic}[1]
\renewcommand{\algorithmicrequire}{\textbf{Input:}}
\renewcommand{\algorithmicensure}{\textbf{Output:}}

\Require Design matrix $\mathbf{X} = (\mathbf{x}_1, \ldots, \mathbf{x}_n)^\top \in \mathbb{R}^{n \times p}$, target distributions $\{\nu_i\}_{i=1}^n$ with CDFs $\{f_i\}_{i=1}^n$
\Require Number of particles $M$, step size $\tau > 0$, number of iterations $T$
\Ensure Particles $\{\boldsymbol{\beta}_j\}^M_{j=1} \subset \mathbb{R}^p$ representing the approximate minimizer $Q^T$

\State \textbf{Initialize:} $\{\boldsymbol{\beta}_j^0\}_{j=1}^M \subset \mathbb{R}^p$ \Comment{Particle representation of $Q^0$}

\For{$k = 0, 1, \ldots, T-1$}
    \For{$i = 1, \ldots, n$}
        \State $\bz_{i,j}^k \leftarrow \mathbf{x}_i^\top \boldsymbol{\beta}_j^k$ for $j = 1, \ldots, M$ \Comment{Pushforward samples}
        \State $\widehat g^k_{\mathbf{x}_i} \leftarrow \text{empirical CDF of } \{\bz_{i,j}^k\}_{j=1}^M$
        \State $\nabla\varphi_{Q^k_{\mathbf{x}_i} \to \nu_i} \leftarrow f_i^{-1} \circ \widehat g^k_{\mathbf{x}_i}$ \Comment{Brenier map}
        \State $T_{i,j}^k \leftarrow \nabla\varphi_{Q^k_{\mathbf{x}_i} \to \nu_i}(\bz_{i,j}^k)$ for $j = 1, \ldots, M$ \Comment{Transported samples}
    \EndFor
    \For{$j = 1, \ldots, M$}
        \State $\boldsymbol{\beta}_j^{k+1} \leftarrow \boldsymbol{\beta}_j^k + \frac{\tau}{n} \sum_{i=1}^n (T_{i,j}^k - \bz_{i,j}^k) \mathbf{x}_i$ \Comment{Gradient step}
    \EndFor
\EndFor

\State \Return $\{\boldsymbol{\beta}_j^T\}_{j=1}^M$
\end{algorithmic}
\end{algorithm}

\begin{remark}[Convergence]
Neither Algorithm~\ref{alg:bw_gd} nor Algorithm~\ref{alg:wls-gd-particle} comes with a formal convergence guarantee at this stage.
Nevertheless, \cref{theorem:smoothness,prop:convergence_rate} in \cref{appendix:FOG} establish that $G$ is $\eta$-smooth along Wasserstein geodesics and that the iterates~\eqref{eq:gradient_descent_G} reach an $\varepsilon$-critical point in $O(\varepsilon^{-2})$ steps.
Though we are not able to rule out the possibility of spurious stationary points, our empirical experiments (see \cref{sec:experiments}) suggest that these algorithms work well in practice.
\end{remark}

\begin{remark}[Polynomial time algorithms in fixed dimension]
	When $\nu_1, \dots, \nu_n$ are finitely supported on at most $m$ atoms, the derivation of the multimarginal formulation in the proof of \cref{thm:lifted_main} shows that the optimal $Q$ is also finitely supported, on at most $m^n$ atoms, and can be found by linear programming; of course, this algorithm runs in polynomial time only if $n = O(1)$.
	However, in fixed dimension (i.e., if both $p$ and $d$ are constant), it is possible to exploit the geometric structure of the Wasserstein least squares problem to obtain algorithms with $\mathrm{poly}(n, m)$ running time (with exponent depending on $p$ and $d$) by using the techniques of~\cite{altschuler2021wasserstein_jmlr}.
	Though the resulting algorithm runs in polynomial time, it is far less practical than the gradient flow algorithms we employ.
\end{remark}

\section{Estimation via Wasserstein least squares}\label{sec:estimation}

In this section, we propose a statistical model for the Wasserstein linear regression problem.
As we shall see, the lifted objective defined in \Cref{sec:lifting} will furnish a natural estimator for this setting.

We retain the same notational convention as above: given a probability measure $Q \in \cP(\RR^{p \times d})$ and $\bx \in \RR^p$, we write $Q_{\bx} \in \cP(\RR^d)$ for the push-forward of $Q$ under the map $\bB \mapsto \bB^\top \bx$.

Let $\{(\bx_i, \nu_i)\}_{i = 1}^n$ be a collection of $n$ covariate/response pairs in $\RR^p \times \cP_2(\RR^{d})$.
We consider the following \textit{template deformation model}~\citep{Boissard2015Distribution,zemelFrechetMeansProcrustes2018a,LeGouic2022FastCE,Panaretos_regression} for Wasserstein linear regression:
\begin{equation}\label{eq:model}\tag{WLR}
	\nu_i = (\nabla \phi_{i})_\# Q^\star_{\bx_i} \quad \quad i = 1, \dots, n\,,
\end{equation}
where $Q^\star \in \cP(\RR^{p \times d})$ is a fixed, unknown distribution, and where $\{\phi_i\}_{i=1}^n$ are i.i.d.\ random functions\footnote{We view $\phi_i$ as a random element of the space of $C^2$ convex functions, equipped with the $C^2_{\mathrm{loc}}$ topology, so all the objects discussed below are measurable.} on $\RR^d$ satisfying the following conditions:
\begin{enumerate}
	\renewcommand{\theenumi}{C\arabic{enumi}}
	\item The functions $\phi_i$ are almost surely convex, twice differentiable, and satisfy $\phi_i(0) = 0$ and $\E[\|\nabla \phi_i(0)\|^2]\leq d\sigma^2$. \label[condition]{cond:convex}
		\item For each $\by \in \RR^d$, the function $\phi_i$ satisfies $\E[\nabla \phi_i(\by)] = \by$. \label[condition]{cond:identity}
	\item There exist positive constants $\alpha, \beta$ such that $\alpha \mathrm{I} \preceq \nabla^2 \phi_i \preceq \beta \mathrm{I}$ almost surely. \label[condition]{cond:alphabeta}
\end{enumerate}
The random maps $\nabla \phi_i$ represent the source of noise in the model.
Indeed, under our conditions, the random maps $\nabla \phi_i$ may be viewed as \textit{mean-zero perturbations in the Wasserstein space}: \cref{cond:convex} implies that $\nabla \phi_i$ is an \textit{optimal} transport map, and
\cref{cond:identity} says that $\E[\nabla \phi_i]$ is the identity function on $\RR^d$.
The responses $\nu_i$ in \eqref{eq:model} are therefore equal to $Q^\star_{\bx_i}$ up to corruptions which, on average, leave the measure unchanged.
More rigorously, this interpretation can be justified by the fact that the above conditions guarantee that $Q^\star_{\bx_i}$ is the Fréchet mean of the random measure $(\nabla \phi_{i})_\# Q^\star_{\bx_i}$:
\begin{equation}\label{eq:perturbed_is_barycenter}
	Q^\star_{\bx_i} = \argmin_{P \in \cP(\RR^d)} \E[W_2^2((\nabla \phi_{i})_\# Q^\star_{\bx_i}, P)]\,.
\end{equation}
We refer to the works cited above for a discussion of the statistical interpretation of~\eqref{eq:perturbed_is_barycenter} as an average on the Wasserstein space.

At the level of random variables, \eqref{eq:model} can be written as
\begin{equation}\label{eq:model'}\tag{WLR${}^\prime$}
	\nu_i = \law(\bY_i \mid \phi_i)\,, \quad \quad \text{where } \bY_i  = \nabla \phi_i(( \bB^\star)^\top \bx_i)\,.
\end{equation}
In this formulation, $\bB^\star \sim Q^\star$ is a random matrix independent of $\phi_i$.
Written in this way, we see that~\eqref{eq:model} models the situation in which the response is the \textit{law} of a linear projection of a random variable $\bB^\star$, corrupted by (possibly nonlinear) noise.
In particular, taking $\phi_i(\by) = \|\by\|^2/2 + \langle \by, \eps_i \rangle$ for i.i.d.\ mean-zero random vectors $\eps_i$, we see that~\eqref{eq:model} includes the random-effects model~\eqref{eq:random_effects}.

Since it is impractical to assume that the statistician can observe a probability distribution $\nu_i$ directly, we complete our model by assuming that the statistician has access to \textit{estimators} $\widehat \nu_i$ of $\nu_i$.
In the simplest case, we assume that these estimators correspond to simply observing $m$ i.i.d.\ samples from $\nu_i$:
\begin{equation}\label{eq:nu_hat_empirical}
	\widehat \nu_i = \sum_{j=1}^m \delta_{Y_{i, j}}\,, \quad \quad  Y_{i,j} \sim \nu_i \text{ i.i.d.}
\end{equation}
We discuss alternative options, where $\widehat \nu_i$ is potentially a nonparametric estimator of $\nu_i$, in Remark~\ref{remark:other_estimators}, below.

The observed data is the $n$ pairs $\{(\bx_i, \widehat \nu_i)\}_{i=1}^n$.
Under this model, we will study the performance of the Wasserstein least squares estimator, as defined in \Cref{sec:lifting}:
\begin{equation}\label{eq:empirical_estimate}
    \widehat Q \in \argmin_{Q \in \cP_2(\RR^{p \times d})} \frac 1n \sum_{i=1}^n W_2^2(\widehat \nu_i, Q_{\bx_i})\,.
\end{equation}
Note that this estimator is similar to the definition of an empirical Wasserstein barycenter~\citep{LeGouic2022FastCE}.
This connection is not superficial, and we shall see in \cref{section:barycenter} that our techniques give new convergence bounds in this setting as well.

In order to assess the performance of $\widehat{Q}$ as an estimator of $Q^\star$, we will evaluate its \textit{in-sample error}~\cite[Chapter 7]{hastie2009elements}:
\begin{equation}\label{eq:W-in-sample-error}
    \E[\sum^n_{i=1} \frac{1}{n}W^2_2(\widehat{Q}_{\bx_i},Q^\star_{\bx_i})]\,.
\end{equation}
We will also assume a standard incoherence condition on the fixed design~\citep{Hoaglin-1978,candes-2009}:
\begin{equation}\label{eq:incoherence}
	\bx_i^\top(\bX^\top\bX)^+\bx_i \leq \mu \frac p n \quad \forall i \in [n]
\end{equation}
for some $\mu \in \RR$.
Our main statistical theorem is:

\begin{theorem}\label{thm:stat_main}
  Assume \Cref{cond:convex,cond:identity,cond:alphabeta}, and that the
  covariates satisfy the incoherence condition~\eqref{eq:incoherence}.
  Suppose further that the latent responses are uniformly bounded:
  \begin{equation}\label{eq:latent_bounded}
  	\|\bB^\top \bx_i\| \leq M \quad \forall i \in [n], \text{ $Q^\star$-almost surely.}
  \end{equation}
  and that $\widehat\nu_i$ is an empirical measure
  from $m$ i.i.d.\ samples from $\nu_i$.
  Let $R = M\max(1, \beta\sqrt{\mu p})$.
  Then the in-sample error of the
  estimator~\eqref{eq:empirical_estimate} satisfies
  \begin{equation}
    \E\left[\frac{1}{n}\sum_{i=1}^n W_2^2(\widehat Q_{\bx_i},
    Q^\star_{\bx_i})\right] \lesssim R^2 \frac \beta \alpha \sqrt{\frac{pd}{n}}
    + \frac{R^2}{\alpha}\, r_m + \sigma^2\frac{pd}{n}\,,
  \end{equation}
  where
  \begin{equation}
  	r_m := \begin{cases}
  		m^{-2/d} & \text{if $d > 4$} \\
  		m^{-1/2} \log m & \text{if $d = 4$} \\
  		m^{-1/2} & \text{if $d < 4$.}
  	\end{cases}
  \end{equation}
\end{theorem}

\begin{remark}
	It is easy to see that the rate of estimation guaranteed by~\cref{thm:stat_main} is minimax optimal up to a logarithmic factor in its dependence on $m$.
	Indeed,
	the rate $r_m$ is optimal for the problem of estimating a compactly supported probability measure on $\RR^d$ in Wasserstein distance from $m$ i.i.d.\ samples, up to the $\log m$ factor when $d = 4$~\citep{Singh-Poczos}, and our model includes that one as a special case.
	The term $\sigma^2 \tfrac{pd}{n}$ is the familiar minimax rate for linear regression (see, e.g., \cite{rigollet_high_dimensional_2023}), which is also a special case of our model.
	This term is therefore also impossible to improve.
	However, we do not know whether the $n^{-1/2}$ dependence in the first term is optimal.
\end{remark}

\begin{remark}\label{remark:other_estimators}
	In many applications, the error in \cref{thm:stat_main} is dominated by the slow $r_m$ term, which arises due to the slow convergence of $\widehat \nu_i$ to $\nu_i$ in Wasserstein distance, which are an unavoidable feature of empirical estimators for Wasserstein distances in general~\citep[see, e.g.,][]{FournierGuillin2015,WeedBach2019,Manole-NW-smooth}.
	This rate can be ameliorated by imposing smoothness assumptions on $\nu_i$ and replacing the empirical estimator~\eqref{eq:nu_hat_empirical} by a suitable nonparametric estimator.

	For example, \cite{Niles-Weed2022-qm} and \cite{Divol2021} show that if $\nu_i$ is a probability measure on $[0,1]^d$ or the $d$-dimensional flat torus with $s$-smooth density for any $d \geq 3$, then a suitable wavelet or kernel density estimator $\widehat \nu_i$ satisfies
	\begin{equation}
		\E W_2(\widehat \nu_i, \nu_i) \lesssim m^{-(s+1)/(2s + d)}\,.
	\end{equation}
	In this setting, for any compact set $\Omega$, we obtain by the triangle inequality
	\begin{align*}
		\E \sup_{\mu \in \cP(\Omega)} |W_2^2(\mu, \widehat \nu_i) - W_2^2(\mu, \nu_i)| & \leq \E \sup_{\mu \in \cP(\Omega)} (W_2(\mu, \widehat \nu_i) + W_2(\mu, \nu_i)) W_2(\widehat \nu_i, \nu_i) \\
		& \lesssim m^{-(s+1)/(2s + d)}\,,
	\end{align*}
	and so the term $r_m$ in \cref{thm:stat_main} may be replaced by $m^{-(s+1)/(2s + d)}$.
	This is an improvement for any $d \geq 5$ and $s \geq \frac{2d -1}{d -4}$.

\end{remark}

\subsection{Application to Wasserstein barycenters}\label{section:barycenter}
The results of the previous section can be reinterpreted to get a new bound for the estimation of Wasserstein barycenters in  $\cP_2(\mathbb{R}^d)$ under the template deformation model.
In what follows, set $S:=\cP_2(\RR^d)$, $\mathfrak{Q}$ a probability measure in $\cP_2(S)$, and
\begin{equation}\label{eq:w^2_barycenter}
    \bb^\star\in\argmin_{\bb\in S }\E_{\nu\sim\mathfrak{Q}}[W^2_2(\bb,\nu)].
\end{equation}
Recall that  in a \textit{template deformation model for the barycenter} $\bb^\star$  one observes independent copies of measures $\nu$ on the support of $\mathfrak{Q}$ and assumes they are a push-forward measure of $\bb^\star$ through a noisy random distortion $T\in \cA:= L^2(\mathbb{R}^d,\RR^d; \bb^\star)$ \citep{LeGouic2022FastCE,zemelFrechetMeansProcrustes2018a,Boissard2015Distribution}.
In other words,
\begin{equation}\label{eq:tdfm}
    \nu=T_{\sharp}\bb^\star \text{,}\quad T\sim P\in\cP_2(\cA)
\end{equation}
where $P\in \cP_2(\cA)$.
We shall note that equations \eqref{eq:w^2_barycenter} and \eqref{eq:tdfm} coincide with equations \eqref{eq:perturbed_is_barycenter} and \eqref{eq:model} respectively when $P$ is a probability measure supported on the functions satisfying \cref{cond:alphabeta,cond:convex,cond:identity} and if we set $p=1$, $\bx_i=1$ for every $\nu_i\in\supp(\mathfrak{Q})$.
Then we can understand the problem of estimating $W_2$-barycenters through the lens of Theorem \ref{thm:stat_main}.
\begin{corollary}\label{coro:bary_rate}
   Let $\mathfrak{Q}\in\cP_2(S)$ have a barycenter $\bb^\star$ with $\operatorname{supp}(\bb^\star) \subset B_M$.
   Assume we observe $n$ measures $\nu_i$ which come from the template deformation model \eqref{eq:tdfm} where each $T=\nabla\phi_i$ satisfies \cref{cond:alphabeta,cond:convex,cond:identity}. Set
   \[\widehat{\bb}:=\argmin_{\bb} \frac{1}{n}\sum_{i=1}^n W^2_2(\bb,\nu_i).\]
   Then
\begin{equation}\label{eq:bary-rate}
    \E\left[ W_2^2(\widehat{\bb}, \bb^\star)\right] \lesssim  M^2 \frac{\beta^3}{\alpha}\sqrt{\frac{d }{n}} + \frac{\sigma^2 d}{n}.
\end{equation}
\end{corollary}
\begin{remark}[Subtleties of the template deformation model, and a comparison to the lower bound of~\citet{hundrieser2024lower}]
	In our model, we require that $\varphi$ satisfy \Cref{cond:identity}, which reads
	\begin{equation}\label{eq:everywhere}
		\E[\nabla \varphi(\by)] = \by \quad \text{ for all $\by \in \RR^d$.}
	\end{equation}
	This agrees with the definition proposed by~\cite{Boissard2015Distribution} and~\cite{zemelFrechetMeansProcrustes2018a}.
	However, it is stronger than the assumption made by~\citet{LeGouic2022FastCE}, who require only that
	\begin{equation}\label{eq:almost_everywhere}
		\E[\nabla \varphi(\by)] = \by \quad\text{ for  $\bb^\star$-almost every $\by \in \RR^d$.}
	\end{equation}
	In the language of \citet{zemelFrechetMeansProcrustes2018a}, this condition is equivalent to requiring that $\bb^\star$ is a \textit{Karcher mean}.
	The arguments of \citet[Theorem 2]{zemelFrechetMeansProcrustes2018a} show that, under \Cref{cond:alphabeta}, these two assumptions agree if $\bb^\star$ is absolutely continuous, with density positive on a convex subset of $\RR^d$.

	However, a construction of \citet{hundrieser2024lower} shows that there is a significant difference between~\eqref{eq:everywhere} and~\eqref{eq:almost_everywhere} in general.
	Section 3.5 of that work constructs an example showing that it is possible to construct $\varphi$ satisfying \Cref{cond:alphabeta} with $\beta/\alpha = O(1)$ and~\eqref{eq:almost_everywhere}---but not~\eqref{eq:everywhere}---for which \textit{no} estimate like~\eqref{eq:bary-rate} can hold.
	In their example, the barycenter is supported on two Diracs.
\end{remark}
This rate should be understood in comparison with the rate derived in Corollary 4.4 of \cite{LeGouic2022FastCE}. Our rate is neither a generalization nor a special case of theirs. The rate in \cite{LeGouic2022FastCE} holds for measures supported in Hilbert spaces, while ours only works for measures supported in finite-dimensional vector spaces. Even then, when both rates hold, theirs converges considerably faster in terms of sample size.
Perhaps more importantly, in terms of \cite{zemelFrechetMeansProcrustes2018a}, the results in \cite{LeGouic2022FastCE} are stronger since they only require $\bb^\star$ to be a \textit{Karcher mean}. Meanwhile, our result is valid only when $\bb^\star$ is a Fréchet mean.
However, they require that \cref{cond:alphabeta} hold with $\beta - \alpha <1$, which is significantly more stringent than our assumption.
For the case of general $\alpha, \beta > 0$ covered by \cref{coro:bary_rate}, the best known rates, due to \citet{ahidar-coutrixConvergenceRatesEmpirical2020}, are exponentially slower.
\section{Experiments}\label{sec:experiments}

We begin with a large-scale real-data application (\cref{appendix:bmi} contains full pre-processing details and supplementary figures), then turn to two synthetic settings that demonstrate the structural virtues of the Wasserstein least squares framework; complete experimental specifications are in \cref{appendix:synthetic}.

\subsection{Modeling BMI with Wasserstein least squares}
\begin{sloppypar}
We examine data coming from the RAND Health and Retirement Study (HRS) Longitudinal File 2022~\citep{RANDHRS2022V1}, a nationally representative longitudinal survey comprising $16$ biennial waves (1992--2022) with approximately $45{,}000$ U.S.\ respondents aged 50 and older. The HRS is a longitudinal panel: in each wave of the study, individuals are asked about personal, demographic, and health information, yielding a data footprint which can be used to ask questions related to aging both at an \textit{individual} and \textit{population} level.
\end{sloppypar}

One such question is obesity, commonly measured through body mass index (BMI), a widely used epidemiological metric with established associations to metabolic disease, longevity, and healthcare burden~\citep{gutin2018bmi}.
In this experiment we aim to characterize how a person's BMI changes across decades by using distributional data. Capturing this calls for an approach that is at once a distributional regression and a model of individual random variation, in the spirit of the linear mixed effects framework~\citep{Laird1982}. Based on \eqref{eq:model} and \eqref{eq:model'}, Wasserstein least squares is both.

Prior work has fitted linear mixed effects models in which age, birth cohort, and gender are covariates and individual heterogeneity is encoded as a random effect on the regression coefficients~\citep{life-course}. We adopt the same model structure but fit entirely from distributional data.
That is, we do \emph{not} use longitudinal per-individual data when computing our estimator. This per-individual data collected as part of the survey can then be used to validate our findings. For the purposes of our analysis, each observation is therefore the BMI distribution of a demographic \emph{cell}---a group of respondents sharing the same birth cohort ($5$-year bins: 1930--34 through 1955--59), survey wave, and gender---yielding $n = 164$ cells after quality filtering. Details can be found in Appendix \ref{appendix:bmi}.

Under the template deformation model~\eqref{eq:model} paradigm, this translates into describing the cell distributions $\nu_i$ through the quadratic design vectors
\begin{equation}\label{eq:bmi_covariate}
\bx_i = \bigl(1, \widetilde{a}_i, \widetilde{a}_i^2, \widetilde{c}_i, \widetilde{g}_i\bigr)^\top,
\end{equation}
where $\widetilde{a}_i$, $\widetilde{c}_i$, and $\widetilde{g}_i \in \{-1,+1\}$ are normalized age, birth cohort, and gender ($+1$ = female), so that
\begin{equation}\label{eq:bmi_wlr}
\nu_i = \mathrm{Law}\bigl(\nabla\phi_i\bigl((\bB^\star)^\top \bx_i\bigr)\bigr), \qquad \bB^\star \sim Q^\star \in \cP_2(\RR^p).
\end{equation}
The quadratic age term captures the empirical rise in BMI until approximately age 70, followed by a plateau or decline~\citep{life-course}. The symmetric gender coding ensures $\beta_0$ represents the predicted distribution at the mean age and cohort, averaged across genders.

We represent $\widehat{Q}$, the Wassertein least squares estimator for this data, as a cloud of $M = 20{,}000$ particles fit via Algorithm~\ref{alg:wls-gd-particle}, and compare against Global Fréchet regression~\citep{PetersenMuller2019} on the same cells. The central inferential goal is the probability of crossing BMI $= 30$~\citep{cdc_bmi_2024} and how that probability updates as successive measurements accumulate, a question we assess directly against the individual HRS trajectories.

\subsubsection*{Results and analysis}
Both methods achieve similar in-sample fit: Wasserstein $R^2 \approx 0.90$ for the quadratic Wasserstein least squares and Fréchet models alike, and leave-one-out cross-validation errors differing by a mean $W_2$ gap of just $0.025$ BMI units.
This is expected since the two methods produce identical marginal predictions at the observed cells in expectation. Their architectures diverge in how they represent the \emph{joint} distribution of coefficients (\Cref{fig:corner_main}) and, consequently, in their capacity for conditional inference (\Cref{fig:double_cond_main}).

Both models also agree at the level of mean effects on the primary demographic trends: a baseline BMI of approximately $28.5$ kg/m$^2$; a positive cohort effect ($\widehat\beta_{\rm cohort} \approx 0.88$, consistent with the global rise of obesity \citep{WHO2000Obesity}); and a concave age trajectory ($\widehat\beta_{{\rm age}^2} \approx -0.37$), rising to a peak in early retirement before declining.
These headline figures are consistent across the two methods precisely because they reflect the dominant variation in the data and are well-captured by any reasonable regression model.
The differences emerge when we look beyond these mean effects.

The Wasserstein least squares estimator returns a full joint distribution $\widehat{Q}$ over the coefficient vector $$\bbeta = (\beta_0, \beta_{\rm age}, \beta_{{\rm age}^2}, \beta_{\rm cohort}, \beta_{\rm gender})^\top.$$
\Cref{fig:corner_main} (left) shows its corner plot: diagonal panels display the marginal distribution of each coefficient, and off-diagonal panels reveal genuine correlations among demographic effects.
On the other hand, as we describe in \cref{appendix:bmi}, Fr\'echet regression also implicitly yields a joint distribution over the coefficient vector: the predicted distribution corresponding to a covariate vector $\bx$ has quantile function $p \mapsto \bx^\top \widehat \beta(p)$ for $p \in (0, 1)$, where $\widehat \beta(p)$ is a least-squares estimate constructed from the quantile functions of the responses.
At the level of random variables, this also gives rise to a joint distribution over the coefficient vector, namely $\mathrm{Law}(\widehat \beta_0(U), \widehat \beta_{\rm age}(U), \widehat \beta_{{\rm age}^2}(U), \widehat \beta_{\rm cohort}(U), \widehat \beta_{\rm gender}(U))$, for $U \sim \mathrm{Unif}([0,1])$.
\cref{fig:corner_main} (right) shows the corresponding corner plot.

\begin{figure}[h]
    \centering
    \begin{subfigure}[t]{0.49\textwidth}
        \centering
        \includegraphics[width=\textwidth]{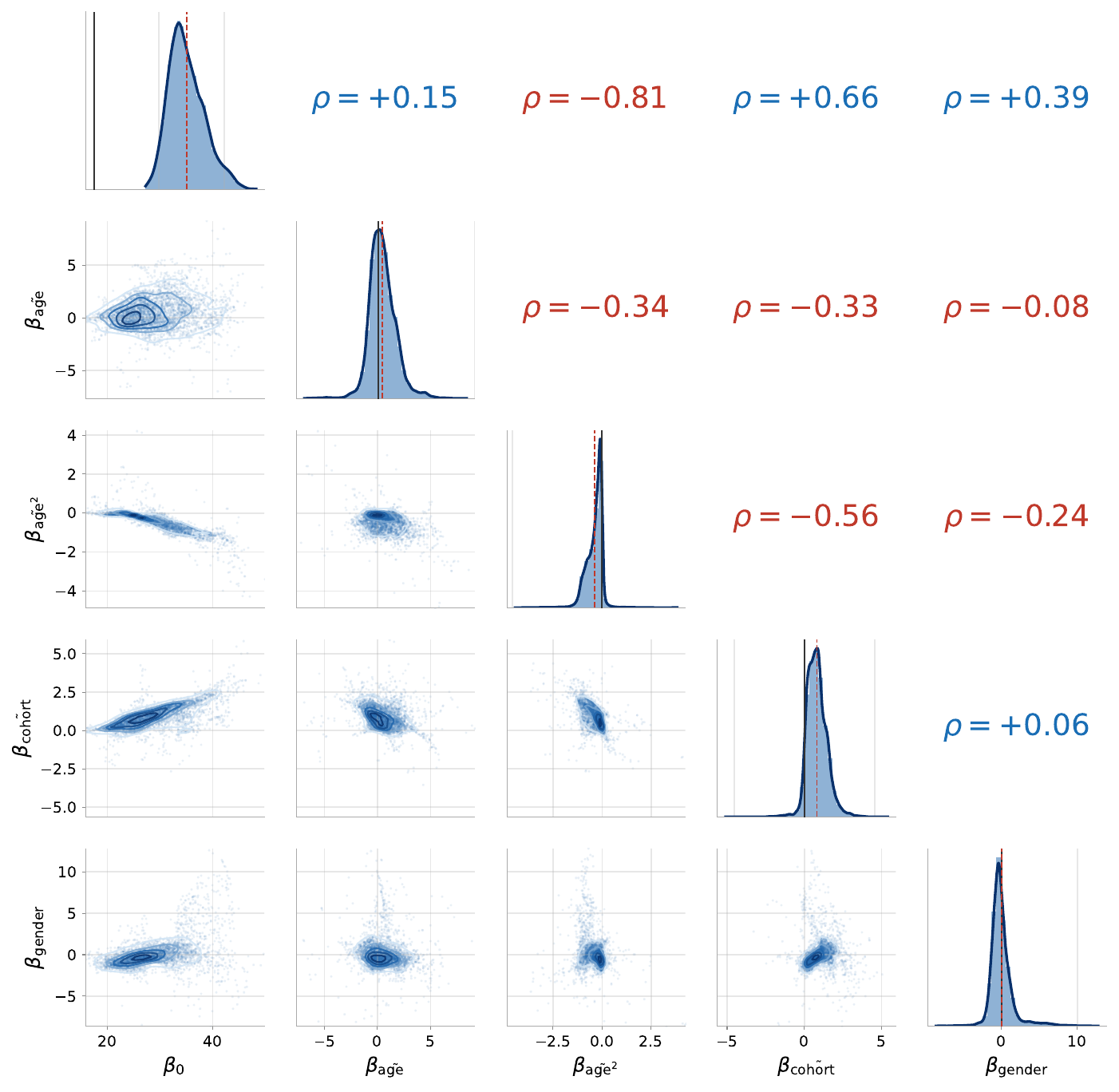}
        \caption{Wasserstein least squares: genuine joint distribution $\widehat{Q}$.}
        \label{fig:corner_main_wls}
    \end{subfigure}
    \hfill
    \begin{subfigure}[t]{0.49\textwidth}
        \centering
        \includegraphics[width=\textwidth]{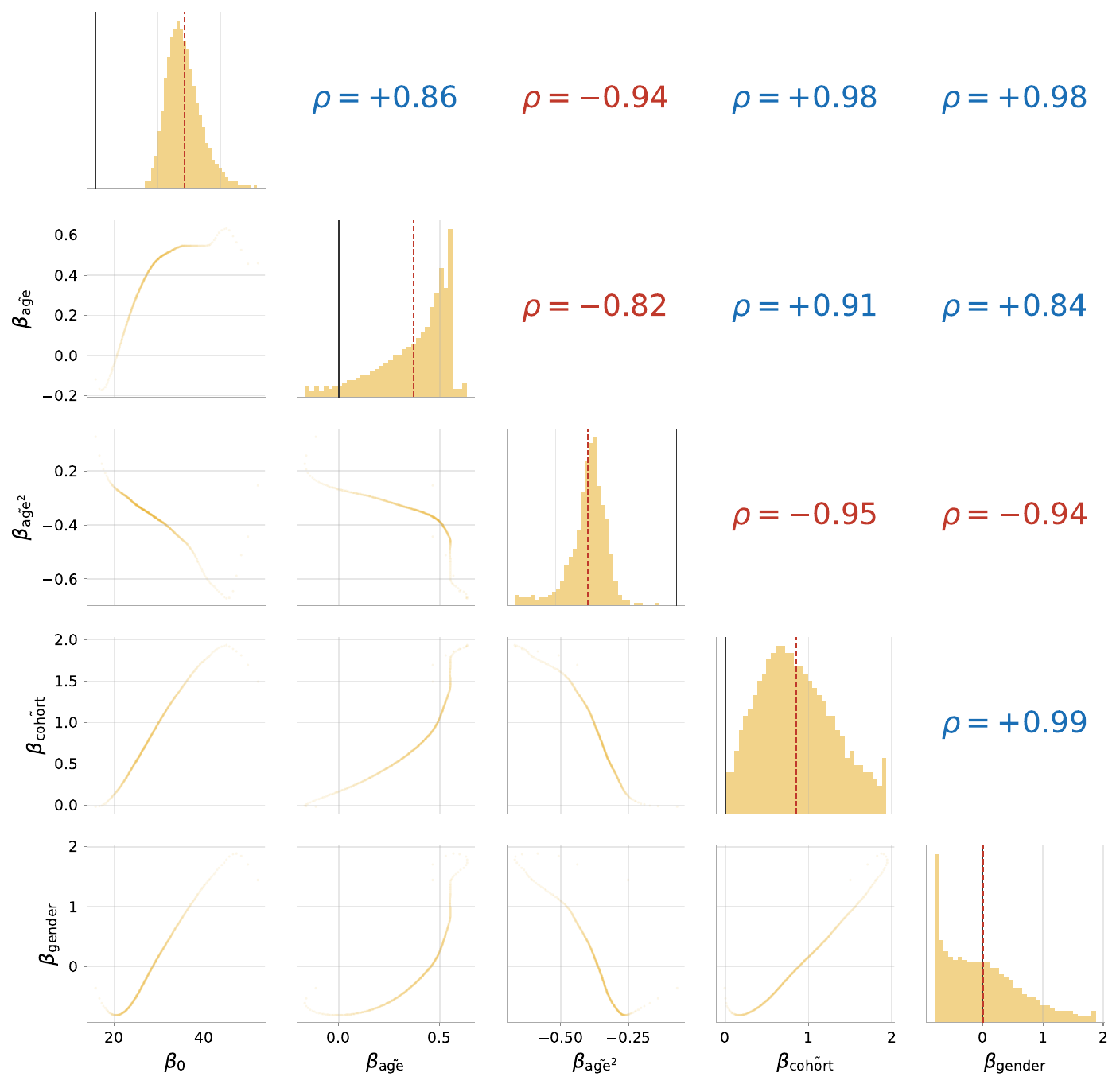}
        \caption{Fréchet: one-dimensional deterministic curve.}
        \label{fig:corner_main_frechet}
    \end{subfigure}
    \caption{Corner plots of estimated coefficient distributions on HRS BMI data (quadratic model, $n = 164$ cells).
    Diagonal: marginal distributions; lower triangle: pairwise scatter plots; upper triangle: Pearson correlations.
    \textbf{Wasserstein least squares} (left) recovers a genuine multivariate distribution with $p = 5$ independent degrees of freedom.
    \textbf{Fréchet} (right) collapses to a curve in $\RR^5$ parametrized by the single scalar $\pp$, forcing all pairs into near-perfect correlation and suppressing all subpopulation heterogeneity.}
    \label{fig:corner_main}
\end{figure}

Several features stand out in the Wasserstein least squares analysis.
Approximately $10\%$ of the particle mass satisfies $\beta_{{\rm age}^2} > 0$, identifying a subpopulation whose BMI trajectory is convex --- accelerating rather than plateauing with age.
The raw HRS trajectories confirm this: roughly $10\%$ of individuals in the 1940--44 female cohort exhibit a genuine BMI minimum followed by a late-life upturn (\cref{appendix:bmi}).
In the estimator obtained from Fréchet regression, $\widehat\beta_{{\rm age}^2}(p)$ is negative for all $p$, implying that \emph{zero} individuals have convex trajectories.
Similarly, the gender coefficient in Wasserstein least squares carries a standard deviation of $1.69$ --- more than twice the Fréchet estimate of $0.64$ --- reflecting that the male--female BMI gap varies substantially across subpopulations, a feature that independent quantile mapping compresses away.

\begin{sloppypar}
More broadly, since the joint distribution obtained in Fréchet regression is parametrized by this single scalar, the coefficients are perfectly correlated and the support of the joint distribution is a one-dimensional curve in $\RR^p$.
Concretely, knowing the intercept $\beta_0$ from the Fréchet fit determines
every other coefficient; the conditional variance of any slope given $\beta_0$ is essentially zero.
This is a structural consequence of pointwise quantile regression since Fréchet regression has exactly one degree of freedom in its coefficient distribution.
Wasserstein least squares, by contrast, retains all $p = 5$ degrees of freedom, enabling the model to capture localized demographic heterogeneity that is invisible to Fréchet regression.
\end{sloppypar}
\begin{figure}[h]
    \centering
    \includegraphics[width=.9\textwidth]{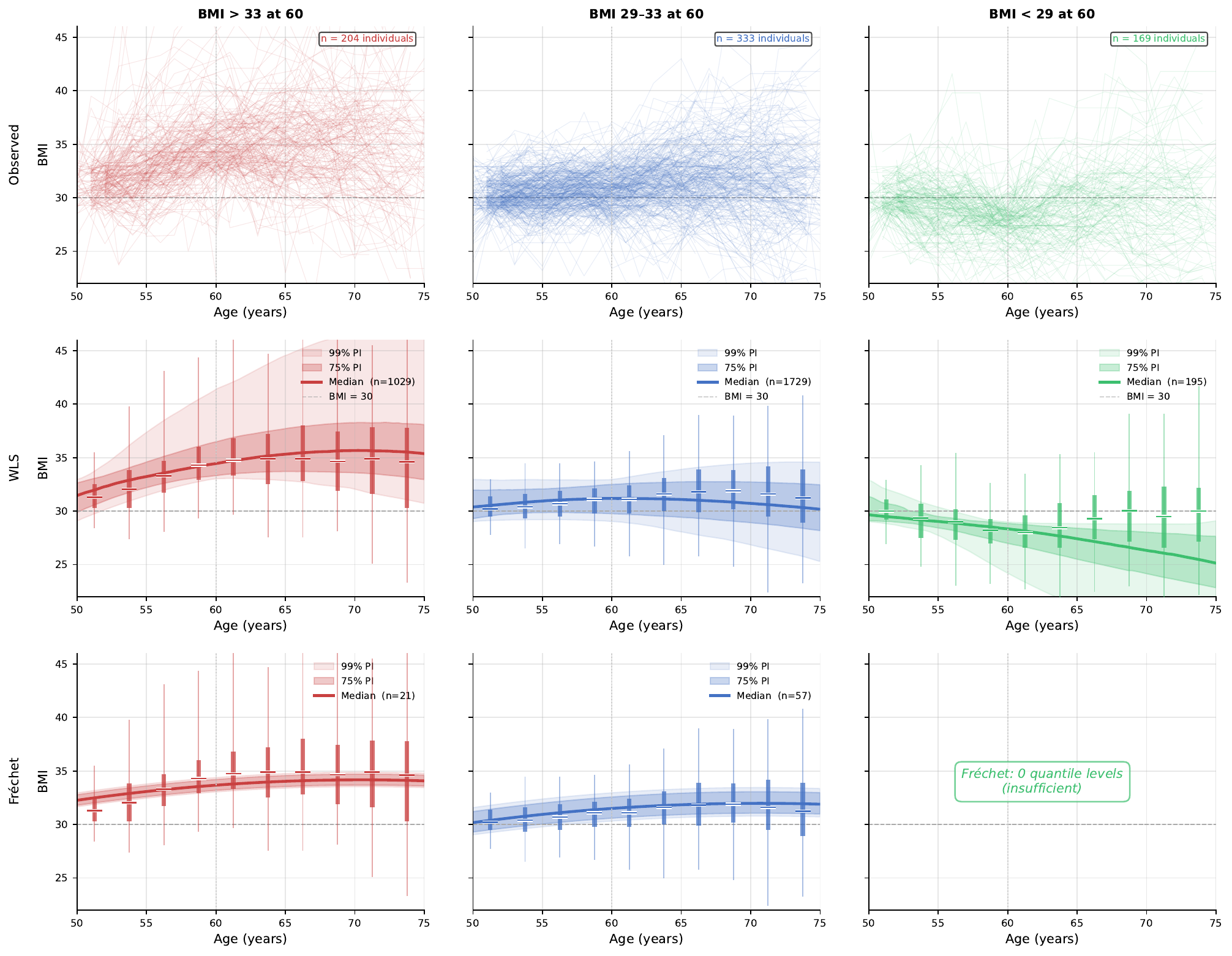}
    \caption{\textbf{Sequential conditioning (BMI $\in [29,33]$ at age~50; three age-60 scenarios), cohorts 1935--39 and 1940--44, both sexes.}
    Columns: worsening (BMI $> 33$, red), stable (BMI $29$--$33$, blue), improved (BMI $< 29$, green).
    Rows: observed trajectories (top), Wasserstein least squares (middle), Fréchet (bottom).
    Shaded bands: 75\% and 99\% prediction intervals; box plots: matched empirical distribution.
    Wasserstein least squares differentiates all three scenarios with well-separated, meaningful prediction bands.
    Fréchet produces results only for worsening and stable;
    the improved scenario is structurally impossible under the fitted model.}
    \label{fig:double_cond_main}
\end{figure}

The consequences of this structural difference become most visible when attempting to make conditional predictions.
For Wasserstein least squares, conditioning on observed BMI $\approx 31$ at age 50 is simple particle selection: retain
\begin{equation*}
    \cM_{50} = \bigl\{m : |\bx(50)^\top\bbeta_m - 31| \leq 2\bigr\}
\end{equation*}
and propagate those particles forward with age.\footnote{
  Here $m$ indexes the $M = 20{,}000$ particles $\{\bbeta_m\}$, each a draw from the fitted
  distribution $\widehat{Q}$ over the coefficient vector
  $\bbeta$.
  The covariate vector $\bx(a)$ encodes age $a$ together with the cohort and gender of the cell
  of interest (see~\eqref{eq:bmi_covariate}), so $\bx(50)^\top\bbeta_m$ is the BMI predicted by
  particle $m$ at age~50.
  The set $\cM_{50}$ therefore collects all particles whose age-50 prediction falls within
  $\pm 2$ BMI units of the target value of~31; the $\leq 2$ tolerance window is chosen to
  retain a sufficient number of particles for reliable downstream inference.}
Even at this single-conditioning stage, the Wasserstein least squares intervals are wide ($\approx 8$--$12$ BMI units at the 80\% level), correctly reflecting the genuine biological uncertainty: individuals at the same current BMI may follow vastly different long-term trajectories.
By contrast, the perfect correlations implicit in the Fréchet regression estimator yield
near-degenerate intervals.
Empirically, Wasserstein least squares achieves $87.7\%$ coverage at the nominal $99\%$ prediction interval; Fréchet achieves only $45.6\%$ (\cref{appendix:bmi}).

Introducing a second observation at age 60 further filters the particle set, separating three outcome groups: continued weight gain ($\text{BMI}_{60} > 33$), stable obesity ($\text{BMI}_{60} \in [29,33]$), and improvement ($\text{BMI}_{60} < 29$).
\Cref{fig:double_cond_main} shows the result.
Wasserstein least squares differentiates all three scenarios: the prediction bands separate clearly, widen appropriately as the retained sample shrinks ($n = 1{,}029$, $1{,}729$, and $195$ particles respectively), and the improved trajectory descends below the obesity threshold with a median that tracks the empirical data.

Fréchet regression again produces near degenerate intervals; in the third scenario, it produces nothing, because every quantile-level trajectory consistent with $\text{BMI}_{50} \approx 31$ also predicts $\text{BMI}_{60} \geq 29$.
A non-negligible fraction of HRS respondents lowers their BMI between ages 50 and 60, yet the Fréchet model cannot represent this outcome at all.
This failure is a direct consequence of the one-dimensional structure of the Fréchet coefficient distribution.
Wasserstein least squares, by retaining the full joint distribution of $Q^\star$, remains capable of expressing the full heterogeneity of the data.

\FloatBarrier
\subsection{Synthetic experiments}

We validate Wasserstein least squares on two controlled settings that illustrate the model's capacity to recover linear distributional structure from noisy observations.
In both settings, each observed distribution $\nu_i = (\nabla\phi_i)_\# Q^\star_{\bx_i}$ arises from the template deformation model~\eqref{eq:model}: a random convex map $\nabla\phi_i$ satisfying C1--C3 pushes the true template $Q^\star_{\bx_i}$ forward, adding noise whose form is unknown to the estimator.
Figure~\ref{fig:transport_main} illustrates this process for five noise families; the deviation $T(y)-y$ (bottom row) confirms that each map stays close to the identity on average, satisfying condition C2.

\begin{figure}[ht]
    \centering
    \includegraphics[width=\textwidth]{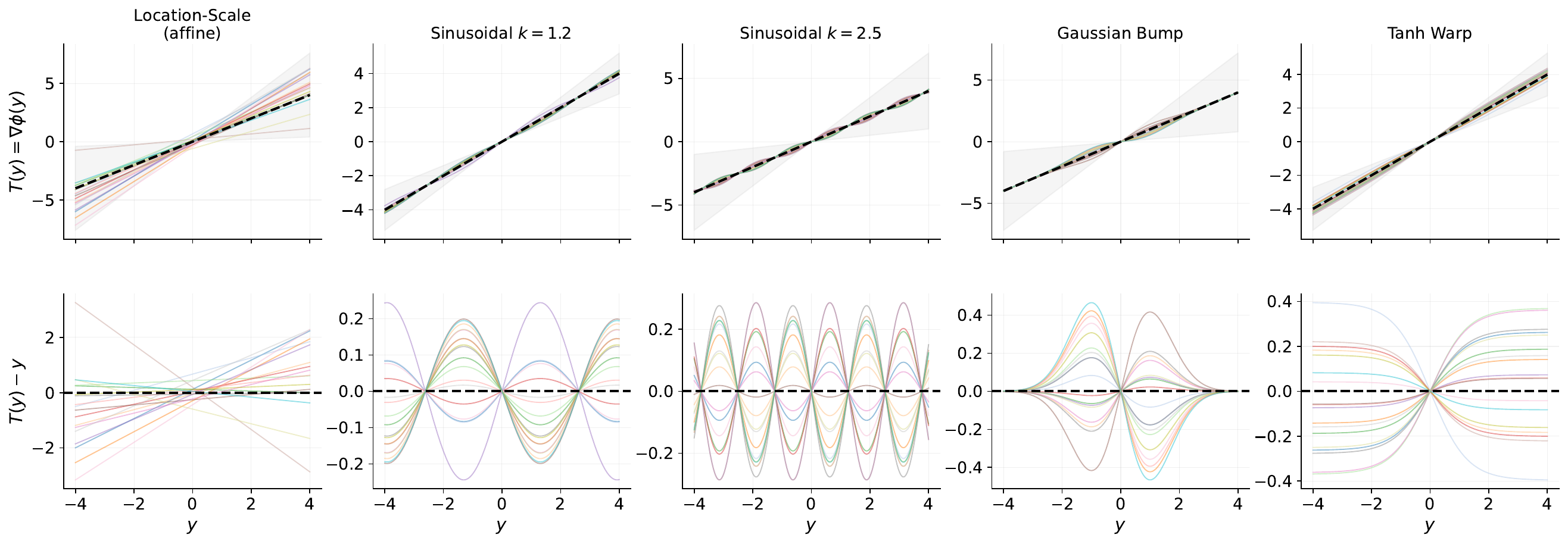}
    \caption{Random transport maps $\nabla\phi_i$ from five noise families (coloured) applied to the template $Q^\star_{\bx}$ (black dashed).
    Top row: the maps $T(y)$; shaded band shows the C3 curvature bounds.
    Bottom row: deviation $T(y)-y$ from the identity, illustrating C2 ($\mathbb{E}[T(y)]=y$).}
    \label{fig:transport_main}
\end{figure}

The two templates are: \textit{(i) Univariate} ($d=1$, $n=50$), with $Q^\star_{\bx}=\mathcal{N}(t,\,1+t^2)$ and covariate $\bx=(1,t)^\top$, $t\sim\mathcal{U}[-2,2]$; the variance is U-shaped, growing quadratically in $|t|$.
\textit{(ii) Bivariate Gaussian} ($d=2$, $n=50$), with $Q^\star_{\bx}=\mathcal{N}(\mu(t),\Sigma(t))$ where $\Sigma(t)=A+t(B+B^\top)+t^2C$ traces a quadratic arc on the cone of symmetric positive definite matrices.
Full specifications and numerical summaries are in \cref{appendix:synthetic}.
\begin{figure}[ht]
\includegraphics[width=.9\textwidth]{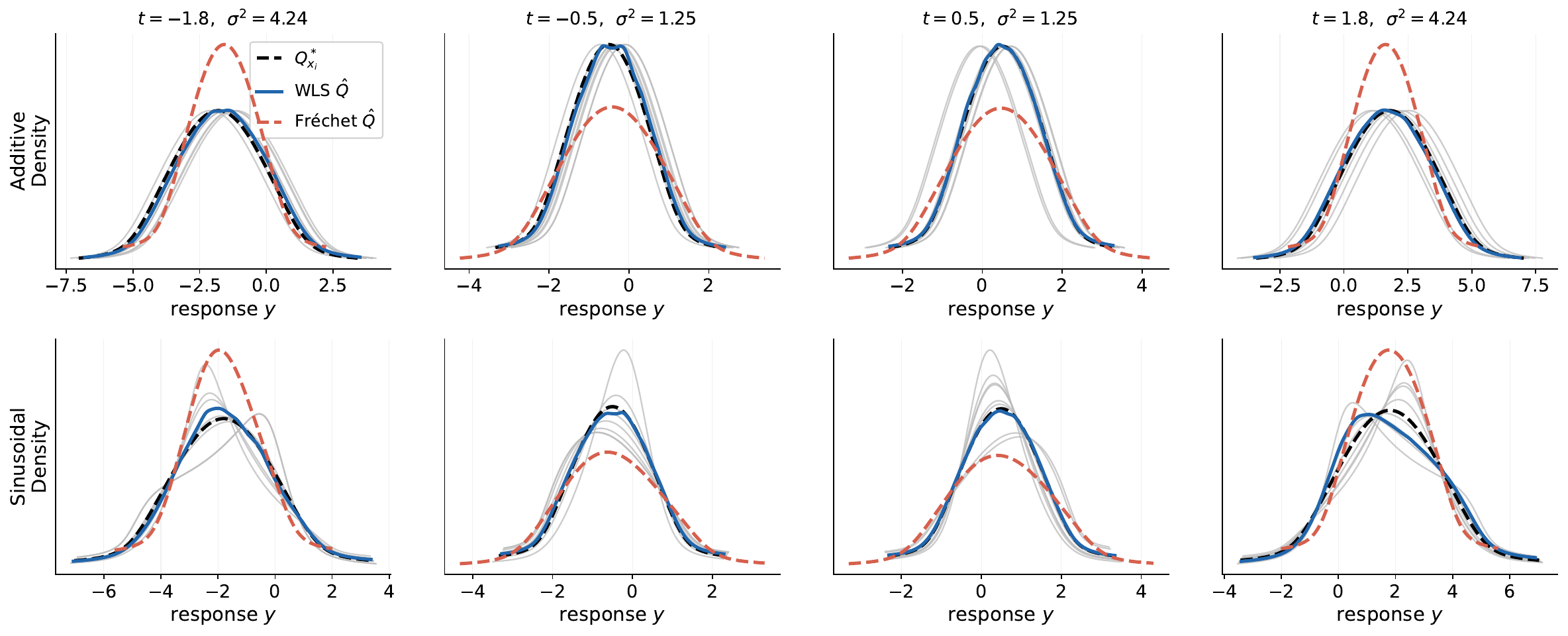}
    \caption{Univariate: estimated densities at $t\in\{-1.8,-0.5,0.5,1.8\}$. True template $Q^\star_{\bx_i}$ (black dashed), noisy observations $\nu_i$ (grey), Wasserstein least squares (blue), Fréchet (red dashed). Wasserstein least squares recovers the true template across the covariate range under additive noise. The U-shaped variance (univariate), Fréchet regression underestimates spread at the extremes.}
    \label{fig:fitting_main}
\end{figure}

Both templates share a feature that illuminates the modeling scope of Wasserstein least squares.
In the univariate Gaussian case, Fréchet regression~\citep[§6.2]{PetersenMuller2019} corresponds to the model $Q_x\sim\mathcal{N}(\mu_0+\beta x,\,(\sigma_0+\gamma x)^2)$.
The Wasserstein least squares model corresponds to $Q_x\sim\mathcal{N}(\mu_0+\beta x,\,\sigma_0^2+\gamma^2 x^2+2\rho x)$ for any $|\rho|\leq\sigma_0\gamma$, which is strictly more flexible: completing the square gives $\sigma_0^2+\gamma^2 x^2+2\rho x = (a+\gamma x)^2+b$ where $a=\rho/\gamma$ and $b=\sigma_0^2-\rho^2/\gamma^2\geq 0$, yielding one extra degree of freedom over Fréchet.

At the level of random variables, this extra freedom is $\mathrm{Cov}(B_0,B_1)=\rho$: the Fréchet parametrization is precisely the boundary case $\rho=\sigma_0\gamma$, which forces $B_0$ and $B_1$ to be perfectly dependent ($\mathrm{Corr}(B_0,B_1)=1$).
Our template $\sigma^2(t)=1+t^2$ corresponds to $\rho=0$, independent random coefficients, which is natural in many applications but cannot be expressed in the Fréchet family.

In the bivariate case, the same principle operates at the level of covariance matrices: Fréchet predictions are confined to the convex hull of the training responses and cannot reach the quadratic arc traced by $\Sigma(t)$.
We compare against two Fréchet baselines: \emph{Fréchet-GD}, which applies the iterative descent algorithm of \citet{bures-frechet}, and \emph{Fréchet-OLS}, a closed-form ordinary least squares fit on the covariance entries (details in \cref{appendix:synthetic}).

Figure~\ref{fig:fitting_main} shows the fitted distributions at four covariate values; Wasserstein least squares tracks the true template across the full range of $t$.
Figure~\ref{fig:necessity_main} focuses on the bivariate covariance trajectory: the true path bends, and Wasserstein least squares follows it faithfully in log-Euclidean coordinates.

\begin{figure}[h]
\includegraphics[width=\textwidth]{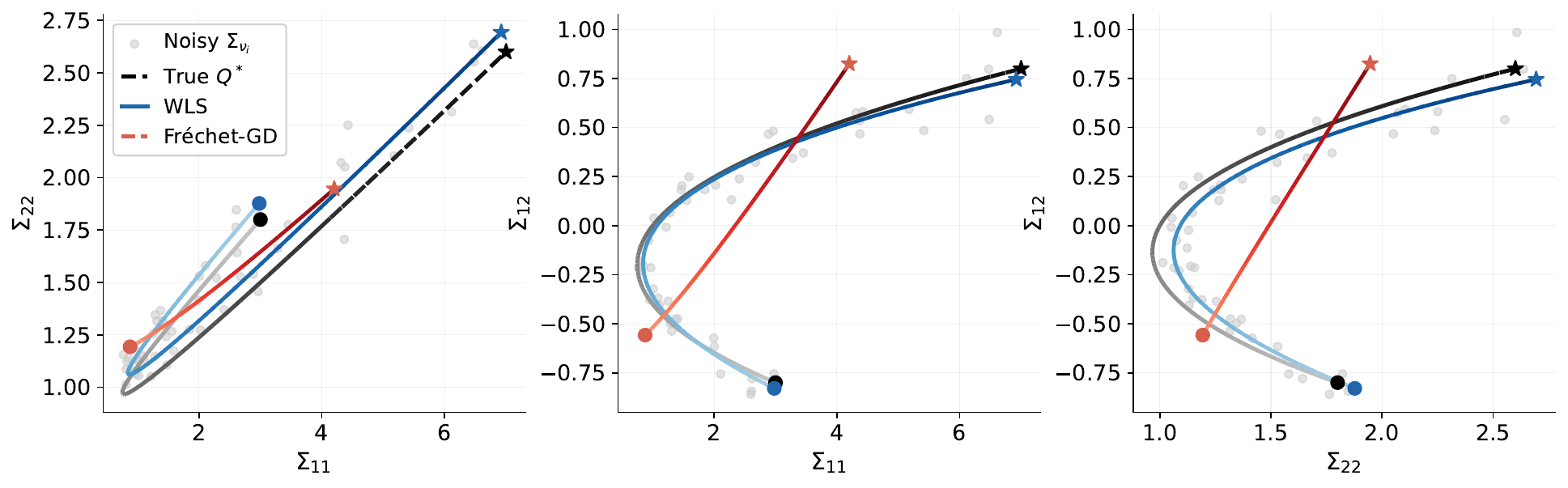}
        \caption{Covariance trajectory $t\mapsto\Sigma(t)$ projected onto three coordinate pairs of the SPD cone. Wasserstein least squares (blue) traces the curved true path (black); Fréchet-GD (red) follows a straight line.Bivariate Gaussian experiment: Wasserstein least squares recovers the quadratic covariance trajectory $\Sigma(t)=A+t(B+B^\top)+t^2C$. The log-Euclidean view (right) makes the curvature explicit; Fréchet regression's linear weighting prevents it from reaching beyond the convex hull of the training responses.}
        \label{fig:log_euclidean_main}
    \label{fig:necessity_main}
\end{figure}

\FloatBarrier

\section*{Acknowledgments}
The authors would like to thank Aram-Alexandre Pooladian, Jacob Shkrob, and Shay Sadovsky for helpful discussions. The second author was supported in part by NSF grant DMS-2339829.

\bibliographystyle{plainnat}
\bibliography{biblio}
\appendix
\section{Experiment: retirement data (full analysis)}\label{appendix:bmi}

Body Mass Index (BMI) is a widely used, simple, and cost-effective metric for population health, offering insights into metabolic disease associated with obesity, longevity, and a variety of healthcare burdens \citep{gutin2018bmi}.
To demonstrate the practical and theoretical advantages of the Wasserstein least squares framework in the analysis of distributional responses, we examine the BMI data from the RAND Health and Retirement Study (HRS) Longitudinal File 2022 (Version 1) \cite{RANDHRS2022V1}. We compare our approach to Global Fréchet regression \cite{PetersenMuller2019} and illustrate how our framework translates standard linear modeling methods—including coefficient interpretation and analysis of variance—directly into the space of probability distributions.

\subsubsection*{Modelling and pre-processing}
The HRS is a nationally representative longitudinal survey of U.S. citizens aged 50 and older and their spouses, sponsored by the National Institute on Aging and conducted biennially by the University of Michigan since 1992. The study comprises $16$ biennial waves spanning $1992$ to $2022$, with approximately $45{,}000$ unique respondents. For each respondent and wave, the database provides self-reported height and weight from which BMI is computed as weight (kg) divided by height squared (m$^2$) \cite{cdc_bmi_2024}.

 We adopt a \textquote{macroscopic} distributional approach where each observation is the BMI distribution of a specific demographic \textquote{cell}. A cell is defined as the set of participants in the study who share three demographic variables: birth cohort, survey wave, and gender. Respondents are assigned to six $5$-year birth cohorts based on birth year: $1930$--$1934$, $1935$--$1939$, $1940$--$1944$, $1945$--$1949$, $1950$--$1954$, and $1955$--$1959$.

Individual BMI observations are filtered to the range $[15, 55]$ kg/m$^2$ to exclude implausible values. Cells with fewer than $200$ valid observations are excluded to ensure reliable density estimation. For each retained cell $(c, w, g)$ corresponding to cohort $c$, wave $w$, and gender $g$, we compute a representative age as $a_{c,w} = y_w - (c + 2.5)$, where $y_w$ denotes the calendar year of wave $w$ and $c + 2.5$ is the cohort midpoint. The final sample comprises $n=164$ cells with ages ranging from approximately $35$ to $90$ years.

From the cell-level BMI observations, we construct continuous density estimates as follows. For each cell $i$, the raw BMI values $\{b_{i,1}, \ldots, b_{i,n_i}\}$ are smoothed using kernel density estimation with a Gaussian kernel and bandwidth selected via Scott's rule, yielding continuous density estimates $\nu_i$ evaluated on a grid of $200$ equally-spaced points spanning $[15, 55]$ kg/m$^2$. Each density is normalized so that it integrates to 1.

To capture the structural evolution of these BMI profiles, we model the cell-level distributions $\nu_i$ using a Wasserstein linear regression framework \eqref{eq:model'}, meaning
\begin{equation*}
    \nu_i = \text{law}(\mathbf{Y}_i)\,, \quad \quad \text{where } \mathbf{Y}_i  = \nabla \phi_i(( \mathbf{B}^\star)^\top \mathbf{x}_i)\, \quad \text{ and}\quad \bB^\star\sim Q^\star\in\cP_2({\RR^p}).
\end{equation*}
This formulation mimics the dynamics of a classical linear mixed-effects model when interpreted at the particle level. For any individual particle comprising the distribution of a demographic cell, the linear predictor $( \mathbf{B}^\star)^\top \mathbf{x}_i$ acts as a shared, group-specific random effect. We interpret the convex gradient $\nabla \phi_i$ as a functional error term capturing variation unexplained by the covariates.

Based on the insights of previous population-level studies of BMI \cite{life-course}, we consider two specifications for the design matrix $\bX$. Let $\bar{a}$ and $s_a$ denote the sample mean and standard deviation of cell ages, and let $\bar{c}$ and $s_c$ denote the analogous quantities for cohort midpoints. Define the normalized covariates
\begin{align*}
    \widetilde{a}_i &= \frac{a_i - \bar{a}}{s_a}, \qquad
    \widetilde{c}_i = \frac{c_i - \bar{c}}{s_c}, \qquad
    \widetilde{g}_i = \begin{cases} -1 & \text{if male} \\ +1 & \text{if female}. \end{cases}
\end{align*}

The \textbf{linear model} uses a design matrix $\bX \in \RR^{n \times 4}$ with rows
\begin{equation}
    \bx_i = (1, \widetilde{a}_i, \widetilde{c}_i, \widetilde{g}_i)^\top
\end{equation}
corresponding to an intercept, normalized age, normalized cohort, and gender indicator.

The \textbf{quadratic model} uses a design matrix $\bX \in \RR^{n \times 5}$ with rows
\begin{equation}
    \bx_i = (1, \widetilde{a}_i, \widetilde{a}_i^2, \widetilde{c}_i, \widetilde{g}_i)^\top
\end{equation}
incorporating a quadratic age term to capture the empirical observation that BMI tends to increase with age until approximately $70$ years, then plateaus or declines.

We normalized the continuous covariates to ensure the coefficients are on comparable scales and to facilitate interpretation. The intercept $\beta_0$ represents the predicted BMI distribution at the mean age and cohort, averaged across genders. The symmetric gender coding centers this covariate at zero, so that $\beta_{\text{gender}}$ captures the half-difference between female and male BMI distributions.

\subsubsection*{Model fitting and selection metrics}

To analyze the relationship between demographic covariates and BMI profiles, we fit two types of vector-covariate, distribution-response regression models to the data: Wasserstein least squares and Global Fréchet regression.

For the Wasserstein least squares approach, we evaluate several model specifications on the $n$ demographic cells, utilizing their covariate vectors $\{\bx_i\}_{i=1}^n$ and target BMI distributions $\{\nu_i\}_{i=1}^n$. These specifications vary primarily in two ways: we include an additional quadratic age covariate ($\text{age}^2$) in most models to capture nonlinear aging dynamics, and we scale the number of particles representing the fit to either $M = 2{,}000$ or $M = 20{,}000$, depending on the distributional resolution required for the analysis.

We compute the fit of the least squares estimator $\widehat{Q}$ using the procedure outlined in Algorithm \ref{alg:wls-gd-particle}. The estimator is represented as an empirical measure of $M$ particles $\{\boldsymbol{\beta}_j\}_{j=1}^M \subset \RR^p$, which we update via gradient descent with a learning rate $\tau_k = 0.1/(1 + 0.001k)$, a momentum coefficient $\rho = 0.9$, and a mini-batch size $B = 5$ over $T = 2{,}000$ iterations. The resulting particle cloud provides a nonparametric estimate of the distribution over the coefficient vectors $\boldsymbol{\beta} = (\beta_0, \beta_{\text{age}}, \beta_{\text{age}^2}, \beta_{\text{cohort}}, \beta_{\text{gender}})^\top$.

From this, we extract summary statistics including marginal distributions, pairwise correlations, and the covariance matrix
$$\Sigma_{\boldsymbol{\beta}}=\frac{1}{M-1}\sum_{j=1}^M(\boldsymbol{\beta}_j-\bar{\boldsymbol{\beta}})(\boldsymbol{\beta}_j-\bar{\boldsymbol{\beta}})^{\top},$$where $\bar{\boldsymbol{\beta}}=\frac{1}{M}\sum_{i=1}^M\boldsymbol{\beta}_i$.

As a benchmark for Wasserstein least squares, we fit the Global Fréchet regression model \citep{PetersenMuller2019} to the same distributional data. Let $S_{\nu_i}: (0,1) \to \RR$ denote the quantile function of the observed BMI distribution $\nu_i$ for cell $i$.
We evaluate these quantile functions on a uniform grid of $L = 500$ probability levels $\{p_\ell\}_{\ell=1}^L$ spanning $[0.001, 0.999]$. For each quantile level $p \in (0,1)$, we first compute the unconstrained least squares estimate:
\begin{equation}\label{eq:coefficients_frechet}
    \widetilde{\beta}(p) = (\beta_0(p),\dots, \beta_\text{gender}(p))^\top=(\bX^\top \bX)^{-1} \bX^\top S(p),
\end{equation}
where $S(p) = (S_{\nu_1}(p), \ldots, S_{\nu_n}(p))^\top \in \RR^n$ is the vector
of observed quantiles at level $\pp$, and $\bX \in \RR^{n \times 5}$ is the design
matrix with rows $\bx_i^\top$.

The unconstrained predicted quantile function for a covariate vector $\bx$ is
$\widetilde{S}_{\bx}(\pp) = \bx^\top \widetilde{\beta}(\pp)$. However, this prediction may
not be monotonically increasing in $\pp$, violating the definition of a quantile
function. Following the computational procedure in Section~6.1 of
\cite{PetersenMuller2019}, we project onto the space of valid quantile functions
via isotonic regression:
\begin{equation}
    \widehat{S}_{\bx} = \argmin_{q: q(\pp_1) \leq \cdots \leq q(\pp_L)}
    \sum_{\ell=1}^L \left( q(\pp_\ell) - \widetilde{S}_{\bx}(\pp_\ell) \right)^2.
\end{equation}
This projection is computed efficiently using the Pool-Adjacent-Violators Algorithm (PAVA) \citep{barlow1972statistical}\footnote{An important subtlety arises regarding monotonicity in Fréchet regression.
The unconstrained coefficients $\widetilde{\beta}(\pp)$ from \eqref{eq:coefficients_frechet} define a curve $\pp \mapsto \widetilde{\beta}(\pp)$ in coefficient space $\mathbb{R}^p$.
When visualizing the joint distribution of coefficient pairs $(\widetilde{\beta}_k(\pp), \widetilde{\beta}_j(\pp))$ as $\pp$ varies and where $j,k\in\{0,\text{age},\text{age}^2, \text{cohort},\text{gender}\}$, one might observe non-injective curves as in Figure \ref{fig:overall_comparison}.
This non-injectivity might suggest that the predicted quantile functions $\widetilde{S}_{\mathbf{x}}(\pp) = \mathbf{x}^\top \widetilde{\beta}(\pp)$ could violate monotonicity for certain covariate vectors $\mathbf{x}$.
In our BMI application, we verified empirically that $\widetilde{S}_{\mathbf{x}}(\pp)$ is monotonic for all observed covariate vectors $\mathbf{x}_i$, rendering the PAVA projection step unnecessary in practice. Nevertheless, for general applications where monotonicity violations may occur, the isotonic regression step described in \cite{PetersenMuller2019} remains essential.}.

To evaluate the goodness-of-fit and generalizability of our models, we conduct a comparative performance analysis using the Wasserstein coefficient of determination $R^2$, introduced in section 6.4 of \cite{PetersenMuller2019}, and leave-one-out cross-validation (LOO-CV).

As summarized in Table \ref{tab:r_squared}, the inclusion of the quadratic age term ($\widetilde{a}^2$) improves model fit for both methodologies.

\begin{table}[t]
    \centering
    \begin{tabular}{lc}
        \toprule
        \textbf{Model} & \textbf{$R^2$} \\
        \midrule
        Wasserstein least squares (Linear) & 0.7254  \\
        Wasserstein least squares (Quadratic) &\textbf{ 0.8958} \\
        Fréchet (Linear) & 0.7536  \\
        Fréchet (Quadratic) & 0.8932  \\
        \bottomrule
    \end{tabular}
    \caption{Wasserstein $R^2$ for distributional regression models on BMI data ($n=164$ cells).
    Computed as $R^2 = 1 - \sum_i W_2^2(\nu_i, \widehat{Q}_{\mathbf{x}_i}) / \sum_i W_2^2(\nu_i, \bar{\nu})$
    following \cite{PetersenMuller2019}, where $\bar{\nu}$ is the Wasserstein barycenter.
    Linear models use covariates $\mathbf{x} = (1, \widetilde{a}, \widetilde{c}, \widetilde{g})^\top$;
    quadratic models add $\widetilde{a}^2$.
    Both quadratic models explain ${\approx}89\%$ of distributional variance;
    Wasserstein least squares ($0.8958$) and Fréchet ($0.8932$) are indistinguishable at this aggregate level.}
    \label{tab:r_squared}
\end{table}

Furthermore, the $R^2$ metrics reveal that at the macroscopic level of explained variability, neither Wasserstein least squares nor Global Fréchet regression strictly dominates.

To assess out-of-sample predictive performance, we perform LOO-CV across the filtered sample of $n = 164$ demographic cells.
Given the computationally intensive nature of LOO-CV, we fitted the Wasserstein least squares model for this specific evaluation using an Adam optimizer with gradient clipping and early stopping ($M = 2000$ particles, patience $= 100$). Table \ref{tab:loo_cv} reports the mean and standard deviation of the $W_2$ distances. Fréchet regression exhibits slightly lower LOO error than Wasserstein least squares, suggesting a marginally higher out-of-sample stability.

\begin{table}[t]
    \centering
    \begin{tabular}{lc}
        \toprule
        \textbf{Model} & \textbf{LOO $W_2$} \\
        \midrule
        Fréchet (Quadratic) & 0.3813 $\pm$ 0.1618 \\
        Wasserstein least squares (Quadratic) & 0.4060 $\pm$ 0.1874 \\
        \bottomrule
    \end{tabular}
    \caption{Leave-one-out cross-validation results for distributional regression on BMI data
    ($n = 164$ cells; errors reported as mean $\pm$ std of $W_2$ distance).
    Fréchet achieves a marginally lower mean error ($0.381$ vs.\ $0.406$), a gap of $0.025$ $W_2$ units
    on average; the two error distributions overlap substantially (\cref{fig:loo_comparison}).}
    \label{tab:loo_cv}
\end{table}

However, a paired cell-by-cell comparison of the LOO errors provides a more nuanced picture (Figure \ref{fig:loo_comparison}). Histograms of the LOO $W_2$ errors show highly overlapping distributions, with the median error for Fréchet at $0.345$ and Wasserstein least squares at $0.370$, closely trailing. When comparing performance on identical cells, Fréchet achieves a lower LOO error in $61\%$ ($100/164$) of the cells. Ultimately, the mean difference in LOO error between the two approaches is a mere $0.025$ BMI units.
\begin{figure}[t]
  \centering
  \begin{subfigure}[t]{0.54\textwidth}
    \centering
    \includegraphics[width=\textwidth]{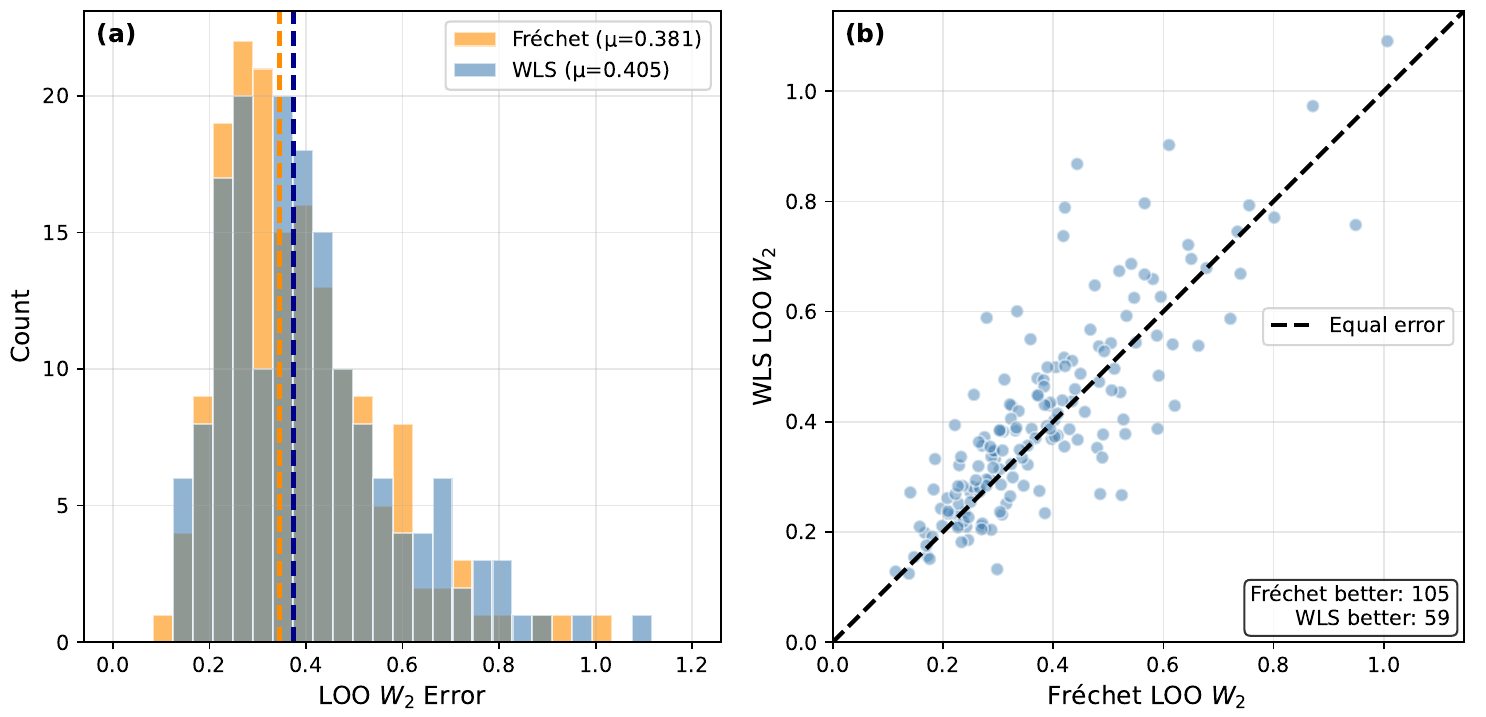}
    \caption{LOO-CV comparison. (a)~Histograms of LOO $W_2$ errors; dashed lines indicate medians. (b)~Paired cell-by-cell LOO errors; points above the diagonal favour Fréchet.}
    \label{fig:loo_comparison}
  \end{subfigure}
  \hfill
  \begin{subfigure}[t]{0.43\textwidth}
    \centering
    \includegraphics[width=\textwidth]{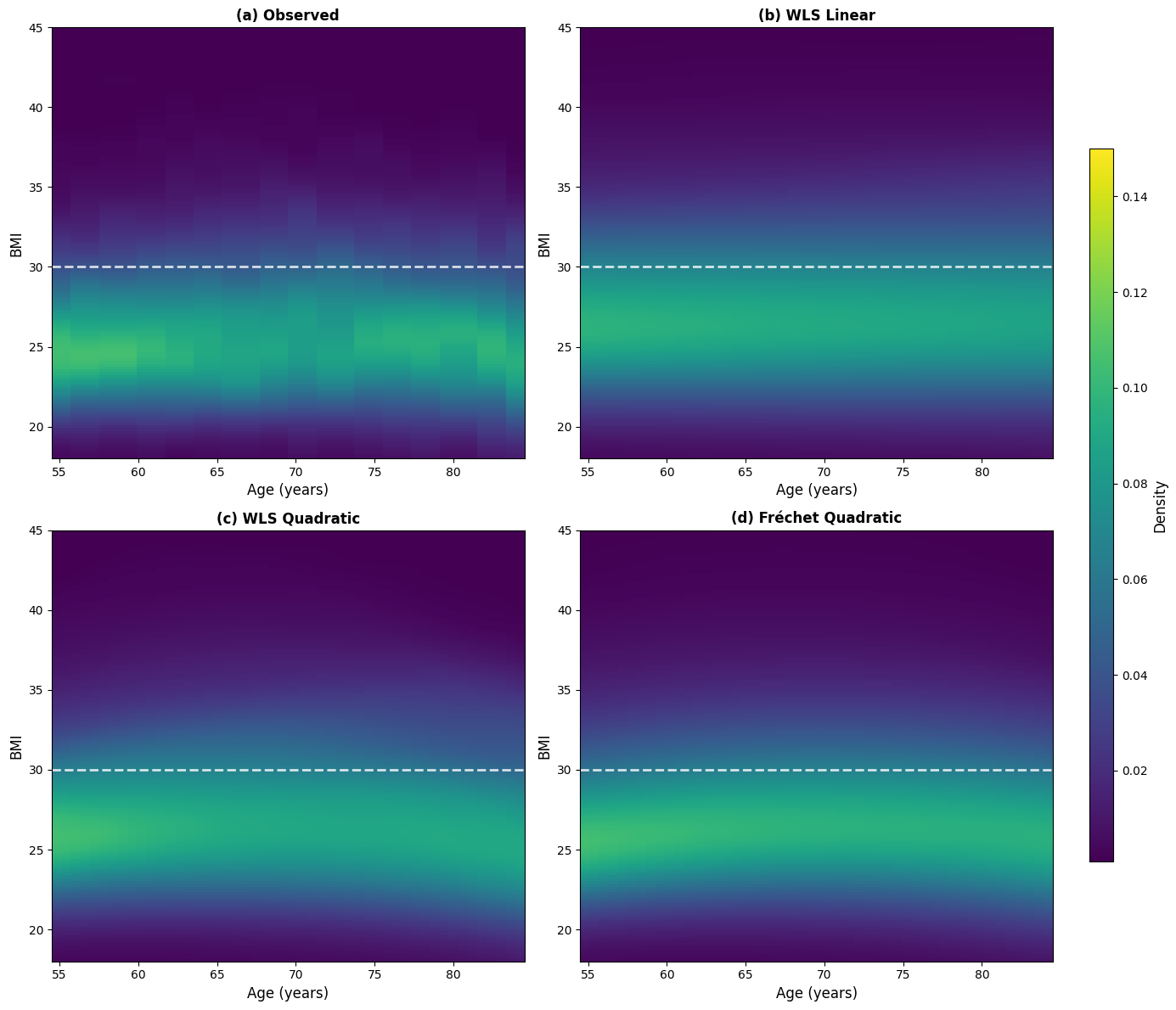}
    \caption{Density heatmaps of BMI distributions as a function of age. Male subjects, cohort 1935--1939. Darker colour $=$ higher probability mass.}
    \label{fig:density_heatmaps}
  \end{subfigure}
  \caption{Predictive parity between Wasserstein least squares and Fréchet regression on the BMI data.
  \emph{Left}: LOO $W_2$ error histograms (top) and paired cell-level comparison (bottom); the median difference is $0.025$ BMI units and $61\%$ of cells favour Fréchet, confirming that neither method dominates on raw out-of-sample fit.
  \emph{Right}: Fitted density heatmaps for male subjects, cohort 1935--1939; both models produce virtually identical distributions across the age range.
  The two methods are interchangeable for interpolation; structural differences emerge only when conditioning on individual-level observations (see \cref{fig:double_cond_main,fig:corner_main}).}
\end{figure}

Taken together, the quantitative $R^2$ and LOO-CV metrics, along with qualitative visual assessments, tell a consistent empirical story: when interpolating within the observed covariate space, Wasserstein least squares and Global Fréchet regression exhibit equivalent performance. This interpolation parity is evident in \cref{fig:density_heatmaps}, which shows the density heatmaps of the fitted BMI distributions $\nu_i$ for a representative demographic slice (males in the 1935–1939 birth cohort). Here, the distributions output by Wasserstein least squares and Fréchet regression are virtually indistinguishable. Both frameworks capture the nonlinear age dynamics present in the observed data.  Since neither model significantly outperforms the other in raw interpolating power, the choice between them hinges on their structural properties.

\subsection*{Coefficient measures, covariances, and joint probability distributions of the coefficients for the fitted models}

While both models interpolate effectively, their mathematical architectures differ in how they handle population heterogeneity. Wasserstein least squares extends the logic of linear mixed-effects modeling into a nonparametric setting by estimating a joint probability measure, $\widehat{Q}$, over the multidimensional coefficient space $\RR^p$. In contrast, Global Fréchet regression constructs its distributional response by independently regressing each quantile level $\pp$.

\begin{table}[t]
\centering
\begin{tabular}{lrrrrr}
\toprule
& \multicolumn{5}{c}{\textbf{Wasserstein least squares} ($M = 20{,}000$ particles)} \\
\cmidrule(lr){2-6}
Coefficient & Mean & SD & 2.5\% & 97.5\% & $\widehat P(\beta_j > 0)$ \\
\midrule
$\beta_0$                   & 28.479 & 5.468 & 19.513 & 41.057 & 100.0\% \\
$\beta_{{\rm age}}$       &  0.451 & 1.412 & $-1.909$ &  3.663 &  60.7\% \\
$\beta_{{\rm age}^2}$     & $-0.393$ & 0.490 & $-1.347$ &  0.143 &   9.9\% \\
$\beta_{{\rm cohort}}$    &  0.877 & 0.718 & $-0.225$ &  2.414 &  94.0\% \\
$\beta_{\rm gender}$        &  0.020 & 1.693 & $-2.084$ &  5.057 &  38.6\% \\
\midrule
& \multicolumn{5}{c}{\textbf{Fr\'echet} ($K = 500$ quantile levels)} \\
\cmidrule(lr){2-6}
Coefficient & Mean & SD & 2.5\% & 97.5\% & $\widehat P(\beta_j > 0)$ \\
\midrule
$\beta_0$                   & 28.483 & 5.699 & 19.458 & 41.930 & 100.0\% \\
$\beta_{{\rm age}}$       &  0.372 & 0.176 & $-0.072$ &  0.554 &  95.2\% \\
$\beta_{{\rm age}^2}$     & $-0.368$ & 0.082 & $-0.602$ & $-0.250$ &   0.0\% \\
$\beta_{{\rm cohort}}$    &  0.854 & 0.452 &  0.110 &  1.814 &  99.6\% \\
$\beta_{\rm gender}$        &  0.014 & 0.636 & $-0.789$ &  1.505 &  45.2\% \\
\bottomrule
\end{tabular}
\caption{
  Coefficient summaries for the Wasserstein least squares and Fr\'echet regression fits on HRS BMI data.
  \textbf{Wasserstein least squares}: each row summarises the marginal distribution of the $j$-th coordinate
  of $\widehat{Q}$ across $M = 20{,}000$ particles $(\beta_m)_{m=1}^M$ computed via Algorithm \ref{alg:wls-gd-particle}; Mean, SD, and prediction interval are empirical
  particle statistics.
  $\widehat P(\beta_j > 0)$ is the fraction of particles with positive $j$-th coordinate.
  \textbf{Fr\'echet}: each row summarises the $j$-th coordinate of the OLS coefficient
  vector $\widehat\beta(\pp)$ across $K = 500$ quantile levels $\pp \in (0,1)$;
  Mean, SD, and interval are taken over quantile levels.
  Covariates are normalized.
}
\label{tab:bmi_coef}
\end{table}

Table \ref{tab:bmi_coef} and Figure \ref{fig:marginals_comparison}  illustrate the practical consequences of this divergence. At a macro level, both methods agree on the baseline BMI ($\beta_0 \approx 28.5$) and identify the same primary population trends: later birth cohorts exhibit higher BMI ($\beta_{{\text{cohort}}} > 0$), and aging generally follows a negative quadratic trajectory ($\beta_{{\text{age}}^2} < 0$).

\begin{figure}[t]
    \centering
    \begin{subfigure}[t]{0.48\textwidth}
        \centering
        \includegraphics[width=\textwidth]{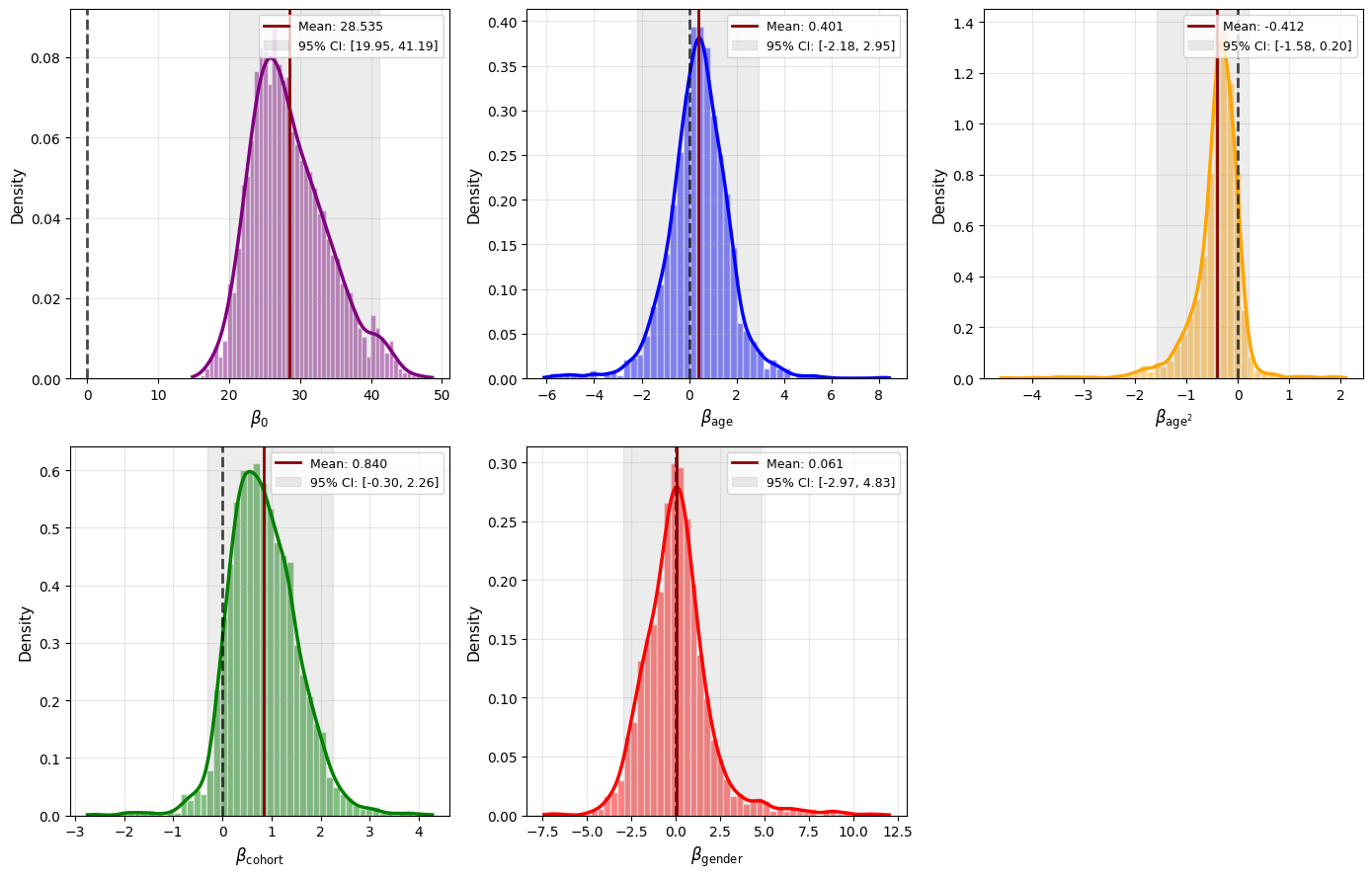}
        \caption{Wasserstein least squares}
        \label{fig:wls_marginals}
    \end{subfigure}
    \hfill
    \begin{subfigure}[t]{0.48\textwidth}
        \centering
        \includegraphics[width=\textwidth]{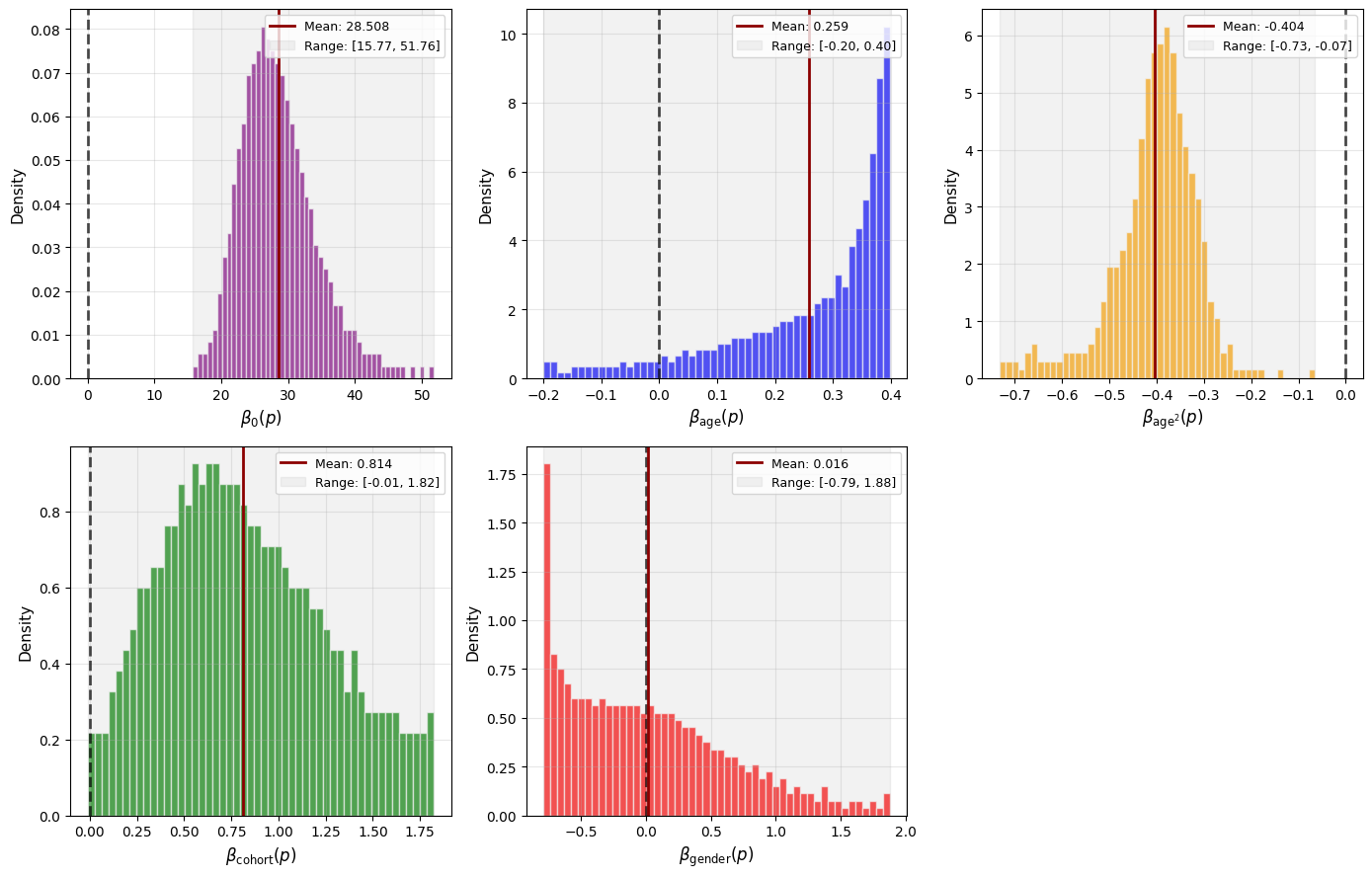}
        \caption{Fréchet regression}
        \label{fig:frechet_marginals}
    \end{subfigure}
    \caption{Comparison of coefficient distributions from Wasserstein least squares and Fréchet regression.
    \textbf{(a)} Marginal distributions of the random coefficient vector $\boldsymbol{\beta} \sim \widehat{Q}$
    estimated via Wasserstein least squares. Each panel shows the empirical distribution
    of one coefficient component from the $M = 2000$ particle representation of $\widehat{Q}$.
    Histogram with kernel density estimate overlay; vertical black dashed line at zero;
    The red vertical line indicates the mean; the gray shaded region shows the range from the $2.5$-th to the $97.5$-th quantile on each distribution.
    \textbf{(b)} Marginal distributions of coefficient functions $\beta_k(\pp)$ from Fréchet regression as in equation \eqref{eq:coefficients_frechet}.
    Unlike Wasserstein least squares, Fréchet yields pointwise estimates $\beta_k(\pp)$ for each quantile level
    $\pp \in [0.001, 0.999]$. }
    \label{fig:marginals_comparison}
\end{figure}

However, the marginal distributions reveal that Wasserstein least squares permits a broader range of structural variation. For example, Fréchet regression yields a quadratic age effect that is strictly negative across all quantiles ($P(\beta_{{\text{age}}^2} > 0) = 0$). Wasserstein least squares, conversely, indicates that while the average effect is concave, approximately $10\%$ of the mass allows for near-zero or slightly positive curvature. Likewise, Wasserstein least squares yields a substantially wider spread for the linear age and gender coefficients (e.g., a gender standard deviation of $1.69$ versus $0.64$).

Structurally, this flexibility enables Wasserstein least squares to model scenarios in which different segments of the BMI distribution experience demographic gradients in different ways, a localized heterogeneity that independent quantile mapping tends to constrain.

The structural difference between the two frameworks becomes most pronounced while examining the joint distribution of the regression coefficients (Figure \ref{fig:overall_comparison}). Since Global Fréchet regression constructs its estimates by independently evaluating each probability level $\pp \in (0,1)$, the resulting coefficient functions $\beta_k(\pp)$ are strictly parameterized by this single scalar. This induces near-perfect correlations across all coefficient pairs. Wasserstein least squares, by contrast, estimates a full multidimensional measure $\widehat{Q}$ that captures the covariance structure among the demographic variables.

\begin{figure}[t]
    \centering
    \begin{subfigure}[t]{0.48\textwidth}
        \centering
        \includegraphics[width=\textwidth]{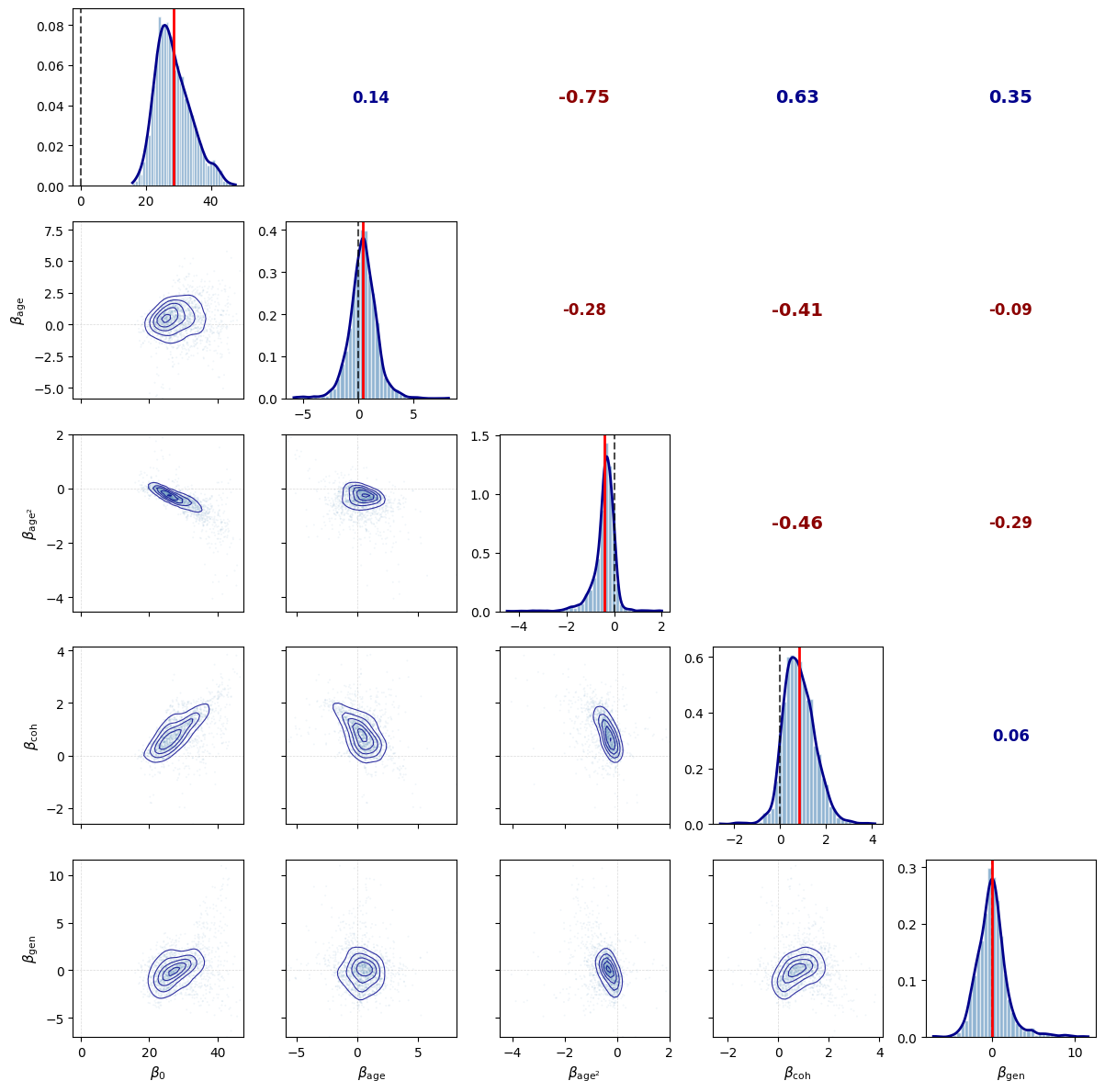}
        \label{fig:corner_plot}
    \end{subfigure}
    \hfill
    \begin{subfigure}[t]{0.48\textwidth}
        \centering
        \includegraphics[width=\textwidth]{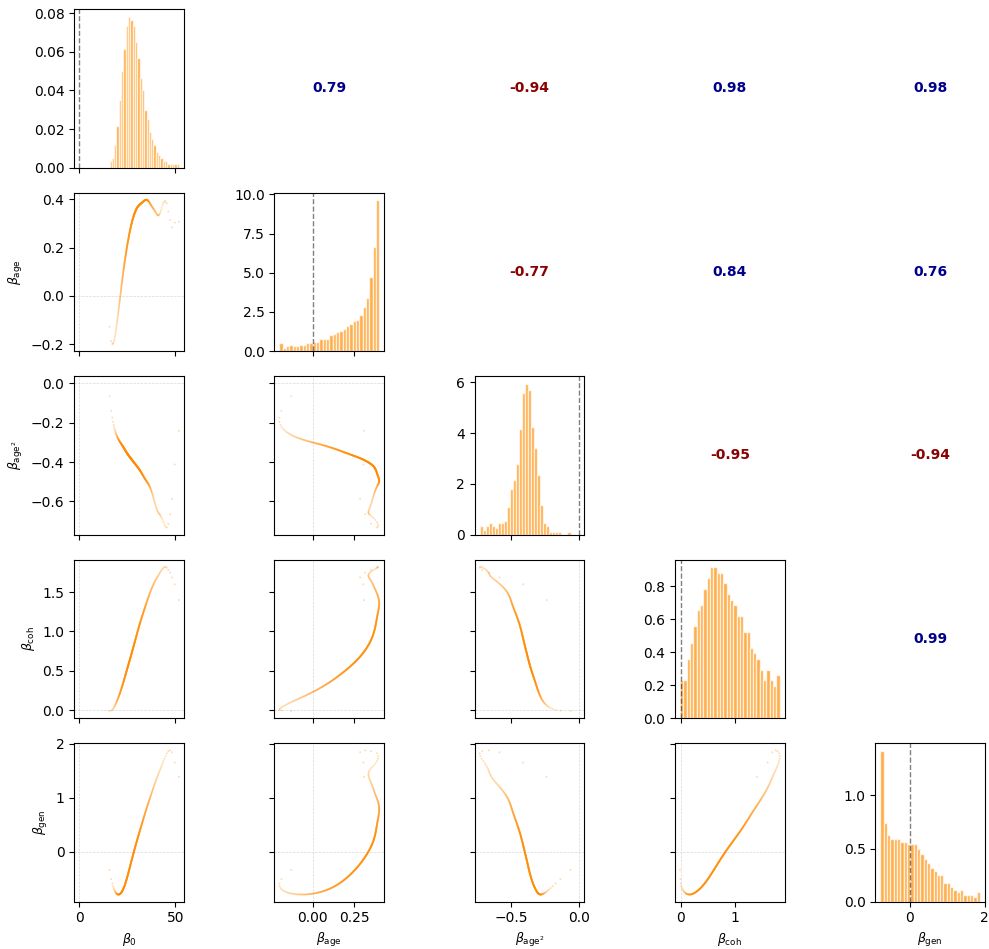}
        \label{fig:corner_comparison}
    \end{subfigure}
    \caption{Comparison of joint distribution structures between Wasserstein least squares and Fréchet regression. \textbf{To the left}: Corner plot showing the joint distribution structure of $\boldsymbol{\beta} \sim \widehat{Q}$.
        Diagonal panels: marginal distributions (histogram with KDE) for each coefficient;
        black vertical line at zero, red vertical line at the mean.
        Lower triangle: bivariate scatter plots showing pairwise joint distributions
        with contour lines from kernel density estimation.
        Upper triangle: Pearson correlation coefficients $\rho_{ij}$ between coefficient pairs;
        Blue text indicates positive correlation, red indicates negative. \textbf{To the right}: corner plots showing the joint distribution structure of the regression
        coefficients $\boldsymbol{\beta} = (\beta_0, \beta_{\mathrm{age}}, \beta_{\mathrm{age}^2},
        \beta_{\mathrm{cohort}}, \beta_{\mathrm{gender}})^\top$ for Fréchet
        regression. Diagonal panels show marginal distributions; lower triangles
        show pairwise scatter plots; upper triangles show Pearson correlations.}
    \label{fig:overall_comparison}
\end{figure}

This architectural difference has severe implications for conditional inference, exposing a critical limitation of the pointwise Fréchet approach, as demonstrated by our variance partition analysis (Figure \ref{fig:variance_partition}).

In a classical multivariate setting, we can compute the conditional standard deviation of the demographic effects given a specific baseline BMI ($\beta_0$) using the Schur complement:$$\Sigma_{\mathrm{rest}|\beta_0} = \Sigma_{22} - \Sigma_{21}\Sigma_{11}^{-1}\Sigma_{12}.$$
For Wasserstein least squares, this conditional variance remains substantial. It reflects the biological reality that subpopulations sharing the same baseline BMI profile can still exhibit highly varied physiological responses to aging and generational shifts.
For Fréchet regression, the conditional standard deviation collapses to approximately zero. Because its coefficients are perfectly correlated by the quantile index $\pp$, knowing the intercept deterministically dictates the exact slopes for age, cohort, and gender. This artificial rigidity structurally prevents the Fréchet model from capturing independent, localized demographic heterogeneity.

\begin{figure}[t]
    \centering
    \includegraphics[width=.6\textwidth]{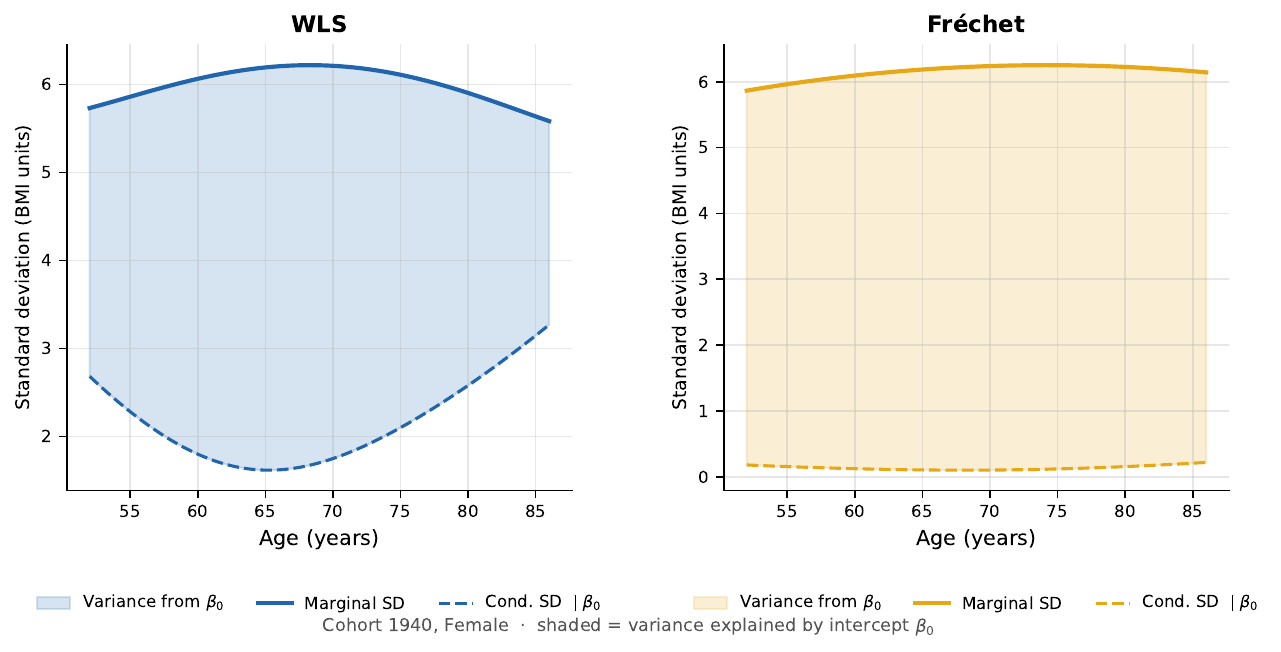}
    \caption{Marginal versus conditional standard deviation (cohort 1940, female). The shaded region represents the variance explained by knowing $\beta_0$.}
    \label{fig:variance_partition}
\end{figure}

Because Wasserstein least squares preserves the full joint distribution of the coefficients, it can explore rare population dynamics. Fréchet regression’s quantile-based construction forces certain structural assumptions—such as a strictly positive cohort effect—across the entire distribution. In contrast, interrogating the Wasserstein least squares particle cloud reveals distinct minority trajectories.

For instance, while the vast majority of the modeled population exhibits a late-life plateau or decline in BMI, approximately $10\%$ of the Wasserstein least squares particles yield a positive quadratic age coefficient ($\beta_{\mathrm{age}^2} > 0$). Within the model framework, this corresponds to a subpopulation experiencing accelerating BMI growth in later years (Figure \ref{fig:extreme_age2}).

To quantify how prevalent this convex pattern is in the data, we fit an individual quadratic regression $\mathrm{BMI}_i(t) = a_i + b_i t + c_i t^2$ to each trajectory in the 1940--44 female cohort and classify it as \emph{convex} if $c_i > 0$, the fitted minimum (at $t^* = -b_i/2c_i$) occurs strictly within the observed age window, and the predicted BMI at the last observation exceeds that minimum by at least 3.5 units. Under this definition, approximately 10.0\% of individuals (197 out of 1,964 with at least four observations) exhibit a genuine upturn; the remaining 90.0\% are concave or monotone. Figure~\ref{fig:convex_vs_concave} visualizes the two groups: the majority (blue) follows the familiar plateau-and-decline pattern, while the minority (orange) shows a clear BMI minimum in mid-follow-up, followed by a rise at older ages.

\begin{figure}[t]
    \centering
    \includegraphics[width=.72\textwidth]{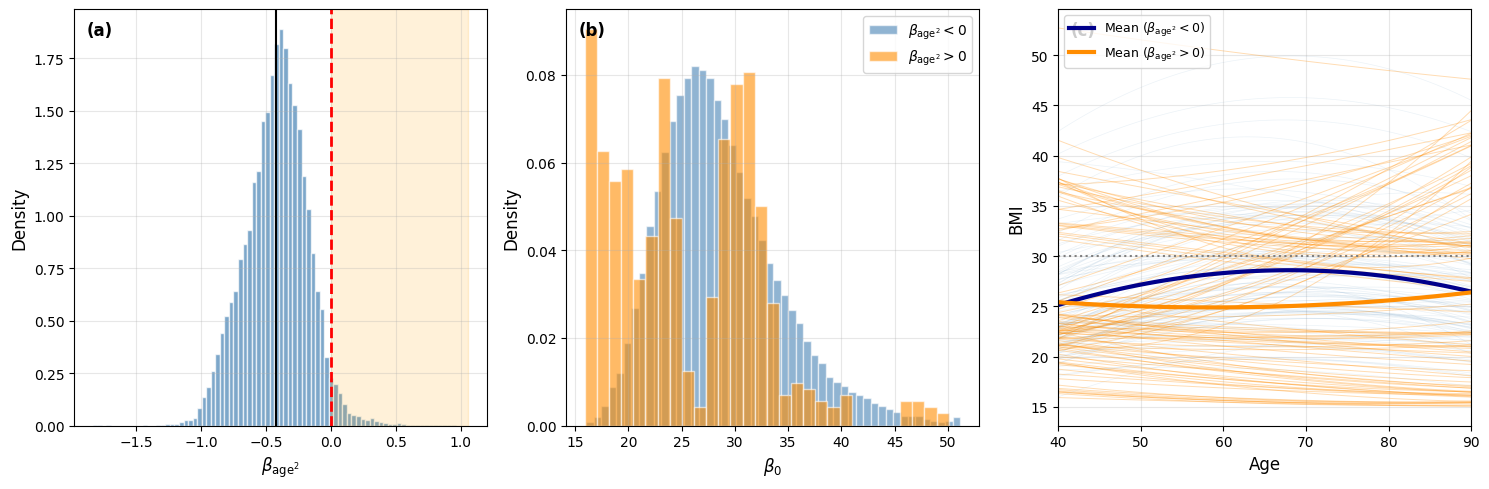}
    \caption{Analysis of particles with $\beta_{\mathrm{age}^2} > 0$ (accelerating BMI growth).
    Using $M = 20{,}000$ particles, approximately $10\%$ exhibit positive
    quadratic age coefficients, representing individuals whose BMI continues to accelerate
    with age rather than plateau or decline.
    (a)~Marginal distribution of $\beta_{\mathrm{age}^2}$; vertical dashed line indicates zero,
    solid line shows the mean. Shaded region highlights particles with $\beta_{\mathrm{age}^2} > 0$.
    (b)~Conditional distribution of $\beta_0$ (intercept) given the sign of $\beta_{\mathrm{age}^2}$.
    (c)~BMI trajectories for cohort 1940, female. Blue curves show 100 sampled trajectories
    with $\beta_{\mathrm{age}^2} < 0$ (typical); orange curves show trajectories with
    $\beta_{\mathrm{age}^2} > 0$ (accelerating). Thick lines indicate conditional means.
    Horizontal line marks the obesity threshold ($\mathrm{BMI} = 30$).}
    \label{fig:extreme_age2}
\end{figure}

Similarly, the model identifies a reverse cohort effect for roughly $10\%$ of the particles ($\beta_{\mathrm{cohort}} < 0$), suggesting a subset of individuals for whom more recent birth cohorts actually exhibit lower BMI scores (Figure \ref{fig:extreme_cohort}).

\begin{figure}[t]
    \centering
    \includegraphics[width=.65\textwidth]{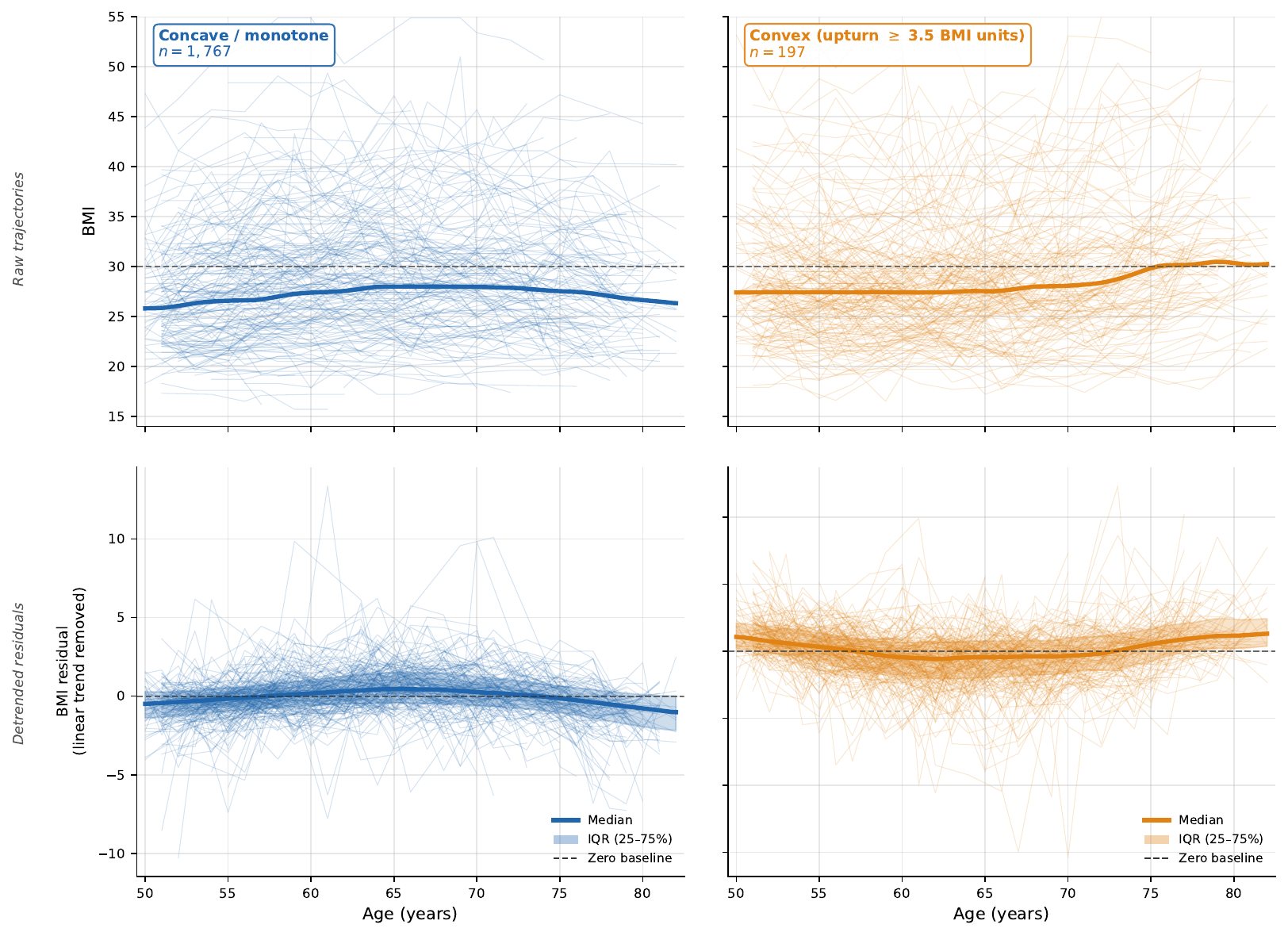}
    \caption{\textbf{Concave/monotone vs.\ convex BMI trajectories, cohort 1940--44, female.}
    \emph{Top row:} individual observed BMI trajectories (thin lines) with group median (thick line);
    dashed line at BMI~$= 30$.
    \emph{Bottom row:} detrended residuals after removing each individual's personal linear trend, with smoothed median and interquartile range (IQR) band (25th--75th percentile); dashed line at zero (no curvature).
    \emph{Left (blue):} the 90\% majority whose trajectories are concave or monotone.
    \emph{Right (orange):} the 10\% minority classified as convex ($n = 197$ out of 1,964 with $\geq 4$ observations) --- BMI reaches a minimum in mid-follow-up and rises again at older ages (upturn $\geq 3.5$ BMI units above the fitted vertex). The  residuals confirm that the upward curvature in the convex group is a systematic feature, not an artifact of baseline level differences.}
    \label{fig:convex_vs_concave}
\end{figure}

Crucially, these minority coefficients do not appear as uniformly distributed noise within the Wasserstein least squares framework.  When mapped onto the joint distribution of the intercept and linear age effect $(\beta_0, \beta_{\mathrm{age}})$, the particles exhibiting these extreme $\beta_{\mathrm{age}^2}$ and $\beta_{\mathrm{cohort}}$ values cluster in specific regions (Figure \ref{fig:extreme_joint}). This localized clustering points toward a complex demographic profile. For instance, the cluster of particles with $\beta_{\mathrm{age}^2} > 0$ formed around $(30, -1)$ suggests that a subgroup of retirees with a high starting BMI exhibits a convex weight trajectory, meaning these individuals experience accelerating weight gain toward the later years of their lives.

\begin{figure}[t]
    \centering
    \includegraphics[width=.72\textwidth]{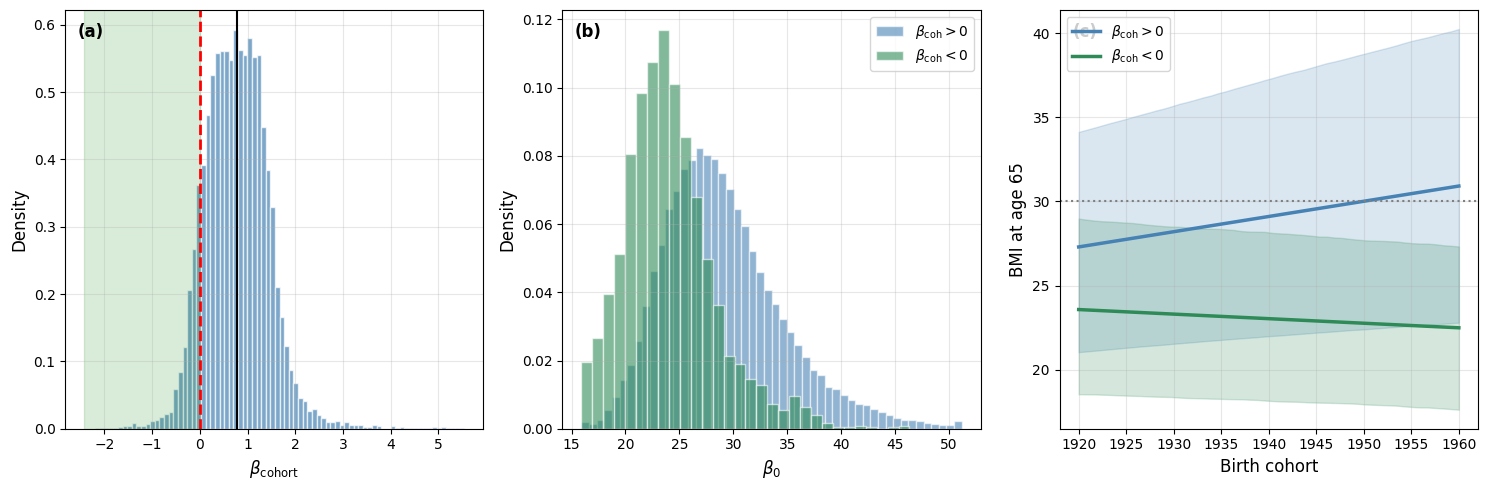}
    \caption{Analysis of particles with $\beta_{\mathrm{cohort}} < 0$ (reverse cohort effect).
    Using $M = 20{,}000$ particles, approximately $6\%$ exhibit negative
    cohort coefficients, representing individuals for whom later birth cohorts have lower
    BMI---opposite of the typical obesity epidemic trend.
    (a)~Marginal distribution of $\beta_{\mathrm{cohort}}$; vertical dashed line indicates zero.
    Shaded region highlights particles with $\beta_{\mathrm{cohort}} < 0$.
    (b)~Conditional distribution of $\beta_0$ (intercept) given the sign of $\beta_{\mathrm{cohort}}$.
    (c)~Mean predicted BMI at age 65 as a function of birth cohort for female respondents,
    stratified by the sign of $\beta_{\mathrm{cohort}}$. Shaded bands show 80\% prediction intervals.
    For typical particles ($\beta_{\mathrm{cohort}} > 0$), later cohorts have higher BMI;
    for the minority with $\beta_{\mathrm{cohort}} < 0$, the trend reverses.
    Horizontal line marks the obesity threshold.}
    \label{fig:extreme_cohort}
\end{figure}

While establishing these minority trajectories as distinct biological phenomena requires subsequent empirical verification, identifying them illustrates a practical utility of the Wasserstein least squares framework. By extending the machinery of classical linear regression to the Wasserstein space, Wasserstein least squares provides an interpretability tool alongside its predictive capabilities. It allows researchers to observe localized structural heterogeneity directly from the model outputs, offering a principled way to generate hypotheses from complex distributional data.

\begin{figure}[t]
    \centering
    \includegraphics[width=.65\textwidth]{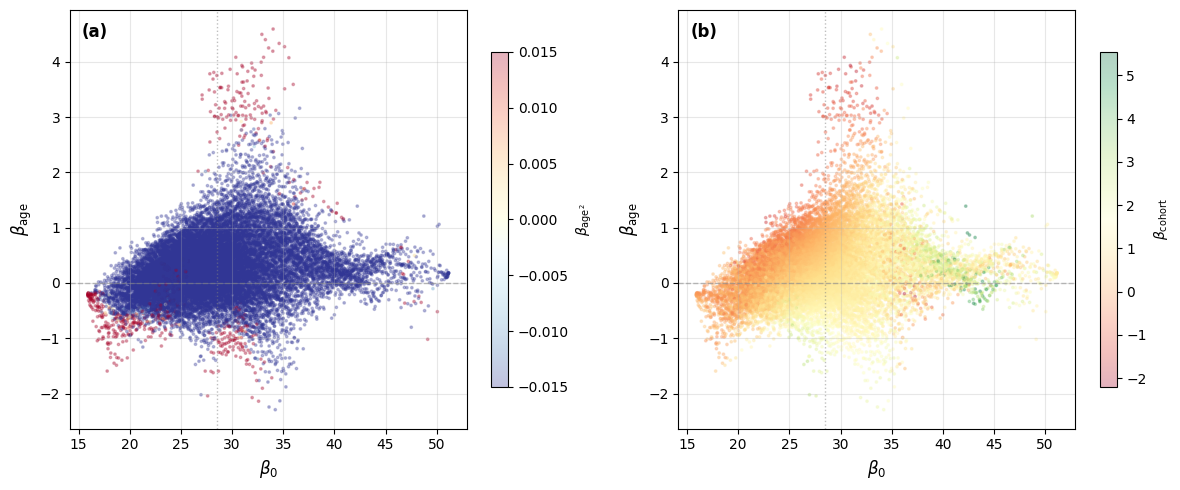}
    \caption{Joint distribution of coefficients colored by extreme values ($M = 20{,}000$ particles).
    (a)~Scatter plot of $(\beta_0, \beta_{\mathrm{age}})$ colored by $\beta_{\mathrm{age}^2}$.
    Particles with $\beta_{\mathrm{age}^2} > 0$ (red/orange) tend to cluster in specific regions
    of the $(\beta_0, \beta_{\mathrm{age}})$ space, revealing the correlation structure
    among coefficients.
    (b)~Same scatter plot colored by $\beta_{\mathrm{cohort}}$. Particles with
    $\beta_{\mathrm{cohort}} < 0$ (red) show distinct patterns in the joint distribution.}
    \label{fig:extreme_joint}
\end{figure}

\subsection*{Conditional trajectory and obesity persistence}

The obesity epidemic represents one of the most pressing public health crises of our time \citep{Ward2019Projected}. From an epidemiological perspective, modeling longitudinal BMI trajectories is critical, particularly regarding the persistence of obesity, which the Centers for Disease Control and Prevention (CDC) defines as BMI $\geq$ 30 \citep{cdc_bmi_2024}. Population-level interventions rely on understanding not just average weight gain, but the genuine probability that individuals at the threshold will either cross into severe obesity or return to healthier baseline levels over time. It is in this predictive, conditional context that the structural advantages of Wasserstein least squares over Fréchet regression become apparent.

To demonstrate this clinical utility, we compute conditional BMI trajectories and obesity probabilities for both models. For Wasserstein least squares, consider the coefficient distribution $\widehat{Q}$ represented via $M = 20{,}000$ particles
$\{\boldsymbol{\beta}^{(m)}\}_{m=1}^{M}$. To compute
trajectories conditional on $\text{BMI} \approx 31$ at age 50, we identify the subset
of particles satisfying the conditioning criterion
\begin{equation}
\mathcal{M}_{50} = \bigl\{ m : \bigl| \mathbf{x}(50)^\top \boldsymbol{\beta}^{(m)} - 31 \bigr| \leq 2 \bigr\},
\end{equation}
where $\mathbf{x}(a) = (1, \widetilde{a}, \widetilde{a}^2, \widetilde{c}, \widetilde{g})^\top$ denotes
the covariate vector at age $a$ with normalized components $\widetilde{a} = (a - \bar{a})/s_a$,
$\widetilde{c} = (c - \bar{c})/s_c$, and $\widetilde{g} \in \{-1, +1\}$ for gender. The conditional
mean trajectory at age $a$ is then
\begin{equation}
\widehat{\mu}(a) = \frac{1}{|\mathcal{M}_{50}|} \sum_{m \in \mathcal{M}_{50}} \mathbf{x}(a)^\top \boldsymbol{\beta}^{(m)},
\end{equation}
and the $C\%$ prediction interval is given by the $(100-C)/2$th and $( 100+C)/2$th percentiles of the
distribution $\{ \mathbf{x}(a)^\top \boldsymbol{\beta}^{(m)} \}_{m \in \mathcal{M}_{50}}$.
The probability of obesity at age $a$, conditional on BMI at age 50, is computed as
the proportion of conditional particles exceeding the threshold:
\begin{equation}
\mathrm{P}(\mathrm{BMI} \geq 30 \mid a, \mathrm{BMI}_{50} \approx 31)
= \frac{\bigl| \bigl\{ m \in \mathcal{M}_{50} : \mathbf{x}(a)^\top \boldsymbol{\beta}^{(m)} \geq 30 \bigr\} \bigr|}{|\mathcal{M}_{50}|}.
\end{equation}

For Fr\'echet regression, the coefficient vectors $\{\boldsymbol{\beta}(\pp)\}_{\pp \in \mathcal{P}}$
are indexed by quantile level $\pp \in (0,1)$, discretized over a fine grid $\mathcal{P}$.
Conditioning proceeds analogously by identifying quantile levels that produce the
target BMI at age 50
\begin{equation}
\mathcal{P}_{50} = \bigl\{ \pp \in \mathcal{P} : \bigl| \mathbf{x}(50)^\top \boldsymbol{\beta}(\pp) - 31 \bigr| \leq 2 \bigr\}.
\end{equation}
Mean trajectories and prediction intervals are computed over this subset, and the
conditional obesity probability takes the form
\begin{equation}
\mathrm{P}(\mathrm{BMI} \geq 30 \mid a, \mathrm{BMI}_{50} \approx 31)
= \frac{\bigl| \bigl\{ \pp \in \mathcal{P}_{50} : \mathbf{x}(a)^\top \boldsymbol{\beta}(\pp) \geq 30 \bigr\} \bigr|}{|\mathcal{P}_{50}|}.
\end{equation}

To incorporate a second observation, we further restrict to particles satisfying
conditions at both ages. For instance, to condition on $\text{BMI}_{50} \approx 31$
and $\text{BMI}_{60} > 31$, we define
\begin{equation}
\mathcal{M}_{50,60} = \bigl\{ m : \bigl| \mathbf{x}(50)^\top \boldsymbol{\beta}^{(m)} - 31 \bigr| \leq 2
\;\;\text{and}\;\; \mathbf{x}(60)^\top \boldsymbol{\beta}^{(m)} > 31 \bigr\}.
\end{equation}
Trajectories and prediction intervals are computed over $\mathcal{M}_{50,60}$ using
the same formulas as above. The conditional obesity probability with sequential
conditioning becomes
\begin{equation}
\mathrm{P}(\mathrm{BMI} \geq 30 \mid a, \mathrm{BMI}_{50}, \mathrm{BMI}_{60} \in \mathcal{C})
= \frac{\bigl| \bigl\{ m \in \mathcal{M}_{50,60} : \mathbf{x}(a)^\top \boldsymbol{\beta}^{(m)} \geq 30 \bigr\} \bigr|}{|\mathcal{M}_{50,60}|},
\end{equation}
where $\mathcal{C}$ denotes the conditioning set for BMI at age 60. In our analysis,
we consider $\mathcal{C} = (33, \infty)$ for continued weight gain, $\mathcal{C} = (31, \infty)$
for sustained obesity, and $\mathcal{C} = (-\infty, 30)$ for return below the obesity threshold.

The practical consequences of these different conditioning architectures can be understood in two complementary ways: the forecasted BMI trajectories (Figure \ref{fig:obesity_persistence}) and the corresponding age-dependent probability of obesity (Figure \ref{fig:obesity_probability_tetraptych}).

When conditioning on a single observation at the threshold (BMI = 31 at age 50), Wasserstein least squares yields wide 80\% prediction intervals spanning roughly 8 to 12 BMI units. This is not a lack of precision but a reflection of genuine biological uncertainty; as observed, the empirical distribution of the trajectories fans out in a manner similar to the prediction interval.

Consequently, the conditional probability of obesity (Panel 1 of Figure \ref{fig:obesity_probability_tetraptych})  produces smooth, clinically meaningful curves. An individual with a BMI of 31 at age 50 faces a steady decline in obesity risk, reaching roughly 50\% by age 90, indicating that individuals with the same current BMI may follow vastly different long-term trajectories.

In contrast, conditioning the Fréchet model on the same single observation exhibits a highly constrained behavior. Because the target BMI selects only a narrow cluster of quantile levels (roughly 39 to 48 out of 500), and each level defines a deterministic trajectory, the resulting prediction intervals are exceptionally narrow. This forces the conditional obesity probabilities (Panel 2 of Figure \ref{fig:obesity_probability_tetraptych}) into near-binary states with sharp transition zones. The apparent predictive precision in this case is largely a structural consequence of pointwise quantile regression, which inherently limits the representation of individual-level variance within the forecasted paths.

This structural divergence is further highlighted when sequential conditioning is introduced to update longitudinal predictions. For Wasserstein least squares, incorporating a second measurement at age 60 refines the forecast. As seen in the right panel of Figure \ref{fig:obesity_persistence}, stratifying the age 60 observations into three outcomes (continued weight gain, stable obesity, or returning below the threshold) differentiates future trajectories and narrows the prediction intervals.

The probability curves adjust accordingly (Panel 3 of Figure \ref{fig:obesity_probability_tetraptych}): a second measurement below the obesity threshold shifts the predicted outcome, dropping the long-run obesity probability toward zero. This indicates that Wasserstein least squares accommodates sequential updating, where the initial predictive variance is reduced as new observations are incorporated.

By contrast, applying this sequential conditioning illustrates a structural constraint within the Fréchet framework. By age 60, only 7 to 38 quantile levels satisfy the criteria for the stable or weight-gain scenarios. Notably, for the subset representing a return below the obesity threshold ($\text{BMI} < 30$ at age 60), zero quantile levels survive. Because every Fréchet trajectory that predicts $\text{BMI} \approx 31$ at age 50 in this specific fit also predicts $\text{BMI} \geq 30$ at age 60, this particular sequence of observations cannot be represented under the fitted model.

\begin{figure}[t]
    \centering
    \includegraphics[width=\textwidth]{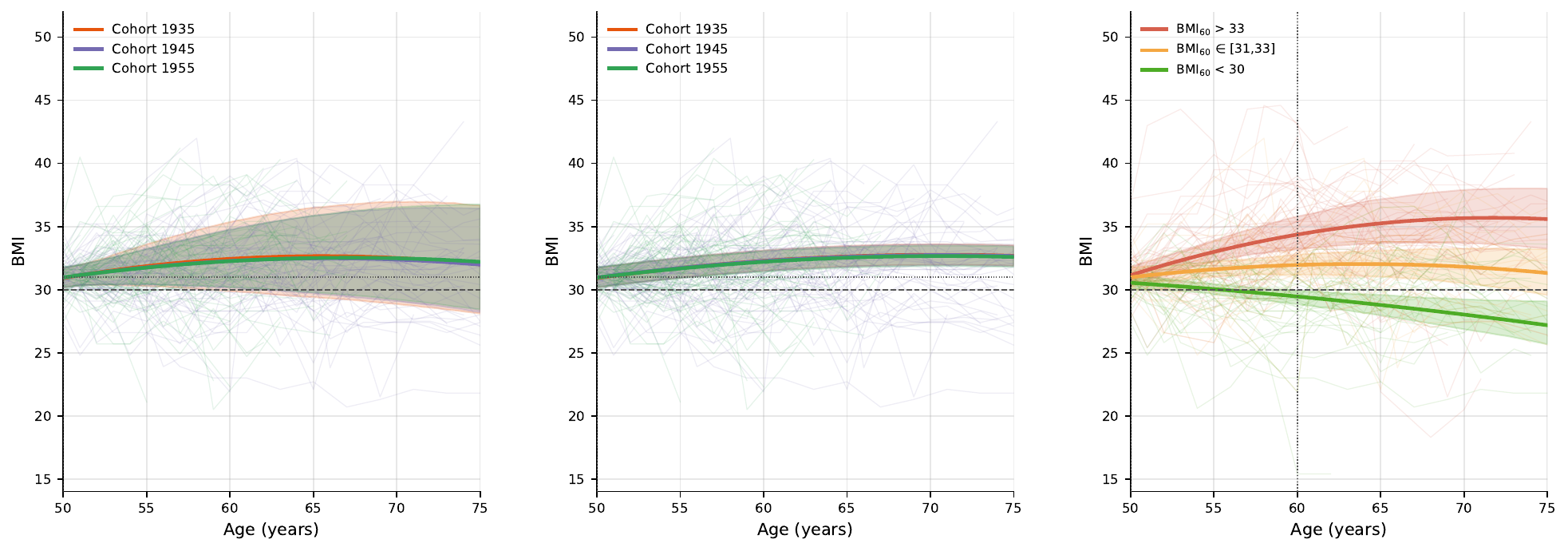}
    \caption{\textbf{Observed BMI trajectories (thin lines) overlaid on model-implied
    80\,\% prediction bands (shaded) for individuals at the obesity threshold
    (BMI\,$= 31$ at age\,50), ages 50--75.}
    Prediction bands are the 10th--90th percentile of
    $\{x(\text{age})^\top \boldsymbol{\beta}_m\}_{m=1}^{M'}$,
    where $\boldsymbol{\beta}_m$ are the $M' \leq M$ particles (Wasserstein least squares) or quantile-level coefficient vectors (Fr\'echet) that survive the conditioning
    criterion $\widehat{y}(50) \in [30,32]$.
    Dashed horizontal line: obesity threshold (BMI\,$= 30$);
    dotted horizontal line: conditioning level (BMI\,$= 31$);
    dotted vertical line: conditioning age (50).
    \textbf{(Left)~Wasserstein least squares} ($M = 20{,}000$ particles, three birth cohorts, female respondents): orange = cohort 1935, purple = cohort 1945, green = cohort 1955.
    The 80\,\% prediction bands are wide ($\approx$\,8--12 BMI units).
    Conditioning on a single BMI observation does not uniquely identify the coefficient vector, so the joint distribution $\widehat{Q}$ retains substantial spread.
    Observed trajectories largely fall within the prediction intervals.
    \textbf{(Center)~Fr\'echet} (same conditioning):
    After conditioning on BMI\,$=31$ at age\,50, only 39--48 quantile-level coefficient vectors $\widehat\beta(\pp)$ satisfy the criterion, and since all coefficients are determined by the same quantile level $\pp$,
    The entire trajectory is nearly fixed.
    The near-degenerate prediction bands contrast with the heterogeneous spread of observed individual paths, revealing that Fr\'echet's apparent precision is a structural artifact rather than genuine predictive accuracy.
    \textbf{(Right)~Wasserstein least squares with sequential conditioning}
    (cohort 1945, female; dotted vertical lines at ages 50 and 60): starting from BMI\,$\approx 31$ at age\,50, a second observation at age\,60 focuses predictions.
    Red: BMI$_{60} > 33$ (continued weight gain, $n = 541$ particles, $n_{\text{obs}} = 42$);
    orange: BMI$_{60} \in [31,33]$ (stable obesity, $n = 697$, $n_{\text{obs}} = 34$); green: BMI$_{60} < 30$ (returned below threshold, $n = 132$, $n_{\text{obs}} = 32$). The three groups trace clearly separated trajectories with substantially narrowed prediction bands, demonstrating that Wasserstein least squares uncertainty is meaningful and resolves as additional observations accumulate.}
    \label{fig:obesity_persistence}
\end{figure}

\begin{figure}[t]
    \centering
    \includegraphics[width=\textwidth]{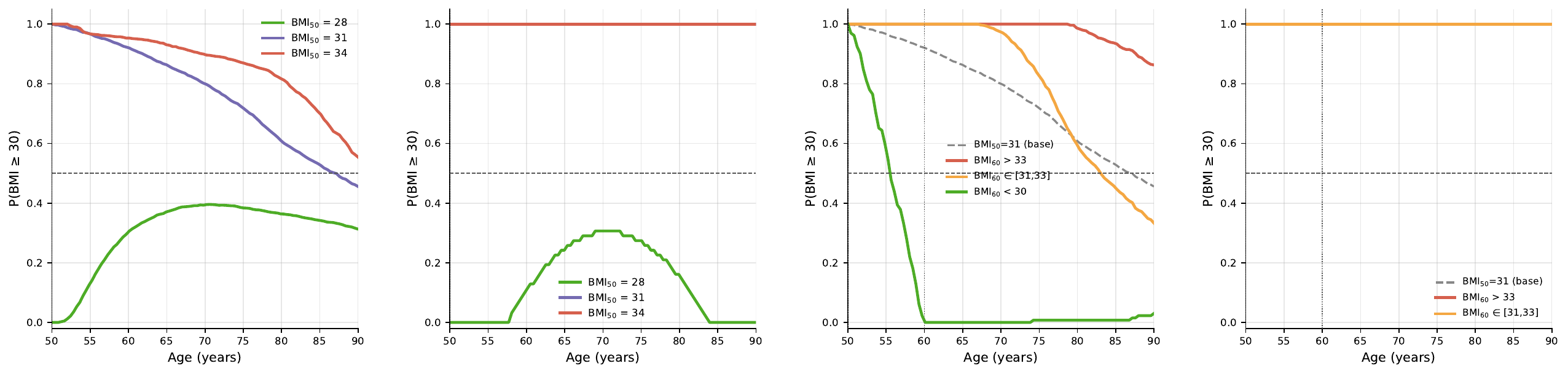}
    \caption{\textbf{Probability of obesity ($\mathrm{BMI} \geq 30$) as a
    function of age for females in cohort 1945, comparing Wasserstein least squares and
    Fr\'echet regression under single and sequential conditioning.}
    \textbf{(Panel 1)~Wasserstein least squares, single conditioning.}
    Conditioning on BMI at age~50 produces smooth, clinically meaningful probability curves.
    Individuals starting below the threshold ($\mathrm{BMI}_{50}=28$, green) show a modest, gradually rising probability of obesity, peaking at $30\%$--$40\%$ in their seventies before declining.
    Those just above ($\mathrm{BMI}_{50}=31$, purple) begin at certainty and see their probability decline steadily, reaching roughly $50\%$ by age~90, reflecting long-run uncertainty.
    Individuals well into obesity ($\mathrm{BMI}_{50}=34$, red) maintain high probability ($>80\%$) throughout, with a gradual decline toward $60\%$ at age~90.
    The smooth transitions arise because the Wasserstein least squares posterior retains heterogeneity among the coefficient draws that survive conditioning: individuals with the same current BMI may follow different future trajectories.
    \textbf{(Panel 2)~Fr\'echet, single conditioning.}
    Near-binary probabilities (close to 0 or~1) with sharp transition zones.
    In Fr\'echet regression, the fitted object is a collection of $K=500$ quantile-level coefficient vectors $\widehat\beta_k$, one per quantile level $k/K$; conditioning on an observed BMI range selects only those quantile levels whose predicted value falls within the observed range.
    Since each $\widehat\beta_k$ traces a deterministic trajectory, the selected quantile levels are clustered ($n \approx 28$--$62$ levels survive), implying near-certain future BMI predictions given a single current measurement.
    This is a structural feature of pointwise quantile regression rather than a reflection of the underlying biological uncertainty.
    \textbf{(Panel 3)~Wasserstein least squares, sequential conditioning.}
    Starting from $\mathrm{BMI}_{50} \approx 31$ (gray dashed, base), a second observation at age~60 differentiates long-run outcomes.
    When $\mathrm{BMI}_{60} > 33$ (red), the obesity probability rises toward $100\%$ and remains persistently high.
    When $\mathrm{BMI}_{60} \in [31,33]$ (orange), the probability stays elevated but begins to decline after age~70, reaching roughly $60\%$ at age~80.
    When $\mathrm{BMI}_{60} < 30$ (green), the probability drops toward zero---a second healthy-range measurement effectively reclassifies the long-run prognosis.
    \textbf{(Panel 4)~Fr\'echet, sequential conditioning.}
    Applying the same double conditioning to the Fr\'echet quantile-level coefficients leaves only 7--38 quantile levels for the $\mathrm{BMI}_{60} > 33$ and $\mathrm{BMI}_{60} \in [31,33]$ scenarios; the improved scenario retains zero levels, as every quantile-level trajectory that predicts $\mathrm{BMI}_{50} \approx 31$
    also predicts $\mathrm{BMI}_{60} \geq 30$, making a healthy second observation structurally impossible under the fitted model.
    Horizontal dashed line: $P = 0.5$; vertical dotted lines: conditioning ages (50 and 60).}
    \label{fig:obesity_probability_tetraptych}
\end{figure}

To empirically evaluate the predictive structures of both models, we compare their conditional forecasts against the observed trajectories of the corresponding HRS cohorts (males and females born 1935–1939 and 1940–1944 with $\text{BMI} \in [29, 33]$ at age 50). Table \ref{tab:coverage_single} summarizes the empirical coverage of the prediction intervals averaged across age bins 50 to 75.

Wasserstein least squares, which retains 2,953 posterior draws after conditioning, yields an empirical coverage of 0.877 for its 99\% PI and 0.579 for its 75\% PI. In contrast, the Fréchet model selects 78 quantile-level vectors, resulting in empirical coverages of 0.456 and 0.355 for the 99\% and 75\% PIs, respectively. The larger negative gap between empirical and nominal coverage in the Fréchet model indicates that its constrained trajectories do not fully encompass the observed spread of individual outcomes.

\begin{table}[t]
  \centering
  \small
  \begin{tabular}{llcccr}
    \toprule
    Method & PI level & $n_\text{ptcl}$ & $n_\text{real}$ & Mean emp.\ coverage & Gap \\
    \midrule
    Wasserstein least squares & 99\% & 2953 & 632 & 0.877 & -0.113 \\
    Wasserstein least squares & 75\% & 2953 & 632 & 0.579 & -0.171 \\
    Fréchet & 99\% & 78 & 632 & 0.456 & -0.534 \\
    Fréchet & 75\% & 78 & 632 & 0.355 & -0.395 \\
    \bottomrule
  \end{tabular}
  \caption{\textbf{Empirical PI coverage for single conditioning} (BMI $\in [29, 33]$ at age~50, cohorts 1935--39 and 1940--44, both sexes; trajectory shown for the 1935--39 female reference group). For each method, we report the mean empirical coverage across age bins 50--75 for the 75\% and 99\% prediction intervals, together with the gap (empirical $-$ nominal). Wasserstein least squares retains heterogeneous posterior draws, whereas Fr\'echet selects quantile-level coefficient vectors; the large negative gap for Fr\'echet reflects that a handful of near-deterministic quantile trajectories cannot represent the spread of individual outcomes.}
  \label{tab:coverage_single}
\end{table}

This difference in coverage is visually apparent in Figure \ref{fig:conditioning_wls_vs_frechet}. The Wasserstein least squares prediction intervals gradually widen with age, remaining generally centered on the empirical box plots and accommodating the varying long-term trajectories of individuals starting from a similar baseline.

By contrast, the Fréchet intervals are substantially narrower and tend to drift above the empirical distributions at older ages. Because the pointwise quantile regression framework largely pre-determines a quantile's future path once its level is fixed at age 50, it structurally understates the inherent variability in long-term BMI progression.

\begin{figure}[t]
    \centering
    \includegraphics[width=.68\textwidth]{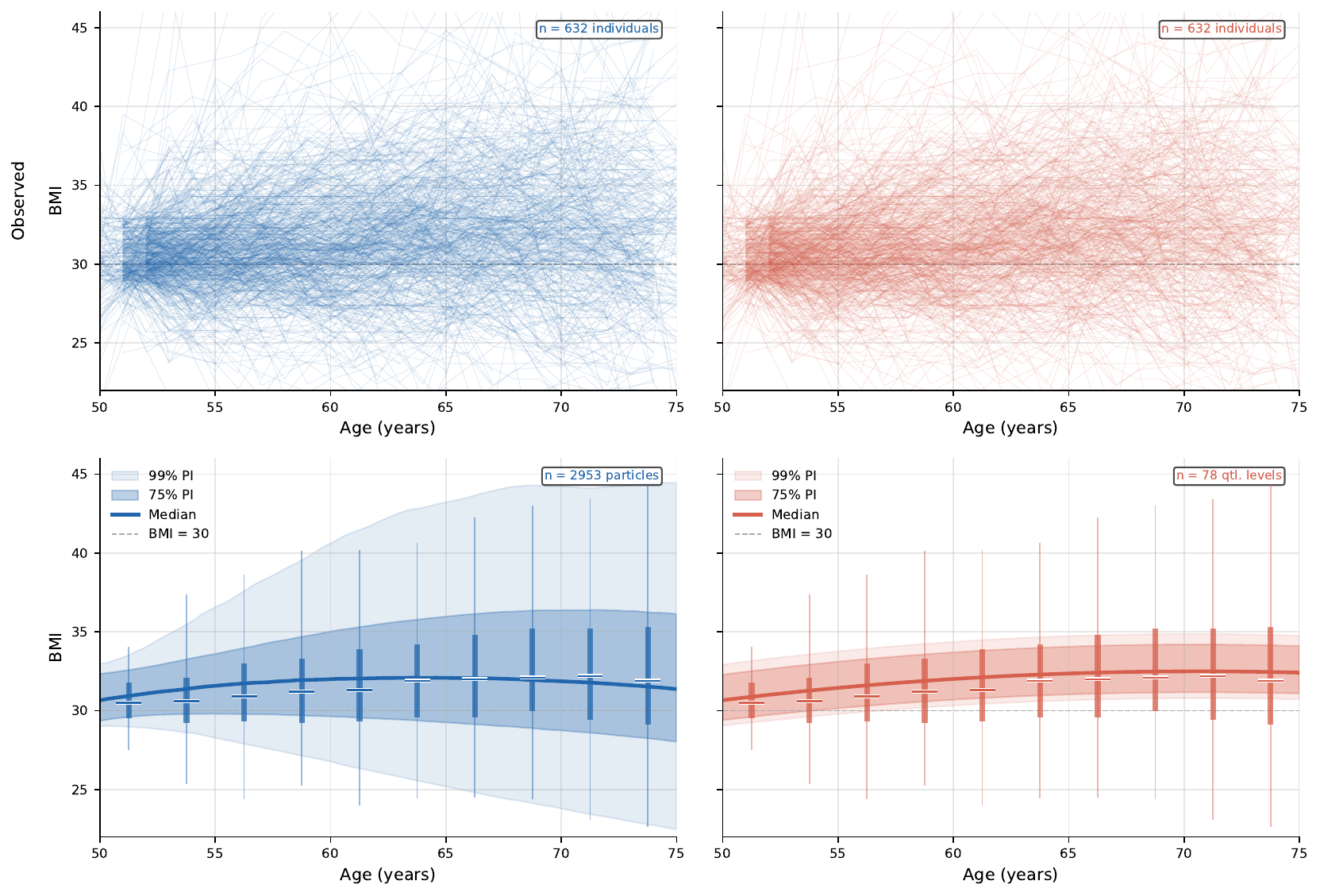}
    \caption{\textbf{Predicted BMI trajectories under single conditioning
    (BMI $\in [29,33]$ at age~50), cohorts 1935--39 and 1940--44, both sexes.}
    The top row shows individual observed BMI trajectories from the matched HRS cohort for Wasserstein least squares (left, blue) and Fr\'echet (right, red).
    The bottom row shows the Wasserstein least squares (left) and Fr\'echet (right) prediction intervals together with empirical box plots of the real cohort's BMI distribution at each age bin.
    Box plots show the interquartile range (thick segment), 2.5th--97.5th percentile whiskers, and median tick; trajectory shown for the 1935--39
    female reference group.
    \textbf{Wasserstein least squares} (left) retains $n=2{,}953$ posterior draws after conditioning, producing smooth 75\% and 99\% prediction intervals that grow as age increases and remain well-centered on the empirical distribution.
    The band reflects genuine uncertainty: individuals with BMI $\approx 31$ at age~50 may follow very different trajectories, as evidenced by both the width of the intervals and the spread of the real box plots.
    \textbf{Fr\'echet} (right) selects $n=78$ quantile-level coefficient vectors after conditioning, producing prediction intervals that are substantially narrower than those of Wasserstein least squares and that drift above the empirical box plots at older ages.
    This reflects the near-deterministic nature of the individual quantile trajectories under pointwise quantile regression: once a quantile level is pinned at age~50, its future path is largely pre-determined, understating the true variability in long-run BMI outcomes.}
    \label{fig:conditioning_wls_vs_frechet}
\end{figure}

Building on the single conditioning results, we extend our empirical evaluation to sequential conditioning. We assess how well both models capture the observed trajectories of individuals measured at both age 50 ($\text{BMI} \in [29, 33]$) and age 60 across three distinct clinical scenarios: worsening ($\text{BMI} > 33$), stable ($\text{BMI} \in [29, 33]$), and improved ($\text{BMI} < 29$). Table \ref{tab:coverage_double} details the empirical coverage for these subgroups. While Wasserstein least squares exhibits some under-coverage relative to nominal levels, it consistently maintains higher coverage than the Fréchet model.
\begin{table}[t]
  \centering
  \resizebox{\textwidth}{!}{%
  \begin{tabular}{lllcccr}
    \toprule
    Scenario & Method & PI level & $n_\text{ptcl}$ & $n_\text{real}$ & Mean emp.\ coverage & Gap \\
    \midrule
    BMI $> 33$ at 60 & Wasserstein least squares & 99\% & 1029 & 204 & 0.692 & -0.298 \\
     & Wasserstein least squares & 75\% & 1029 & 204 & 0.417 & -0.333 \\
     & Fréchet & 99\% & 21 & 204 & 0.173 & -0.817 \\
     & Fréchet & 75\% & 21 & 204 & 0.138 & -0.612 \\
    \addlinespace[2pt]
    BMI $29$--$33$ at 60 & Wasserstein least squares & 99\% & 1729 & 333 & 0.709 & -0.281 \\
     & Wasserstein least squares & 75\% & 1729 & 333 & 0.489 & -0.261 \\
     & Fréchet & 99\% & 57 & 333 & 0.445 & -0.545 \\
     & Fréchet & 75\% & 57 & 333 & 0.351 & -0.399 \\
    \addlinespace[2pt]
    BMI $< 29$ at 60 & Wasserstein least squares & 99\% & 195 & 169 & 0.491 & -0.499 \\
     & Wasserstein least squares & 75\% & 195 & 169 & 0.244 & -0.506 \\
    \bottomrule
  \end{tabular}}
  \caption{\textbf{Empirical PI coverage for double conditioning} (BMI $\in [29,33]$ at age~50, three age-60 scenarios, cohorts 1935--39 and 1940--44, both sexes; trajectory shown for the 1935--39 female reference group). Fr\'echet quantile-level coefficients are shown only when at least 15 levels survive the conditioning step; in the improved scenario, no Fréchet levels are eligible (all trajectories predict BMI $\geq 29$ at age~60) and are therefore omitted. Wasserstein least squares under-covers relative to its nominal level because the prediction band captures uncertainty about the population-level template $Q^\star$, not individual random variability.}
  \label{tab:coverage_double}
\end{table}

For the worsening and stable scenarios, Fréchet's coverage drops significantly, with its 99\% PIs capturing only 17.3\% and 44.5\% of the empirical data, respectively. This large negative gap indicates that a small subset of deterministic quantile trajectories struggles to encompass the natural variance of observed longitudinal outcomes. Furthermore, for the improved scenario, Fréchet retains zero quantile levels, meaning coverage cannot even be calculated.

These coverage metrics translate directly to the forecasted trajectories shown in Figure \ref{fig:double_conditioning}. The top row demonstrates that Wasserstein least squares successfully differentiates the three conditional pathways. The resulting prediction intervals adapt logically to the second observation; whether the trajectory worsens, remains stable, or improves, the Wasserstein least squares intervals remain broad enough to generally bracket the empirical box plots of the real HRS cohort. For the improved scenario, Wasserstein least squares does produce a median trajectory that descends below the obesity threshold, but its predicted path is monotonically decreasing, failing to capture the convexity visible in the empirical trajectories, where BMI typically bottoms out and then rises again at older ages; an observation actually supported by our model (see \cref{fig:extreme_age2} panel (c), where the sampled trajectories with $\beta_{\text{age}^2}>0$ predict a behavior similar to the improved case). This limitation is at least partly a data artifact: observations become sparser at older ages as cohort members attrit from the study, so the linear model is effectively pulled toward the more data-rich middle-age trajectories, biasing the predicted trend at the tail of the follow-up window.

The bottom row of Figure \ref{fig:double_conditioning} illustrates the structural constraints of applying sequential conditioning within the pointwise Fréchet framework. For both the worsening and stable scenarios, the surviving quantile levels yield prediction bands that are substantially narrower than those of Wasserstein least squares, often failing to capture the spread of the empirical data. For the improved scenario, the model cannot generate a prediction at all. Because every Fréchet trajectory that places a patient near the obesity threshold at age 50 strictly predicts an elevated BMI at age 60, the framework mathematically precludes the empirical reality of a patient successfully reducing their BMI over that specific decade.

\begin{figure}[t]
    \centering
    \includegraphics[width=\textwidth]{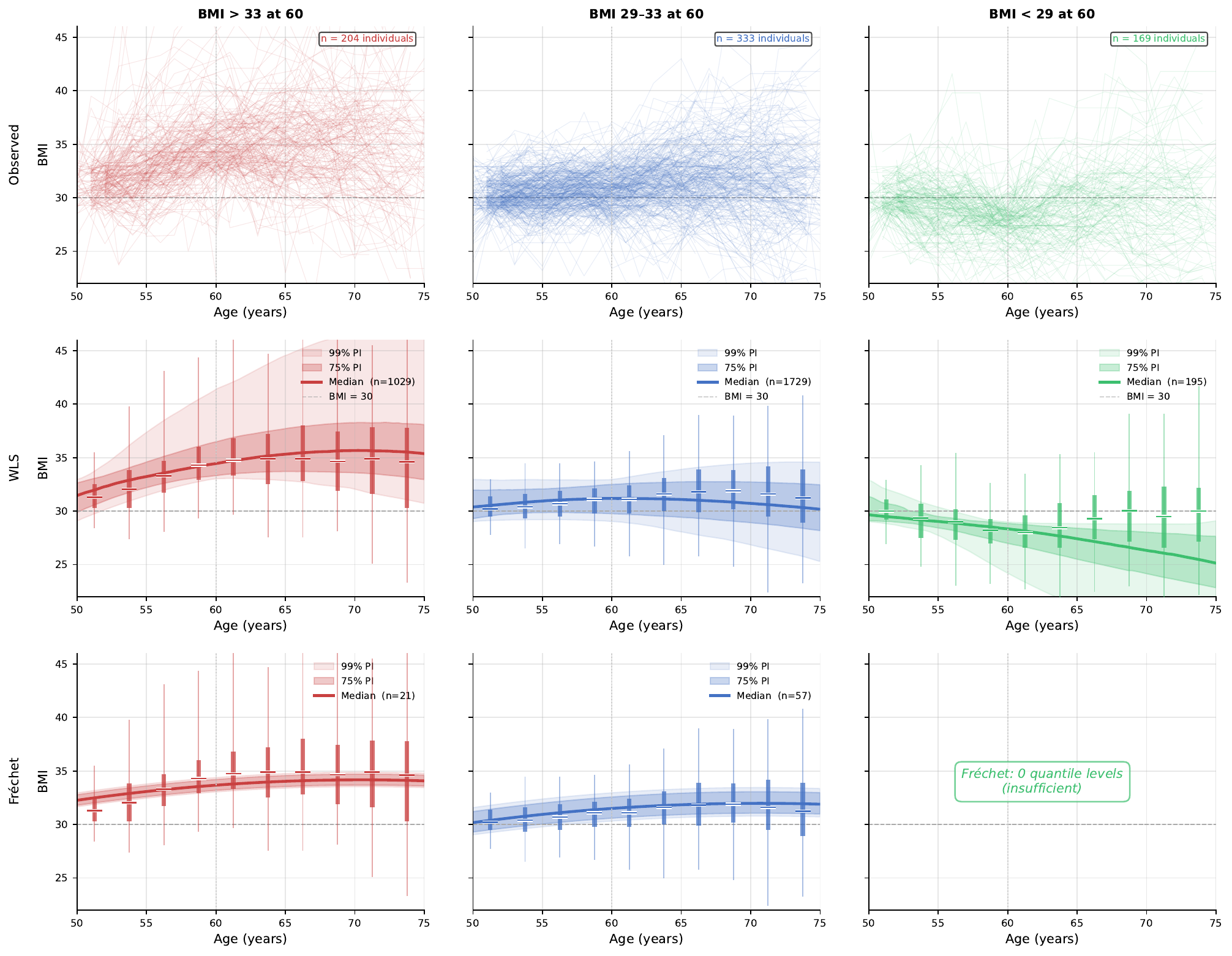}
    \caption{\textbf{Predicted BMI trajectories under sequential conditioning
    (BMI $\in [29,33]$ at age~50 followed by three scenarios at age~60),
    cohorts 1935--39 and 1940--44, both sexes.}
    Columns correspond to three second-observation scenarios:
    worsening (BMI $> 33$ at 60, red),
    stable (BMI $29$--$33$ at 60, blue), and
    improved (BMI $< 29$ at 60, green).
    The top row shows individual observed trajectories from the matched HRS cohort;
    the middle row shows Wasserstein least squares predictions; the bottom row shows Fr\'echet.
    Shaded bands are 75\% (darker) and 99\% (lighter) prediction intervals;
    the solid line is the predicted median.
    Thin box plots show the empirical BMI distribution of matched real
    individuals at each age bin (IQR as thick segment, 2.5th--97.5th
    percentile whiskers).
    Vertical dotted lines mark the two conditioning ages (50 and 60).
    Trajectory shown for the 1935--39 female reference group.
    \textbf{Wasserstein least squares} successfully differentiates all three scenarios.
    The worsening trajectory maintains a median BMI well above 30
    throughout follow-up ($n=1{,}029$ posterior draws);
    the stable scenario shows a modestly elevated trajectory ($n=1{,}729$);
    the improved scenario produces a median that crosses below 30
    within a few years and remains low, though with wide prediction
    intervals reflecting the small retained sample ($n=195$).
    \textbf{Fr\'echet} produces meaningful results only for the worsening
    ($n=21$ quantile levels) and stable ($n=57$) scenarios.
    Both intervals are narrower than their Wasserstein least squares counterparts, underscoring the structural limitations of pointwise quantile regression under sequential conditioning.
    The improved scenario retains zero quantile levels---every
    quantile-level trajectory that places BMI $\approx 31$ at age~50 also predicts BMI $\geq 29$ at age~60, making the improvement condition structurally empty under the Fr\'echet model.}
    \label{fig:double_conditioning}
\end{figure}

\FloatBarrier
\section{Synthetic data experiments}\label{appendix:synthetic}
\subsubsection*{One-Dimensional Responses}
\label{sec:sim-1d}
We validate the Wasserstein least squares estimator under the template deformation model
\eqref{eq:model} with scalar responses ($d=1$) and a two-dimensional
covariate $$\bm{x}_i = (1, t_i)^\top \in \mathbb{R}^2.$$
The true template is chosen as
$$
  Q^\star \;=\; \mathcal{N}(m_Q,\, \Sigma_Q)
  \;\text{ on }\;\mathbb{R}^2,
  \qquad
  m_Q = \begin{pmatrix}0\\1\end{pmatrix},
  \quad
  \Sigma_Q = I_2,
$$
so that the marginal template at covariate $\bm{x} = (1,t)^\top$ is
\begin{equation}\label{eq:sim-marginal}
  Q^\star_{\bm{x}} \;=\; \mathcal{N}\!\bigl(t,\; 1+t^2\bigr),
  \qquad t \in [-2,2].
\end{equation}
The variance $\sigma^2(t) = 1+t^2$ is U-shaped in $t$.
This is a deliberate stress test: the marginal variance grows quadratically with
$|t|$, a non-linear feature that Fréchet regression on quantile functions
cannot capture (\cref{fig:heatmap_sinusoidal}).
\begin{figure}[ht]
    \centering
    \includegraphics[width=0.55\linewidth]{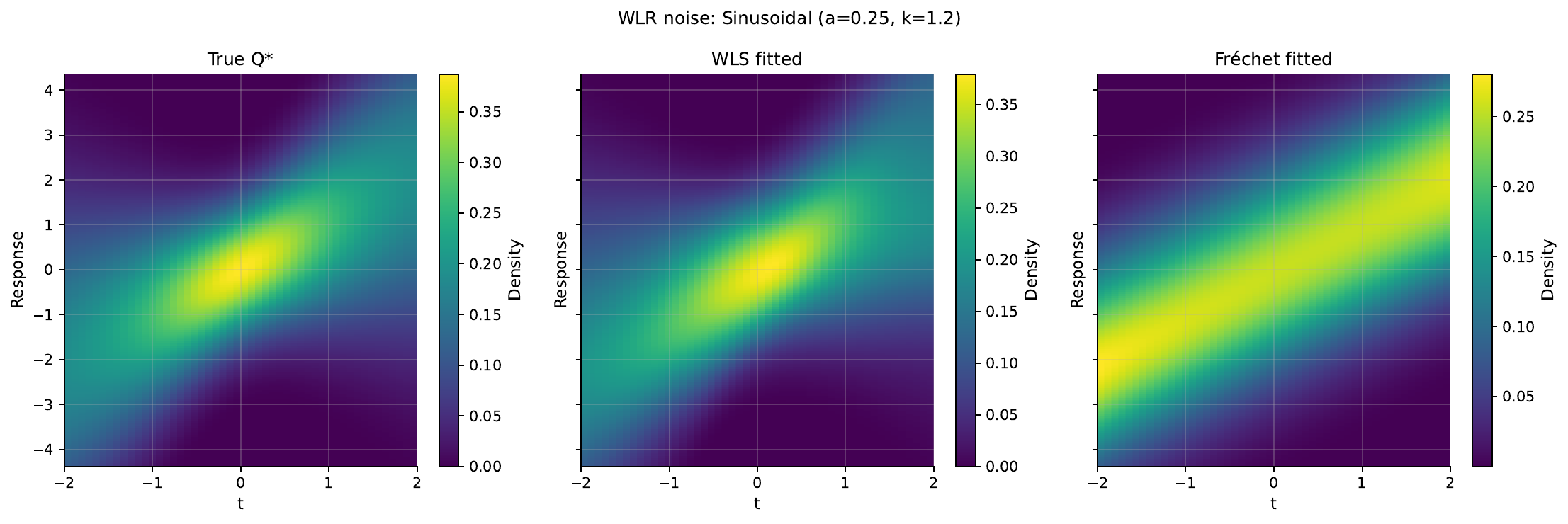}
    \caption{Template recovery under sinusoidal deformation ($k=1.2$, $n=50$).
    Density heatmaps of the true $Q^\star_{\bm{x}}$ (left), Wasserstein least squares fit (center), and Fréchet fit (right) over the covariate range $t\in[-2,2]$.
    Wasserstein least squares reproduces the U-shaped variance of the true template; Fréchet produces nearly uniform spread across all $t$.}
    \label{fig:heatmap_sinusoidal}
\end{figure}
Covariates $t_i$ are drawn uniformly from $[-2,2]$; each response
$\nu_i = (\nabla\phi_i)_\# Q^\star_{\bm{x}_i}$ is approximated by an
empirical measure of $m = 500$ i.i.d.\ draws.

We consider five families of random transport maps $\nabla\phi_i$, each
satisfying conditions C1--C3 (see \cref{tab:noise-models-1d}).
\begin{table}[ht]
\centering
\small
\resizebox{\linewidth}{!}{
\begin{tabular}{llccc}
\toprule
Name & Transport map $T(y)$ & Parameters & $\alpha$ & $\beta$ \\
\midrule
\multicolumn{5}{l}{\textit{Affine / near-affine models}} \\[2pt]
Additive
  & $y + \varepsilon$,\quad $\varepsilon\sim\mathcal{N}(0,\sigma^2)$
  & $\sigma=0.3$
  & $1.00$ & $1.00$ \\
Radial scaling
  & $(1+\eta)\,y$,\quad $\eta\sim\mathcal{U}(-a,a)$
  & $a=0.3$
  & $0.70$ & $1.30$ \\
Location-scale
  & $ay+b$,\quad $a\sim\mathcal{N}(1,\sigma_s^2)_{>0}$,\;$b\sim\mathcal{N}(0,\sigma_s^2)$
  & $\sigma_s=0.2$
  & $0.40$ & $1.60$ \\[6pt]
\multicolumn{5}{l}{\textit{Non-linear models}} \\[2pt]
Sinusoidal ($k=1.2$)
  & $y + A\sin(ky)$,\quad $A\sim\mathcal{U}(-A_{\max},A_{\max})$
  & $k=1.2,\;A_{\max}=0.25$
  & $0.70$ & $1.30$ \\
Sinusoidal ($k=2.5$)$^\dagger$
  & $y + A\sin(ky)$,\quad $A\sim\mathcal{U}(-A_{\max},A_{\max})$
  & $k=2.5,\;A_{\max}=0.30$
  & $0.25$ & $1.75$ \\
Gaussian Bump$^\dagger$
  & $y + Aye^{-y^2/(2\sigma_b^2)}$,\quad $A\sim\mathcal{U}(-A_{\max},A_{\max})$
  & $A_{\max}=0.8,\;\sigma_b=1.0$
  & $0.20$ & $1.80$ \\
Tanh warp
  & $y + A\tanh(ky)$,\quad $A\sim\mathcal{U}(-A_{\max},A_{\max})$
  & $k=0.8,\;A_{\max}=0.4$
  & $0.68$ & $1.32$ \\
\bottomrule
\end{tabular}
}
\caption{
  \textbf{Noise models for the 1-D template deformation experiment.}
  Each model defines a random transport map $T(y)=\nabla\phi_i(y)$ satisfying conditions C1--C3.
  The curvature band $[\alpha,\beta]$ gives the almost-sure range of $T'(y) = \nabla^2\phi_i(y)$, i.e.\ the C3 bounds on the slope of the transport map.
  Models marked $\dagger$ appear in Panel~(a) of \cref{fig:wlr-superfig} only; the remaining five are used in the comparison of \cref{tab:wlr-1d-comparison}.
}
\label{tab:noise-models-1d}
\end{table}
The first three are affine; the last two are non-linear deformations that create non-Gaussian responses even when applied to a Gaussian template.
\Cref{fig:wlr-superfig}(a) illustrates 18 independent draws of $T(y)$ and the deviation $T(y)-y$ for each noise model.

We compare two estimators: Wasserstein least squares, estimated with the particle-based gradient descent estimator $\widehat{Q}$ defined in Algorithm~\ref{alg:wls-gd-particle}, implemented with $M=1{,}000$ particles and $1{,}500$ gradient steps (learning rate $\tau_0 = 0.1$, exponential decay $10^{-3}$, momentum $0.9$, mini-batch size $5$) for the visual panels of \cref{fig:wlr-superfig}, and $M=2{,}000$ particles and $3{,}000$ steps for the averaged errors in \cref{tab:wlr-1d-comparison}; predictions $\widehat{Q}_{\bm{x}}$ are formed via the $\bm{x}$-linear pushforward of the particle system.
We also compare against Global Fréchet regression \citep{PetersenMuller2019}, computed with OLS on the quantile functions, $\widehat{\beta} = (X^\top X)^{-1}X^\top \mathcal{Q}$, where $\mathcal{Q}$ is the $n \times K$ matrix of quantile functions evaluated on a grid of $K=200$ quantile levels; the fitted quantile $\widehat{Q}_{\bm{x}}$ is the $\ell^2$-projection on monotone functions (PAVA).

\begin{figure}[!h]
  \centering
  \includegraphics[width=\textwidth]{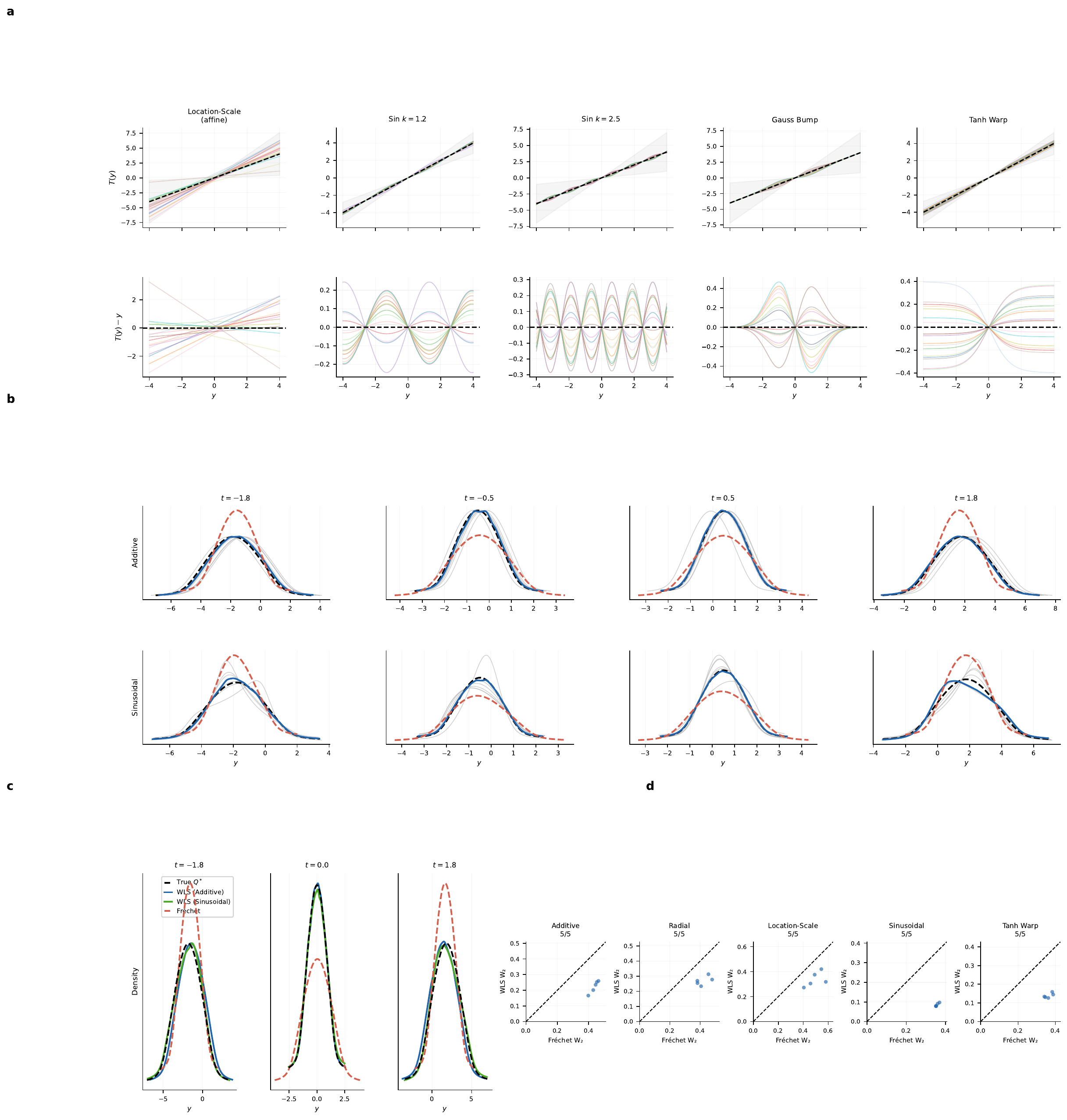}
   \caption{
    \textbf{Wasserstein least squares recovers $Q^\star$ under affine and non-linear noise;
    Fréchet regression is structurally misspecified.}
    DGP: $\nu_i = (\nabla\phi_i)_\# Q^\star_{\bx_i}$,\;
    $Q^\star_{\bx_i} \sim \mathcal{N}(t,\,1{+}t^2)$,\;
    $n=50$, five noise models (see \cref{tab:noise-models-1d}).
    \textbf{(a)}~$18$ draws of $T(y)=\nabla\phi_i(y)$ per model
    (coloured), identity (black dashed), C3 curvature band (grey fill); bottom row shows the deviation $T(y)-y$.
    \textbf{(b)}~Densities at $t\in\{-1.8,-0.5,0.5,1.8\}$ for Additive and Sinusoidal ($k{=}1.2$) noise:
    true $Q^\star_{\bx_i}$ (black dashed), noisy $\nu_i$ (grey), Wasserstein least squares (blue), Fréchet (red dashed).
    Wasserstein least squares spread widens with $|t|$; Fréchet width is flat.
    \textbf{(c)}~$Q^\star$ recovery at $t\in\{-1.8,0,1.8\}$:
    Wasserstein least squares under Additive (blue) and Sinusoidal (green) noise tracks the truth;
    Fréchet (red dashed) predicts constant spread at all $t$.
    \textbf{(d)}~Per-observation $W_2$ scatter (Wasserstein least squares vs.\ Fréchet, in-sample);
    points below the diagonal favor Wasserstein least squares.
    See \cref{tab:wlr-1d-comparison} for numerical summaries.
  }
  \label{fig:wlr-superfig}
\end{figure}
\Cref{fig:wlr-superfig}(b) displays, at four covariate values $t \in \{-1.8,-0.5,0.5,1.8\}$, the true marginal $Q^\star_{\bx_i}$ (black dashed), six independent noisy realisations $\nu_i$ (grey), the Wasserstein least squares fit (blue), and the Fréchet fit (red dashed), for the additive and sinusoidal noise models. Several features are apparent.
First, the Wasserstein least squares fit closely tracks the true density at all four covariate values and under both noise models: the estimated spread widens correctly as $|t|$ increases. Second, the Fréchet fit consistently under-estimates the spread at $|t| = 1.8$ and over-estimates it near $t = 0$; its width is nearly constant across the four panels, a direct consequence of the linear quantile model. Third, Wasserstein least squares is robust to the type of noise: the blue curves under additive and sinusoidal noise are virtually indistinguishable, confirming that the estimator does not require knowledge of the noise distribution.

\Cref{fig:wlr-superfig}(c) focuses on template recovery.
At $t \in \{-1.8, 0.0, 1.8\}$, we plot the true $Q^\star_{\bm{x}}$ alongside the Wasserstein least squares fits under additive (blue) and sinusoidal (green) noise, and the Fréchet fit (red dashed, common to both noise models).
Both Wasserstein least squares curves agree closely with the truth, while the Fréchet curve exhibits the predicted mis-calibration: it produces roughly the same spread at all three $t$ values, matching neither the narrow distribution at $t=0$ nor the wide distributions at $t = \pm 1.8$. \Cref{tab:wlr-1d-comparison} reports the average $W_2$ error against $\{\nu_i\}_{1\leq i\leq 50}$ and the error against the true template $Q^\star$, averaged over $n_{\rm rep}=5$ independent replicates of size $n=50$.

For the error versus $Q^\star$, the Wasserstein least squares estimator achieves values between $0.039$ and $0.167$ across all five noise models, while Fréchet regression ranges from $0.343$ to $0.364$.
The Wasserstein least squares improvement is roughly $5\times$ for mild (sinusoidal, additive) noise and $2\times$ for the more challenging location-scale model.
\Cref{fig:wlr-superfig}(d) provides a per-observation view: each scatter plot places the Fréchet $W_2$ error on the $x$-axis and the Wasserstein least squares $W_2$ error on the $y$-axis for the same observation; points below the diagonal correspond to observations where Wasserstein least squares achieves a smaller error. In all five noise models, all the points fall below the diagonal.

\begin{table}[ht]
\centering
\caption{
  \textbf{1-D template deformation: Wasserstein least squares vs.\ Fréchet regression.}
  Average $W_2$ error against $\{\nu_i\}_{1\leq i\leq 50}$ and error versus the true template $Q^\star$,
  mean $\pm$ std over 5 replicates ($n=50$ observations each);
  Wasserstein least squares fitted with $M=2{,}000$ particles and $T=3{,}000$ steps.
  All Wasserstein least squares errors are strictly smaller than the corresponding Fréchet errors.
}
\label{tab:wlr-1d-comparison}
\small

\resizebox{\linewidth}{!}{
\begin{tabular}{lcccccc}
\toprule
& \multicolumn{2}{c}{$[\alpha,\beta]$}
& \multicolumn{2}{c}{Average $W_2$ (vs $\nu_i$)}
& \multicolumn{2}{c}{Vs.\ true $Q^\star$ ($W_2$)} \\
\cmidrule(lr){2-3}\cmidrule(lr){4-5}\cmidrule(lr){6-7}
Noise model & $\alpha$ & $\beta$
  & Wasserstein least squares & Fréchet & Wasserstein least squares & Fréchet \\
\midrule
Additive ($\sigma=0.3$)         & 1.00 & 1.00
  & $0.223 \pm 0.035$ & $0.437 \pm 0.023$
  & $0.071 \pm 0.026$ & $0.356 \pm 0.012$ \\
Radial ($a=0.3$)                & 0.70 & 1.30
  & $0.270 \pm 0.026$ & $0.420 \pm 0.039$
  & $0.143 \pm 0.055$ & $0.356 \pm 0.009$ \\
Location-Scale ($\sigma_s=0.2$) & 0.40 & 1.60
  & $0.338 \pm 0.053$ & $0.496 \pm 0.062$
  & $0.167 \pm 0.065$ & $0.364 \pm 0.012$ \\
Sinusoidal ($k=1.2$)            & 0.70 & 1.30
  & $0.085 \pm 0.008$ & $0.357 \pm 0.007$
  & $0.039 \pm 0.011$ & $0.343 \pm 0.001$ \\
Tanh Warp ($k=0.8$)             & 0.68 & 1.32
  & $0.139 \pm 0.012$ & $0.364 \pm 0.019$
  & $0.067 \pm 0.024$ & $0.345 \pm 0.003$ \\
\bottomrule
\end{tabular}
}
\end{table}

\FloatBarrier
\subsubsection*{Bivariate Gaussian Responses}
\label{sec:sim-gaussian}

\begin{figure}[!h]
  \centering
  \includegraphics[width=.9\textwidth]{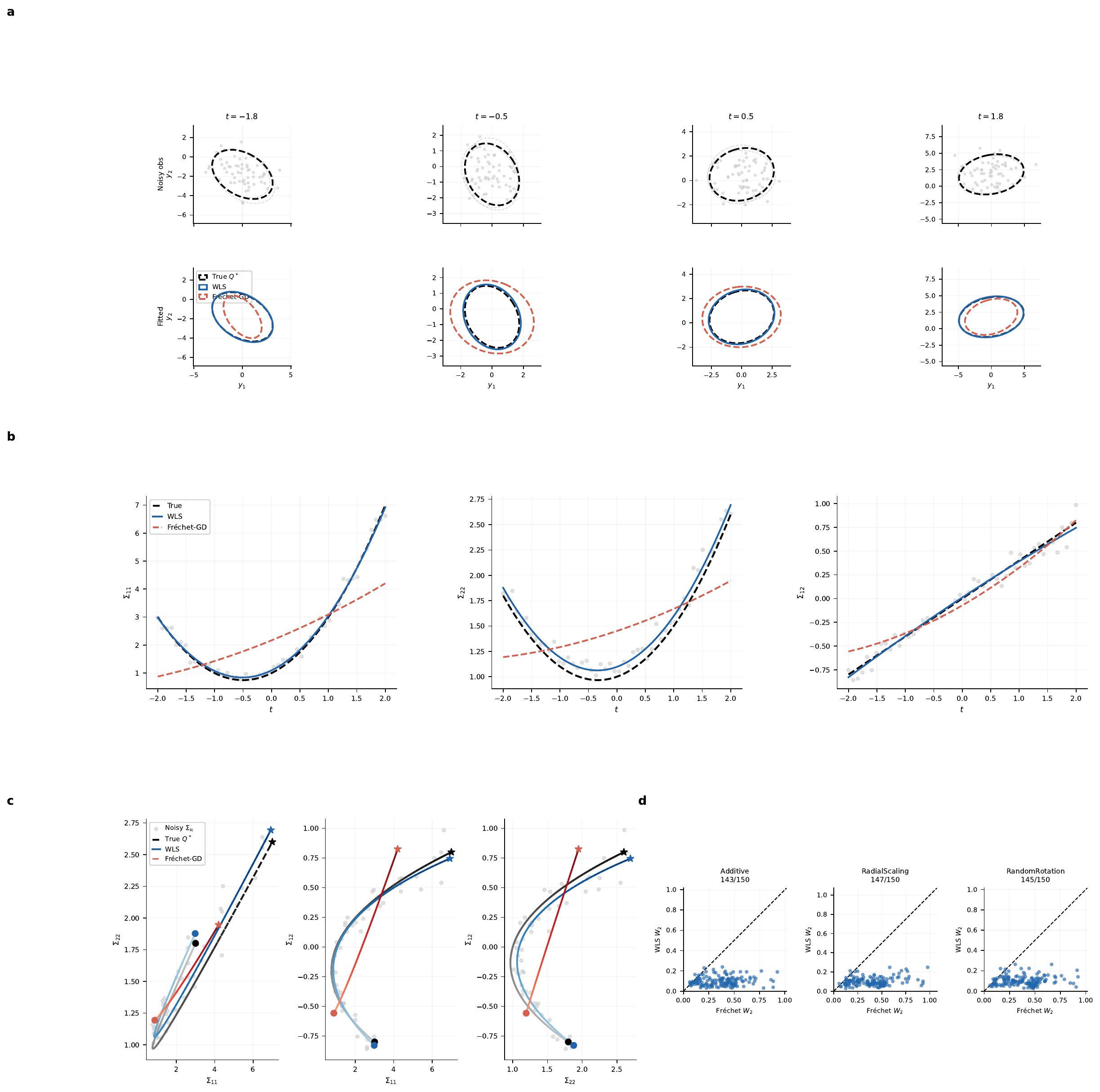}
  \caption{
    \textbf{Wasserstein least squares recovers a quadratic covariance trajectory on the SPD manifold;
    Fréchet-GD is structurally misspecified.}
    DGP: $\nu_i = (\nabla\phi_i)_\# Q^\star_{x_i}$,\;
    $Q^\star_{x_i} \sim \mathcal{N}(\mu(t),\,\Sigma(t))$ with
    $\Sigma(t) = A + t(B{+}B^\top) + t^2 C$ (quadratic on the SPD manifold),\; $n=50$, three noise models (see \cref{tab:noise-models-gaussian}).
    \textbf{(a)}~$2\sigma$ ellipses at $t\in\{-1.8,-0.5,0.5,1.8\}$ under
    Additive noise.
    \emph{Top row}: noisy sample clouds and noisy $\nu_i$ ellipses (grey).
    \emph{Bottom row}: true $Q^\star_{x_i}$ (black dashed), Wasserstein least squares (blue), Fréchet-GD (red dashed).
    Wasserstein least squares ellipses track the evolving shape; Fréchet-GD ellipses are too small at extreme $t$.
    \textbf{(b)}~Covariance entries $\Sigma_{11}$, $\Sigma_{22}$, $\Sigma_{12}$ vs $t$: true curve (black dashed), Wasserstein least squares (blue), Fréchet-GD (red dashed),  noisy observations (grey dots).
    Wasserstein least squares recovers the quadratic growth; Fréchet-GD recovers only the linear part.
    \textbf{(c)}~Trajectory $t\mapsto(\Sigma_{11},\Sigma_{22},\Sigma_{12})$ projected onto three coordinate pairs in the space of SPD entries.
    Colour gradient encodes $t\in[-2,2]$ (dark = $t{=}{-2}$, light = $t{=}2$);
    circles mark $t{=}{-2}$, stars mark $t{=}2$.
    Grey dots: noisy $\Sigma_{\nu_i}$.
    The true trajectory (black) is a quadratic arc; Wasserstein least squares (blue) tracks it faithfully, while Fréchet-GD (red) collapses toward the training mean and cannot reproduce the nonlinear curvature.
    \textbf{(d)}~Per-observation BW scatter (Wasserstein least squares vs.\ Fréchet-GD, in-sample);
    all points lie below the diagonal.
    See \cref{tab:wlr-gaussian-comparison} for numerical summaries.
  }
  \label{fig:wlr-gaussian-superfig}
\end{figure}

We extend the simulation to bivariate responses ($d=2$) with a
two-dimensional covariate $\bm{x}_i = (1,t_i)^\top \in \mathbb{R}^2$.
This setting exercises the full multivariate structure of
model~\eqref{eq:model}: the template $Q^\star \in \mathcal{P}(\mathbb{R}^{dp})$ with $p=d=2$ must capture how a $2\times 2$ random covariance matrix varies with the scalar predictor $t$.

We take $Q^\star = \mathcal{N}(m_Q,\Sigma_Q)$ on $\mathbb{R}^4$, with
\[
  m_Q = (0,0,0,1)^\top,
  \qquad
  \Sigma_Q = \begin{pmatrix} A & B \\ B^\top & C \end{pmatrix},
\]
where
\[
  A = I_2,\quad
  B = \begin{pmatrix}0.5&0.2\\0.2&0.1\end{pmatrix},\quad
  C = \begin{pmatrix}1.0&0\\0&0.3\end{pmatrix}.
\]
Via the Kronecker marginalisation in \cref{prop:covariance_gaussian_marginals}, the template at $\bm{x}=(1,t)^\top$ is
\begin{equation}\label{eq:sim-gaussian-marginal}
  Q^\star_{\bm{x}} \;=\; \mathcal{N}\!\bigl(\mu(t),\; \Sigma(t)\bigr),
  \qquad
  \mu(t) = \begin{pmatrix}0\\t\end{pmatrix},
  \quad
  \Sigma(t) = A + t(B+B^\top) + t^2 C.
\end{equation}
The mean $\mu(t)$ shifts linearly in $t$, while the covariance $\Sigma(t)$ traces a \emph{quadratic} trajectory on the cone of symmetric positive definite matrices.
At $t=0$ the marginal is circular ($A = I_2$); as $|t|$ grows, the ellipse elongates and rotates.
This is again a stress test for Fréchet regression.

We use three families of random maps $\nabla\phi_i : \mathbb{R}^2 \to \mathbb{R}^2$,
each satisfying C1--C3:

\begin{table}[ht]
\centering
\small
\resizebox{\linewidth}{!}{
\begin{tabular}{llc}
\toprule
Name & Transport map $T(\bm{y})$ & Parameters \\
\midrule
Additive
  & $\bm{y} + \bm{\varepsilon}$,\quad $\bm{\varepsilon}\sim\mathcal{N}(\bm{0},\sigma^2 I_2)$
  & $\sigma = 0.3$ \\
Radial scaling
  & $(1+\eta)\,\bm{y}$,\quad $\eta\sim\mathcal{U}(-a,a)$
  & $a = 0.3$ \\[6pt]

Random rotation--scale
  & $R^\top D R\,\bm{y}$,\quad $R\sim\mathrm{SO}(2)$, $D=\mathrm{diag}(s_1,s_2)$
  & $\theta\sim\mathcal{N}(0,0.3^2)$,\;$s_k\sim\mathcal{U}(0.8,1.2)$ \\
\bottomrule
\end{tabular}
}
\caption{
  \textbf{Noise models for the bivariate Gaussian experiment.}
  All maps satisfy C1 (convex, $C^2$), C2 ($\mathbb{E}[T(\bm{y})]=\bm{y}$), and C3 (uniform curvature bounds).
  Each $\nu_i$ is approximated by the sample mean and covariance of $m=500$ i.i.d.\ draws from $(\nabla\phi_i)_\# Q^\star_{\bm{x}_i}$.
}
\label{tab:noise-models-gaussian}
\end{table}

We compare two estimators:
\begin{enumerate}
  \item {Wasserstein least squares} (Bures--Wasserstein gradient descent): the closed-form
        Gaussian specialization of the Wasserstein least squares estimator, implemented with \cref{alg:bw_gd}.
        Fitted with $K=300$ gradient steps from a random initialization.
  \item {Fréchet-GD}: the global Fréchet regression estimator of
        \cite{PetersenMuller2019}, implemented for covariance matrices via iterative descent (Algorithm~1 in  \cite{bures-frechet}) with localization parameter
        $\rho=0.01$, step size $\eta=0.5$, and at most $200$ iterations
        per prediction point.
\end{enumerate}

\Cref{fig:wlr-gaussian-superfig}(a) shows $2\sigma$ ellipses at four covariate values $t \in \{-1.8,-0.5,0.5,1.8\}$ under Additive noise.
The top row displays the noisy sample clouds and three independent realizations of $\nu_i$ (grey ellipses), confirming that the individual responses are substantially noisier than the template.
The bottom row overlays the fitted ellipses.
The Wasserstein least squares ellipses (blue) closely match the true template (black dashed) at all four values of $t$: the elongation and orientation of the ellipse evolve correctly.
The Fréchet-GD ellipses (red dashed) are systematically too small at $|t| = 1.8$, where the quadratic term $t^2 C$ dominates; they recover the shape only near $t=0$.

\Cref{fig:wlr-gaussian-superfig}(b) plots the three independent entries of $\Sigma(t)$ against $t$ on a fine grid.
The true curves (black dashed) are quadratic in $t$; the Wasserstein least squares curves (blue) track them closely, while the Fréchet-GD curves (red) approximate only the linear part and underestimate variance at the extremes.
The structural origin of this limitation is the same as in the one-dimensional case: Fréchet regression's linear weighting scheme cannot produce predictions outside the convex hull of training responses.

\Cref{fig:wlr-gaussian-superfig}(c) visualises the trajectory
$t\mapsto\Sigma(t)$ as a curve in the three-dimensional space of SPD entries $(\Sigma_{11},\Sigma_{22},\Sigma_{12})$, shown through three coordinate
projections.
The true trajectory (black) sweeps a quadratic arc; the Wasserstein least squares curve (blue) tracks it closely across the full range $t\in[-2,2]$.
The Fréchet-GD curve (red) remains confined near the center of the training cloud, confirming the convex-hull constraint: it cannot extrapolate the quadratic growth that drives the trajectory at $|t|=2$.

\Cref{tab:wlr-gaussian-comparison} reports the mean average $W_2$ error (versus $\nu_i$) and the error against the true template $Q^\star$, averaged over $n_{\rm rep}=5$ replicates of size $n=50$.
Wasserstein least squares achieves BW errors against the true template between $0.026$ and $0.059$, a factor of $6$--$14\times$ smaller than Fréchet-GD ($0.362$--$0.369$).
The ranking is consistent across all three noise models, including the non-linear rotation--scale model.
Panel~(d) provides the per-observation view: across all three noise models and all $n \times n_{\rm rep} = 250$ training observations,  nearly every single point in the scatter falls below the diagonal, meaning Wasserstein least squares achieves a
smaller in-sample BW error than Fréchet-GD on every observation.

\begin{table}[ht]
\centering
\caption{
  \textbf{Bivariate Gaussian template deformation: Wasserstein least squares vs.\ Fréchet-GD.}
  Bures--Wasserstein error averaged over $5$ replicates ($n=50$ each).
  All Wasserstein least squares errors are strictly smaller than Fréchet-GD.
}
\label{tab:wlr-gaussian-comparison}
\small
\resizebox{\linewidth}{!}{
\begin{tabular}{lcccc}
\toprule
& \multicolumn{2}{c}{Average $W_2$ (vs $\nu_i$)}
& \multicolumn{2}{c}{Vs.\ true $Q^\star$ (BW)} \\
\cmidrule(lr){2-3}\cmidrule(lr){4-5}
Noise model & Wasserstein least squares & Fréchet-GD & Wasserstein least squares & Fréchet-GD \\
\midrule
Additive ($\sigma=0.3$)             & $0.102 \pm 0.006$ & $0.371 \pm 0.009$
                                    & $0.059 \pm 0.005$ & $0.369 \pm 0.001$ \\
Radial scaling ($a=0.3$)            & $0.106 \pm 0.008$ & $0.389 \pm 0.007$
                                    & $0.045 \pm 0.005$ & $0.365 \pm 0.001$ \\
Random rotation--scale ($\theta_0=0.3$) & $0.104 \pm 0.006$ & $0.383 \pm 0.005$
                                    & $0.026 \pm 0.005$ & $0.362 \pm 0.001$ \\
\bottomrule
\end{tabular}
}
\end{table}

\FloatBarrier

\subsubsection*{Rate Validation}
\label{sec:sim-rate}

We validate the $n^{-1/2}$ rate of \cref{thm:stat_main} by isolating its $n$-dependent term. Taking $m=\infty$ (responses $\nu_i$ observed exactly as $K=500$-point quantile functions), the bound reduces to
\begin{equation}\label{eq:rate_target}
  \frac{1}{n}\sum_{i=1}^n W_2^2\!\bigl(\widehat{Q}_{\bm{x}_i},\,Q^\star_{\bm{x}_i}\bigr)
  \;\lesssim\;\frac{1}{\sqrt{n}},
\end{equation}
predicting a log--log slope of $-1/2$ in $n$.

We use the Bures--Wasserstein gradient estimator (\cref{alg:bw_gd}) rather than
the particle method (\cref{alg:wls-gd-particle}) so that the empirical rate
reflects only the error sources controlled by \cref{thm:stat_main}: the
statistical term $\sqrt{pd/n}$ and the within-measure term $r_m$ (governed by
$m$). The particle method carries a separate $M$-dependent discretization error
whose joint scaling with $n$ and $m$ is a distinct question that we do not
address here.

We use the template $Q^\star_{\bm{x}} = \mathcal{N}(t, 1+t^2)$ with $\bm{x}=(1,t)^\top$, $t \sim \mathcal{U}[-2,2]$ ($p=2$, $d=1$), and the three affine noise models from \cref{tab:noise-models-1d} (additive, radial, location-scale), which produce exactly Gaussian responses. We run $K=500$ gradient descent steps and $20$ independent realizations at each $n \in \{10,25,50,100,200,500\}$.

\cref{fig:rate-validation} shows the median $W_2^2$ error and interquartile range on log--log axes together with $n^{-1/2}$ and $n^{-1}$ reference lines. All three curves decay monotonically. The location-scale and radial models display a rate closer to $n^{-1}$ over this range.

\begin{figure}[ht]
  \centering
  \includegraphics[width=0.75\textwidth]{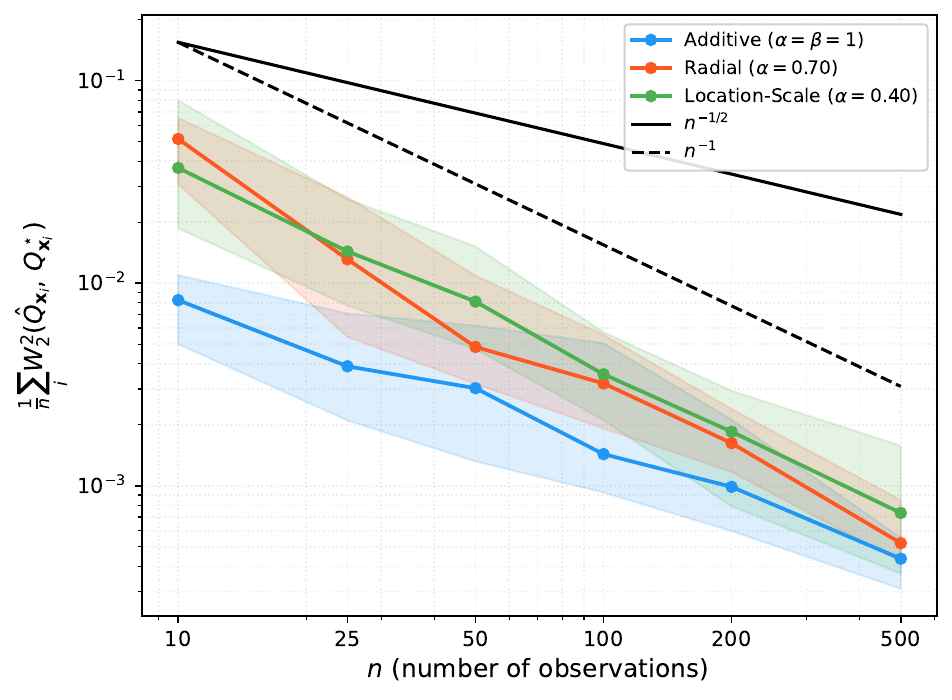}
  \caption{
    \textbf{Rate validation for \cref{thm:stat_main}: in-sample $W_2^2$ error
    vs.\ $n$ in the $m{=}\infty$ regime.}
    Template: $Q^\star_{\bm{x}_i}=\mathcal{N}(t_i,1{+}t_i^2)$, three affine noise models
    (see \cref{tab:noise-models-1d}), $n\in\{10,25,50,100,200,500\}$, $20$ seeds.
    Estimator: Bures--Wasserstein gradient descent ($K{=}500$ steps).
    Solid and dashed lines show $n^{-1/2}$ and $n^{-1}$ reference slopes.
    All three models decay at a rate consistent with the $n^{-1/2}$ bound
    of \cref{thm:stat_main}; location-scale and radial curves display a rate
    closer to $n^{-1}$ over this range.
  }
  \label{fig:rate-validation}
\end{figure}

\FloatBarrier

\section{Computational cost}\label{appendix:timing}

All times are wall-clock measurements on a 2022 MacBook Air.  Fréchet regression is closed-form (OLS + optional PAVA
projection) and takes less than one second in every setting.

\paragraph*{Appendix~\ref{appendix:bmi}: Retirement data (HRS)}
\Cref{tab:timing-bmi} reports the fitting cost for the $n=164$ demographic
cells.  The dominant cost is the main Wasserstein least squares quadratic
fit ($M=20{,}000$ particles, $T=10{,}000$ iterations), which represents the
highest-resolution estimate used for coefficient analysis and visualization.
The LOO-CV fit uses a lighter configuration ($M=2{,}000$, early stopping
with patience $=100$, typically converging in $\sim\!300$ iterations).

\begin{table}[ht]
\centering
\small
\caption{Wall-clock times — Appendix~\ref{appendix:bmi} (HRS, $n=164$ cells,
$K=500$, $p=5$). ``WLS'' = Wasserstein least squares.}
\label{tab:timing-bmi}
\resizebox{\textwidth}{!}{%
\begin{tabular}{lccc}
\toprule
Estimator & Parameters & Per fit & Total \\
\midrule
WLS, quadratic (Particle GD) & $M=20{,}000$,\; $T=10{,}000$ & ${\approx}11$ min & ${\approx}11$ min \\
WLS, LOO-CV, 164 folds (Particle GD) & $M=2{,}000$,\; patience$=100$ & ${\approx}1$ s & ${\approx}3$ min \\
Fréchet, quadratic (OLS + PAVA) & $K=500$ & ${<}1$ s & ${<}1$ s \\
\bottomrule
\end{tabular}}
\end{table}

\paragraph*{Appendix~\ref{appendix:synthetic}: Synthetic experiments}
\Cref{tab:timing-synthetic} gives per-fit and total times for each subsection
of Appendix~\ref{appendix:synthetic}.  ``Total'' refers to the full comparison
run that produced the published tables and figures (number of fits in
parentheses).

\begin{table}[ht]
\centering
\small
\caption{Wall-clock times — Appendix~\ref{appendix:synthetic} (synthetic
experiments). ``WLS'' = Wasserstein least squares.}
\label{tab:timing-synthetic}
\begin{tabular}{llccc}
\toprule
Estimator & Algorithm & Parameters & Per fit & Total \\
\midrule
\multicolumn{5}{l}{\textit{One-dimensional responses
  ($n=50$, $K=200$;\; 5 models $\times$ 5 reps $= 25$ fits)}}\\[2pt]
WLS & Particle GD & $M=2{,}000$,\; $T=3{,}000$ & $5$ s & ${\approx}2$ min \\
Fréchet & OLS + PAVA & $K=200$ & ${<}1$ s & ${<}1$ s \\[4pt]
\multicolumn{5}{l}{\textit{Bivariate Gaussian responses
  ($n=50$, $d=p=2$;\; 3 models $\times$ 5 reps $= 15$ fits)}}\\[2pt]
WLS & BW-GD (\cref{alg:bw_gd}) & $K=300$,\; $dp=4$ & $1.5$ s & $22$ s \\
Fréchet-GD & \citet{bures-frechet} Alg.~1 & iter$=200$ & $0.3$ s & $5$ s \\
Fréchet-OLS & Closed-form & --- & ${<}1$ s & ${<}1$ s \\[4pt]
\multicolumn{5}{l}{\textit{Rate validation
  ($d=1$, $p=2$;\; 3 models $\times$ 6 sizes $\times$ 20 seeds $= 360$ fits)}}\\[2pt]
WLS & BW-GD (\cref{alg:bw_gd}) & $K=500$,\; $dp=2$ & $0.2$--$2.5$ s & ${\approx}8$ min \\
\bottomrule
\end{tabular}
\end{table}

\FloatBarrier

\section{Lifting of Euclidean least squares to Wasserstein least squares}

\begin{proof}[Proof of \cref{thm:lifted_main}]
	We apply the results of \cite{lavenantLiftingFunctionalsDefined2023} to establish~\eqref{eq:mult_lift}.
	In the notation of that paper, the right side of \eqref{eq:mult_lift} is $\cT_E(\nu_1, \dots, \nu_n)$~\cite[Definition 3.6]{lavenantLiftingFunctionalsDefined2023}, with $\EEuc$ defined as in~\eqref{eq:NPR}, $Y = \RR^d$, and $X = [n]$ equipped with the uniform probability measure $m$.
	\citet[Proposition 3.13]{lavenantLiftingFunctionalsDefined2023} shows that the minimum on the right side of~\eqref{eq:mult_lift} is attained, justifying the use of $\min$ in place of $\inf$.

	The functional $\EEuc$ is continuous and, as $|X| < \infty$, it automatically satisfies \cite[Assumption A]{lavenantLiftingFunctionalsDefined2023}.
	Therefore \cite[Theorem 3.9]{lavenantLiftingFunctionalsDefined2023} implies that $\cT_E$ satisfies requirement R1 and is the largest convex and lower-semicontinuous function dominated by
	\begin{equation}
		\widetilde \cT_E(\nu_1, \dots, \nu_n) = \begin{cases}
						\EEuc(\by_1,\dots, \by_n) & \text{if $\nu_i = \delta_{\by_i}$ for $i = 1, \dots, n$} \\
						+\infty & \text{otherwise.}
					\end{cases}
	\end{equation}
	Since $\cE$ as defined in~\eqref{eq:ewas_def} is convex, lower semi-continuous, and dominated by $\widetilde \cT_E$, we obtain that
		\begin{equation*}
		\EWas(\nu_1, \dots, \nu_n) \leq \min_{P \in \Pi(\nu_1, \dots, \nu_n)} \int \EEuc(\by_1, \dots, \by_n) \dd P(\by_1, \dots, \by_n)\,.
	\end{equation*}
	On the other hand, since $\cT_E$ satisfies R1 and R2 \cite[Lemma 3.11 and Proposition 3.13]{lavenantLiftingFunctionalsDefined2023}, the reverse inequality also holds, which yields~\eqref{eq:mult_lift}.

	We now show that this expression agrees with~\eqref{eq:marg_lift}.
	By \cref{lemma:measure-lemma}, we may write
	\begin{multline*}
		 \min_{P \in \Pi(\nu_1, \dots, \nu_n)} \int \min_{f\in\mathcal{F}}\frac{1}{n}\sum_{i=1}^n \|\by_i-f(\bx_i)\|^2_2 \dd P(\by_1, \dots, \by_n) = \\  \min_{P \in \Pi(\nu_1, \dots, \nu_n)} \min_{\substack{\mathsf{P} \in \cP_2((\RR^d)^n \times \cF)\\ (\pi_1)_\# \mathsf P = P}} \int \frac{1}{n}\sum_{i=1}^n \|\by_i-f(\bx_i)\|^2_2 \dd \mathsf P(\by_1, \dots, \by_n, f) =\\
		  \min_{\substack{\mathsf{P} \in \cP_2((\RR^d)^n \times \cF)\\ (\pi_1)_\# \mathsf P \in \Pi(\nu_1, \dots, \nu_n)}} \int \frac{1}{n}\sum_{i=1}^n \|\by_i-f(\bx_i)\|^2_2 \dd \mathsf P(\by_1, \dots, \by_n, f)\,.
	\end{multline*}
	where $\pi_1:(\RR^d)^n \times \cF \to (\RR^d)^n $ denotes the canonical projection.

	Interchanging integration and summation, we obtain
	\begin{align*}
		\EWas(\nu_1, \dots, \nu_n) =  \min_{\substack{\mathsf{P} \in \cP_2((\RR^d)^n \times \cF)\\ (\pi_1)_\# \mathsf P \in \Pi(\nu_1, \dots, \nu_n)}} \frac 1n \sum_{i=1}^n \int \|\by_i-f(\bx_i)\|^2_2 \dd \mathsf P(\by_1, \dots, \by_n, f)\,.
	\end{align*}
	We now observe that the $i$th integral depends only on the marginal distribution of the pair $(\by_i, f)$ under $\mathsf P$, which we denote by $\mathsf P_i$.
	If we write $Q \in \cP(\cF)$ for the distribution of $f$ under $\mathsf P$, then each of the $n$ marginal measures $(\mathsf P_1, \dots, \mathsf P_n)$ is a coupling between $\nu_i$ and $Q$.
	Conversely, given a collection of $n$ couplings $(\mathsf P'_1, \dots, \mathsf P'_n)$ between $\nu_i$ and some element $Q' \in \cP(\cF)$, the gluing lemma~\citep{villani_optimal_transport_2009} implies that we can combine these into a $\mathsf P' \in  \cP_2((\RR^d)^n \times \cF)$ such that $(\pi_1)_\# \mathsf P' \in \Pi(\nu_1, \dots, \nu_n)$.
	We may therefore replace minimization over $\mathsf P$ by minimization over $Q$ and $(\mathsf P_1, \dots, \mathsf P_n)$ to obtain
	\begin{align*}
		\EWas(\nu_1, \dots, \nu_n) & = \min_{Q \in \cP(\cF)} \min_{\substack{\mathsf P_i, i = 1, \dots n\\ \mathsf P_i \in \Pi(\nu_i, Q)}} \frac 1n \sum_{i=1}^n \int \|\by_i-f(\bx_i)\|^2_2 \dd \mathsf P_i(\by_i, f) \\
		& = \min_{Q \in \cP(\cF)}   \frac 1n \sum_{i=1}^n \min_{\substack{\mathsf P_i \in \Pi(\nu_i, Q)}} \int \|\by_i-f(\bx_i)\|^2_2 \dd \mathsf P_i(\by_i, f)  \\
		& = \min_{Q \in \cP(\cF)}   \frac 1n \sum_{i=1}^n W_2^2(\nu_i, Q_{\bx_i})\,,
	\end{align*}
	where the last step uses the fact that $\Pi(\nu_i, Q_{\bx_i}) = (\Id, \delta_{\bx_i})_\# \Pi(\nu_i, Q)$, where $\delta_{\bx_i}$ denotes the evaluation map at $\bx_i$.
\end{proof}

\begin{proof}[Proof of \cref{thm:basic_dual}]
	Denote the right side of~\eqref{eq:dual} by $\cS(\nu_1, \dots, \nu_n)$.
	We first show that $\cS \leq \EWas$.
	Let $\psi_1, \dots, \psi_n \in C(\cF)$ satisfy $\sum_i \psi_i = 0$, and let $P \in \Pi(\nu_1, \dots, \nu_n)$ be arbitrary.
	Since $\sum_i \psi_i =0$, we have
	\begin{align*}
		\frac 1n \sum_{i=1}^n S_i \psi_i(y_i) & = \frac 1n \sum_{i=1}^n \inf_{f \in \cF} \|\by - f(\bx_i)\|^2 - \psi_i(f)\\
		& \leq \inf_{f \in \cF} \frac 1n \sum_{i=1}^n  \|\by_i - f(\bx_i)\|^2 - \sum_i \psi_i(f) = E(\by_1, \dots, \by_n)\,,
	\end{align*}
	so
	\begin{equation*}
		\frac 1n \sum_{i=1}^n \int S_i \psi_i \dd \nu_i = \int \frac 1n  \sum_{i=1}^n S_i \psi_i(\by_i) \dd P(\by_1, \dots, \by_n) \leq  \int E \dd P\,.
	\end{equation*}
	Taking the supremum over $\psi_1, \dots, \psi_n$ and infimum over $P$ shows that $\cS \leq \EWas$.

	We now show that $\cS \geq \EWas.$
	As in the proof of \cref{thm:lifted_main}, we appeal to the results of~\cite{lavenantLiftingFunctionalsDefined2023}: by \cite[Theorem 3.9]{lavenantLiftingFunctionalsDefined2023},
	\begin{multline}\label{eq:hugo_dual}
		\EWas(\nu_1, \dots, \nu_n) = \sup_{\phi_1, \dots, \phi_n \in C_b} \Big\{\frac 1n \sum_{i=1}^n \int \phi_i \dd \nu_i: \\\frac 1n \sum_{i=1}^n \phi_i(\by_i) \leq E(\by_1, \dots, \by_n) \quad \forall \by_i, i = 1, \dots, n\Big\}\,.
	\end{multline}
	Let $\phi_1, \dots, \phi_n$ be feasible for~\eqref{eq:hugo_dual}.
	For $i = 1, \dots, n-1$,  define
	\begin{equation}\label{eq:psi-def}
		\psi_i(f) = \inf_{\by} \|\by - f(\bx_i)\|^2 - \phi_i(\by)\,,
	\end{equation}
	and set $\psi_n(f) = - \sum_{i=1}^{n-1} \psi_i(f)$.
	Then $\psi_i \in C(\cF)$ for $i = 1, \dots, n$, and $\sum_{i=1}^n \psi_i = 0$ by construction.

	We now check, for $i = 1, \dots, n-1$,
	\begin{align*}
		S_i \psi_i(\by_i) & = \inf_{f \in \cF}\|\by_i - f(\bx_i)\|^2 - \psi_i(f) \\
		& = \inf_{f \in \cF} \sup_{\by} \|\by_i - f(\bx_i)\|^2 -\|\by - f(\bx_i)\|^2 + \phi_i(\by) \\
		& \geq \phi_i(\by_i)\,.
	\end{align*}
	and
	\begin{align*}
		S_n \psi_n(\by_n) & = \inf_{f \in \cF} \|\by_n - f(\bx_n)\|^2 + \sum_{i=1}^{n-1} \psi_i(f) \\
		& = \inf_{f, \by_1, \dots, \by_{n-1}} \sum_{i=1}^n \|\by_i - f(\bx_i)\|^2 - \sum_{i=1}^{n-1} \phi_i(\by_i) \\
		& \geq \phi_n(\by_n)\,,
	\end{align*}
	where the final inequality follows from the fact that
	\begin{equation}
		\sum_{i=1}^{n-1} \phi_i(\by_i) \leq - \phi_n(\by_n) +  \sum_{i=1}^n  \|\by_i -f(\bx_i)\|^2
	\end{equation}
	for all $\by_1, \dots, \by_n$ and $f \in \cF$ by assumption.

	Therefore $\frac 1 n \sum_{i=1}^n \int S_i \psi_i \dd \nu_i \geq \frac 1n \sum_{i=1}^n \int \phi_i \dd \nu_i$, and $\cS \geq \EWas$ by~\eqref{eq:hugo_dual}.

\end{proof}

\begin{proof}[Proof of \cref{thm:lin_dual}]
		We follow the approach of~\cite{aguehBarycentersWassersteinSpace2011}.
		The proof of \cref{thm:basic_dual} shows that for $i = 1, \dots, n-1$ we can always consider $\psi_i$ of the form in~\eqref{eq:psi-def}:
		\begin{align*}
					\psi_i(\bB) & = \inf_{\by} \|\by - \bB^\top \bx_i\|^2 - \phi_i(\by) \\
					& = \|\bB^\top \bx_i\|^2 - \sup_{\by} (2 \bx_i^\top \bB \by - \|\by\|^2 + \phi_i(\by))
		\end{align*}
		for some $\phi_i \in C_b$.
		In particular, we can assume that for $i = 1, \dots, n-1$, the function $\bB \mapsto \|\bB^\top \bx_i\|^2 - \psi_i(\bB)$ is convex.
			We may also assume by shifting by appropriate constants that $\psi_i(\bm{0}) = 0$ for all $i$.

			Now let $(\psi^m)_{m \geq 1} := (\psi^m_1, \dots, \psi^m_n)_{m \geq 1}$ be a maximizing sequence for $\cS$ satisfying the above assumptions.
			For each $i \in [n]$ and $m \geq 1$, the assumption that $\psi_i^m(\bm{0}) = 0$ implies
			\begin{equation}\label{eq:s_upper_bound}
					S_i\psi_i^m(\by_i) \leq \|\by_i\|^2\,.
				\end{equation}
			And since $\psi^m$ is a maximizing sequence, there exists a constant $C$ such that
			\begin{equation}
					\frac 1n \sum_{i=1}^n \int S_i \psi_i^m \dd \nu_i \geq C \quad \forall i, m\,.
				\end{equation}
			Since each $\nu_i$ has finite second moment, we obtain the existence of a constant $C'$ such that
			\begin{align*}
				\int S_i \psi_i^m \dd \nu_i & \leq C' \\
				\int S_i \psi_i^m \dd \nu_i & \geq nC - \sum_{j \neq i} \int S_j \psi_j^m \dd \nu_j \\
				& \geq nC - (n-1) C'
			\end{align*}
			for all $i$ and $m$.

			For any $\by_i \in \RR^d$ and $\bB \in \RR^{p \times d}$, we have
			\begin{equation}
					\psi^m_i(\bB) \leq \|\by_i - \bB^\top \bx_i\|^2 - S_i\psi^m_i(\by_i)\,,
				\end{equation}
			and integrating the right side with respect to $\nu_i$ and using the boundedness of $\int S_i\psi^m_i \dd \nu_i$ shows that there exists a constant $C''$ such that the functions $\psi^m_i$ satisfy
			\begin{equation}
				\psi_i^m(\bB) \leq C''(1 + \|\bB\|_F^2) \quad \forall i, m
			\end{equation}
			Since we have assumed that $\psi_n^m = -\sum_{i=1}^{n-1} \psi_i^m$, we also obtain
			\begin{equation}
				\sum_{i=1}^{n-1} \psi_i^m(\bB) = - \psi_n^m(\bB) \geq - C''(1 + \|\bB\|_F^2) \quad \forall m\,,
			\end{equation}
			and therefore for $i = 1, \dots, n-1$
			\begin{align*}
				\psi_i^m(\bB) & \geq - C''(1 + \|\bB\|_F^2) - \sum_{\substack{j = 1\\j \neq i}}^{n-1}\psi_j^m(\bB) \\
				& \geq -(n-1) C''(1 + \|\bB\|_F^2) \quad \forall m\,.
			\end{align*}

			We conclude that, for each $i \in [n]$, the sequence $(\psi_i^m)_{m \geq 1}$ satisfies a bound of the form
			\begin{equation}
				|\psi_i^m(\bB)| \leq C''' (1 + \|\bB\|_F^2) \quad \forall m\,.
			\end{equation}
			Moreover, since $\bB \mapsto \|\bB^\top \bx_i\|^2 - \psi_i^m(\bB)$ is convex for each $i = 1, \dots, n-1$ and $m \geq 1$, the functions $\psi_i^m$ are uniformly equicontinuous on any convex, compact subset of $\RR^{p \times d}$
			By the Arzel\`a--Ascoli theorem, we may therefore extract limiting dual variables $(\psi_1, \dots, \psi_n)$ to which a subsequence of $(\psi^m_1, \dots, \psi^m_n)$ converges uniformly on compacts.
			The constraint $\sum_i \psi_i = 0$ is clearly preserved in the limit, so to show that these limiting variables solve~\eqref{eq:dual}, we note that
			\begin{align*}
				\limsup_{m} S_i \psi_i^m(\by) & \leq \inf_{\bB} \left\{\limsup_{m} \|\by - \bB^\top \bx_i\|^2 - \psi_i^m(\bB)\right\} \\
				& = \inf_{\bB}\|\by - \bB^\top \bx_i\|^2 - \psi_i(\bB) \\
				& = S_i \psi_i(\by)\,.
			\end{align*}
			Since $\|\by_i\|^2 - S_i \psi_i^m(\by_i) \geq 0$ by~\eqref{eq:s_upper_bound}, Fatou's lemma implies that
			\begin{align*}
				\cS(\nu_1, \dots, \nu_n) & = \limsup_{m} \frac 1n \sum_{i=1}^n \int S_i \psi_i^m \dd \nu_i \\
				&  \leq \frac 1n \sum_{i=1}^n \int \limsup_{m} S_i \psi_i^m \dd \nu_i \\
				& \leq \frac 1n \sum_{i=1}^n \int S_i \psi_i \dd \nu_i\,,
			\end{align*}
			as claimed.
	\end{proof}

    \begin{proof}[Proof of \cref{thm:opt_duals}]
		Let $\psi_1, \dots, \psi_n$ be solutions to~\eqref{eq:dual}, whose existence is guaranteed by \cref{thm:lin_dual}.
		For each $i = 1, \dots, n$, define
		 $\varphi_i : \RR^d \to \RR$ by
	\begin{equation}\label{eq:phi_def}
		\varphi_i(\bz) := \left(\frac 12 \|\cdot\|^2 - \frac 12 S_i \psi_i\right)^*(\bz)\,.
	\end{equation}
	Clearly $\varphi_i$ is convex.

	We first claim that
	\begin{equation}\label{eq:psi_ineq}
		\frac 12 \psi_i(\bB) \leq \frac{\|\bB^\top \bx_i\|^2}{2} - \varphi_i(\bB^\top \bx_i)
	\end{equation}
	This follows directly from unrolling the definitions: for any $\bB \in \RR^{p \times d}$ and $\by \in \RR^d$, ~\eqref{eq:s_def} shows
	\begin{align*}
		\frac 12 \psi_i(\bB) & \leq \frac 12 \|\by - \bB^\top \bx_i\|^2 - \frac 12 S_i\psi_i(\by) \\
		& = \frac{\|\bB^\top \bx_i\|^2}{2} - \left(\by^\top \bB^\top \bx_i - \left(\frac 12 \|\by\|^2 - \frac 12 S_i \psi_i(\by)\right) \right) \\
		& \leq\frac{\|\bB^\top \bx_i\|^2}{2} -  \left(\frac 12 \|\cdot\|^2 - \frac 12 S_i \psi_i\right)^*(\bz)\,,
	\end{align*}
	by the definition of the convex conjugate.

	In particular, we obtain
	\begin{equation*}
		0 = \frac 1{2n} \sum_{i=1}^n \psi_i(\bB) \leq \frac 1n \sum_{i=1}^n \frac{\|\bB^\top \bx_i\|^2}{2} - \frac 1n \sum_{i=1}^n \varphi_i(\bB^\top \bx_i)\,,
	\end{equation*}
	proving~\eqref{eq:dual_ineq}.

	To show that $\varphi_i$ is an optimal Brenier potential and that~\eqref{eq:dual_ineq} holds $Q$-a.s., we exploit duality.
The optimality of $Q$ and $(\psi_1, \dots, \psi_n)$ implies
\begin{align}
	\frac 1n \sum_{i=1}^n W_2^2(\nu_i, Q_{\bx_i}) & = \frac 1n \sum_{i=1}^n \int S_i \psi_i \dd \nu_i \nonumber\\ & = \frac 1n \sum_{i=1}^n \left(\int S_i \psi_i \dd \nu_i + \int \psi_i \dd Q \right) \nonumber\\
	& \leq \frac 1n \sum_{i=1}^n \left(\int S_i \psi_i \dd \nu_i + \int( \|\cdot\|^2 - 2 \varphi_i) \dd Q_{\bx_i} \right)\,, \label{eq:ineq_1} \\
	& \leq \frac 1n \sum_{i=1}^n \left(\int (\|\cdot\|^2 - 2 \varphi_i^*) \dd \nu_i + \int( \|\cdot\|^2 - 2 \varphi_i) \dd Q_{\bx_i} \right) \label{eq:ineq_2}
	\end{align}
where the first inequality follows from~\eqref{eq:psi_ineq} and the second follows from
\begin{equation*}
	\varphi_i^* =  \left(\frac 12 \|\cdot\|^2 - \frac 12 S_i \psi_i\right)^{**} \leq \frac 12 \|\cdot\|^2 - \frac 12 S_i \psi_i\,.
\end{equation*}
Kantorovich duality for optimal transport~\cite{villani_optimal_transport_2009} then implies that $(\varphi_i, \varphi_i^*)$ are optimal Brenier potentials and that~\eqref{eq:ineq_1} and~\eqref{eq:ineq_2} are both equalities.
In particular, we see that we must have
\begin{equation}
	0 = \frac 1n \sum_{i=1}^n  \int( \|\cdot\|^2 - 2 \varphi_i) \dd Q_{\bx_i} = \int \frac 1n \sum_{i=1}^n \|\bB^\top \bx_i\|^2 - 2 \varphi_i(\bB^\top \bx_i) \dd Q(\bB)\,,
\end{equation}
which, combined with the inequality~\eqref{eq:dual_ineq}, shows that the integrand must vanish $Q$-a.s.
This establishes the forward direction of the claim.

Conversely, suppose that there exists $Q$ and $\varphi_1, \dots, \varphi_n$ satisfying conditions 1 and 2.
For $i = 1, \dots, n-1$, define
\begin{equation*}
	\psi_i(\bB) = \|\bB^\top \bx_i\|^2 - 2 \varphi_i(\bB^\top \bx_i)\,,
\end{equation*}
and set $\psi_n = - \sum_{i=1}^{n-1} \psi_i$.
Condition 2 guarantees that $\psi_n \leq \|\bB^\top \bx_n\|^2 - 2 \varphi_n(\bB^\top \bx_n)$.
Then clearly $(\psi_1, \dots, \psi_n)$ are feasible in~\eqref{eq:dual}, and for $i = 1, \dots, n$,
\begin{align*}
	S_i \psi_i(\by_i) & = \inf_{\bB} \|\by_i - \bB^\top \bx_i\|^2 - \psi_i(\bB) \\
	& \geq  \inf_{\bB} \|\by_i - \bB^\top \bx_i\|^2 - \|\bB^\top \bx_i\|^2 + 2 \varphi_i(\bB^\top \bx_i) \\
	& = \|\by_i\|^2 - 2 \sup_{\bB} (\by_i^\top \bB^\top \bx_i -\varphi_i(\bB^\top \bx_i)) \\
	& = \|\by_i\|^2 - 2 \varphi_i^*(\by_i)\,.
\end{align*}
We obtain
\begin{align*}
	\frac 1n \sum_{i=1}^n W_2^2(\nu_i, Q_{\bx_i}) & = \frac 1n \sum_{i=1}^n \left(\int (\|\cdot\|^2 - 2 \varphi_i^*) \dd \nu_i + \int( \|\cdot\|^2 - 2 \varphi_i) \dd Q_{\bx_i} \right) \\
	& \leq \frac 1n \sum_{i=1}^n \int S_i \psi_i \dd \nu_i + \int \frac 1n \sum_{i=1}^n \|\bB^\top \bx_i\|^2 - 2 \varphi_i(\bB^\top \bx_i) \dd Q(\bB) \\
	& =  \frac 1n \sum_{i=1}^n \int S_i \psi_i \dd \nu_i \,,
\end{align*}
where the first step follows from condition 1 and the last step follows from condition 2.
The claim then follows from \cref{thm:lin_dual}.
	\end{proof}

    \begin{proof}[Proof of \cref{theorem:W-normal-equations}]
		Both sides of~\eqref{eq:dual_ineq} are continuous, and since they are equal $Q$-a.e., they in fact agree everywhere on $\operatorname{supp}(Q)$.
		Let $\bB \in \operatorname{supp}(Q)$ be arbitrary.
		The subdifferential of the left side of~\eqref{eq:dual} at $\bB$ is~\cite[Theorem VI.4.1.1]{Hiriart-Urruty2011-dv}
		\begin{align*}
			\partial\left(\frac 1n \sum_{i=1}^n \varphi_i(\cdot^\top \bx_i)\right)(\bB) & = \left\{\frac 1n \sum_{i=1}^n g_i : g_i \in \partial \varphi_i(\cdot^\top \bx_i)(\bB)\right\} \\
			& = \left\{\frac 1n \sum_{i=1}^n x_i g_i^\top : g_i \in \partial \varphi_i(\bB^\top \bx_i)\right\}\,.
		\end{align*}
		On the other hand, since $\frac 1n \sum_{i=1}^n \varphi_i(\cdot^\top \bx_i) \leq \frac 1n \sum_{i=1}^n \frac{\|\cdot^\top \bx_i\|^2}{2}$ with equality at $\bB$, the subdifferential of the left side at $\bB$ is included in the subdifferential of the right side at $\bB$.
		Therefore
		\begin{equation*}
			\left\{\frac 1n \sum_{i=1}^n x_i g_i^\top : g_i \in \partial \varphi_i(\bB^\top \bx_i)\right\} \subseteq \left\{\Sigma_{XX} \bB \right\}\,.
		\end{equation*}
		Hence each subdifferential $\partial \varphi_i(\bB^\top \bx_i)$ is actually a singleton, so each $\varphi_i$ is differentiable at $\bB^\top \bx_i$, proving~\eqref{eq:W-normal-equations}.
	\end{proof}

\begin{lemma}\label{lemma:measure-lemma}
	Adopt either \cref{assume:compact} or \cref{assume:linear}, and let $P \in \cP_2((\RR^d)^n)$.
	Then
	    \begin{multline}\label{eq:measure_lemma_b}
		\int \min_{f\in\cF} \frac 1n \sum_{i=1}^n \|\by_i-f(\bx_i)\|^2_2 \dd {P}(\by_1, \dots, \by_n) \\
		=\min_{\substack{\mathsf{P} \in \cP_2((\RR^d)^n \times \cF)\\ (\pi_1)_\# \mathsf P = P}} \int \frac{1}{n}\sum_{i=1}^n \|\by_i-f(\bx_i)\|^2_2 \dd \mathsf{P}(\by_1, \dots, \by_n, f)\,.
	\end{multline}

\end{lemma}

\begin{proof}
	Let $\bY = (\by_1, \dots, \by_n) \in (\RR^d)^n$, and write $\Phi(\bY, f) := \frac 1n \sum_{i=1}^n \|\by_i-f(\bx_i)\|^2_2$ for brevity.
	Then $\Phi$ is continuous, and the function
	\begin{equation*}
		m(\bY) := \min_{f\in\cF}\Phi(\bY, f)
	\end{equation*}
	is well defined (the minimum is attained under either \cref{assume:compact} or \cref{assume:linear}) and continuous (by Berge's maximum theorem \cite[Theorem 17.31]{aliprantis-border-2006} under \cref{assume:compact} and by explicit computation under \cref{assume:linear}).

	Let $\mathsf{P}\in\cP_2((\RR^{d})^n \times \cF)$ be any admissible coupling, i.e.\ $(\pi_{1})_{\sharp}\mathsf{P} = P$.
	The fact that $\Phi(\bY, f) \geq m(\bY)$ for all $f \in \cF$ implies
	\begin{equation}\label{eq:lem_ineq_1}
		\int \Phi(\bY, f) \dd \mathsf{P}(\bY, f) \geq \int m(\bY) \dd \mathsf{P}(\bY, f) = \int m(\bY) \dd P(\bY)\,.
	\end{equation}

	On the other hand, suppose we can construct a Borel map $\sigma: (\RR^{d})^n \to \cF$ such that
	\begin{equation*}
		\Phi(\bY, \sigma(\bY)) = m(\bY) \quad \forall \bY \in (\RR^d)^n\,.
	\end{equation*}
	Then $\mathsf{P} = (\Id, \sigma)_\# P$ satisfies $(\pi_1)_\# \mathsf P = P$ and
	\begin{equation}\label{eq:lem_ineq_2}
		\int \Phi(\bY, f) \dd \mathsf{P}(\bY, f) = \int \Phi(\bY, \sigma(\bY)) \dd P(\bY) = \int m(\bY) \dd P(\bY)\,.
	\end{equation}
	Combining \eqref{eq:lem_ineq_1} and \eqref{eq:lem_ineq_2} yields \cref{eq:measure_lemma_b}.

	It remains to construct a suitable $\sigma$.
	Under \cref{assume:compact}, we note that for each $\bY$ the set $\cF_\bY := \argmin_{f \in \cF} \Phi(\bY, f)$ is nonempty and the correspondence $\bY \mapsto \cF_{\bY}$ is upper-hemicontinuous (again by Berge's theorem), hence the Kuratowski--Ryll-Nardzewski measurable selection theorem \cite[Theorem 18.13]{aliprantis-border-2006} guarantees the existence of a Borel $\sigma$ for which $\sigma(\bY) \in \cF_{\bY}.$
	Under \cref{assume:linear}, we may choose the minimum-norm solution $\sigma(\bY) = (\bX^\top \bX )^{+}\bX^\top \bY \in \cF_\bY$, which is clearly also Borel.
	This proves the claim.

\end{proof}

\FloatBarrier
\section{Estimation of Wasserstein least squares}\label{appendix:sample_complexity}
\begin{proof}[Proof of \cref{thm:stat_main}]
	We first reduce to the case where $\nabla \phi_i(0) = 0$ for all $i$.
    Write $\bm{v}_i := \nabla \phi_i(0)$, and define $\bar \phi_i(\by) = \phi_i(\by) - \langle \bm{v}_i, \by \rangle$.
	\Cref{cond:identity} implies that $\E[\bm{v}_i] = 0$; therefore, the function $\bar \phi_i$ still satisfies \cref{cond:convex,cond:alphabeta,cond:identity} as well as the further condition $\nabla \bar \phi_i(0) = 0$.
	Define $\bar \nu_i = (\nabla \bar \phi_i)_\# Q_{\bx_i}^\star$ and let $\widetilde \nu_i$ be an empirical measure corresponding to the shifted measure $\bar \nu_i$.
	Finally, we let $\widetilde Q  = \argmin_{Q \in \cP_2(\RR^{p \times d})} \frac 1n \sum_{i=1}^n W_2^2(\widetilde \nu_i, Q_{\bx_i})$ be the analogue of~\eqref{eq:empirical_estimate} constructed from the shifted measures $\widetilde \nu_i$.

	The data $\widetilde \nu_i$ and estimator $\widetilde Q$ correspond precisely to the model under investigation, under the added assumption that the deformations $\phi_i$ satisfy $\nabla \phi_i(0) = 0$ almost surely.
	Moreover, the following lemma shows that the expected in-sample error for the original data can be bounded by considering the same quantity for the shifted data.
	\begin{lemma}~\label{lem:mean_zero}
		The in-sample error satisfies
		\begin{equation}
			\E \left[	\frac 1n \sum_{i=1}^n W_2^2(\widehat Q_{\bx_i}, Q^\star_{\bx_i})\right] \lesssim \E \left[	\frac 1n \sum_{i=1}^n W_2^2(\widetilde Q_{\bx_i}, Q^\star_{\bx_i})\right] + \sigma^2 \frac{pd}{n}\,.
		\end{equation}
	\end{lemma}

	We therefore adopt the assumption that $\nabla \phi_i(0) = 0$ almost surely in what follows, and conclude the ultimate error bound via \cref{lem:mean_zero}.
	Under this reduction, $\nabla \bar\phi_i(0) = 0$ and $\nabla \bar\phi_i$ is $\beta$-Lipschitz by \cref{cond:alphabeta}, so the assumption~\eqref{eq:latent_bounded} that $\operatorname{supp}(Q^\star_{\bx_i}) \subset B_M$ implies $\operatorname{supp}(\bar\nu_i),\, \operatorname{supp}(\widetilde\nu_i) \subset B_{\beta M}$ almost surely.
	Note that the boundedness assumption~\eqref{eq:latent_bounded} is invariant under this centering reduction, since it is imposed on $Q^\star$ rather than on the deformed responses $\nu_i$.

	Let $\widehat G:\cP_2(\RR^{p\times d})\to\mathbb{R}$ be the empirical loss
	\begin{equation}
		\widehat G(Q):=\frac{1}{n}\sum^n_{i=1}W^2_2(\widehat\nu_i,Q_{\bx_i})
	\end{equation}
	and let $G: \cP_2(\RR^{p\times d})\to\mathbb{R}$ be its population counterpart
	\begin{equation}
		G(Q):=\frac{1}{n}\sum^n_{i=1}W^2_2(\nu_i,Q_{\bx_i})\,.
	\end{equation}
	We exploit a strong convexity inequality for the squared Wasserstein distance~\citep[Theorem 6]{Manole2024PluginES}.
	For completeness, we give a proof of the particular bound we employ in \cref{lem:stab}.
	Since each $\phi_i$ is $\alpha$-strongly convex, \cref{lem:stab} yields
	\begin{multline}\label{eq:intermediate_bound}
		\frac{\alpha}{n} \sum_{i=1}^n W_2^2(\widehat Q_{\bx_i}, Q^\star_{\bx_i})
		\leq \int \left(\frac 1n \sum_{i=1}^n
		\|\cdot^\top\bx_i\|^2 - 2\phi_i(\cdot^\top\bx_i)\right)
		\dd(Q^\star - \widehat Q)\\
		+ (G(\widehat Q) - G(Q^\star))\,.
	\end{multline}

	For any $R > 0$, let $B_R := \{\by \in \RR^d: \|\by\| \leq R\}$ be the closed ball of radius $R$ in $\RR^d$.
	By Lemma~\ref{lem:compact}, there exists $R > 0$ such that, under both $\bB\sim \widehat Q$ and $\bB \sim Q^\star$, the support of $\bX \bB$ lies in $K := B_R^n \subseteq \RR^{n \times d}$.
	Denote by $\cK$ the set of probability measures on $\RR^{p \times d}$ with this property.
	Then
	since $\widehat Q$ minimizes $\widehat G$, we have
	\begin{equation}
		G(\widehat Q) - G(Q^\star) \leq G(\widehat Q) - \widehat G(\widehat Q) + \widehat G(Q^\star) - G(Q^\star) \leq 2\sup_{Q\in \cK} |G( Q) - \widehat{G}(Q)|\,.
	\end{equation}
	For $\mathbf{Y} \in \RR^{n \times d}$ with rows $\by_1^\top, \dots, \by_n^\top$, define the random variable
	\begin{equation}
		Z_{\mathbf{Y}} :=  \frac 1n \sum_{i=1}^n \|\by_i\|^2 - 2 \phi_i(\by_i)
	\end{equation}
	Under \cref{cond:convex,cond:identity}, $\E[\phi_i(\by)] = \|\by\|^2/2$, so $Z_{\bY}$ is a mean-zero stochastic process on $\RR^{n \times d}$, and we have
	\begin{align}
		\left|\int \left(\frac 1n \sum_{i=1}^n {\|\cdot^\top\bx_i\|^2} - 2\phi_i(\cdot^\top\bx_i)\right) \dd(Q^\star - \widehat Q)\right| & = \left| \int  Z_{\bX \bB} \,(\widehat Q - Q^\star)(\mathrm{d} \bB)\right| \\
		& \leq \sup_{\bB : \bX \bB \in  K} |Z_{\bX \bB}|\,.
	\end{align}

	Combining these estimates with~\eqref{eq:intermediate_bound} yields
	\begin{equation}\label{eq:two_terms}
		\frac{\alpha}{n} \sum_{i=1}^n W_2^2(\widehat Q_{\bx_i}, Q^\star_{\bx_i}) \lesssim \sup_{\bB : \bX \bB \in  K} |Z_{\bX \bB}| + \sup_{Q \in \cK} |G(Q) - \widehat G(Q)|\,.
	\end{equation}

	The first term is controlled by Dudley's chaining bound~\cite{vershyninHighDimensionalProbability2018}: \Cref{lem:z_bound} shows that
	\begin{equation}\label{eq:z_bound}
		 \E \sup_{\bB : \bX \bB \in  K} |Z_{\bX \bB}| \lesssim \frac{\beta R^2 \sqrt{pd}}{\sqrt{n}}\,.
	\end{equation}
	A bound on the second term follows from known results in empirical convergence rates for the Wasserstein distance~\cite{chizat-peyre-estimation}: by  \cref{lemma:exchange_sup_expcted},
	\begin{equation}\label{eq:g_bound}
		\E\left[\sup_{Q \in \cK} |G(Q) - \widehat G(Q)|\right]
		\leq \frac{1}{n}\sum_{i=1}^n
		\E\left[\sup_{\mu \in \cP(B_R)} |W_2^2(\mu, \nu_i) - W_2^2(\mu,
		\widehat \nu_i)|\right] \lesssim R^2  r_m\,.
	\end{equation}

	Applying~\eqref{eq:z_bound} and~\eqref{eq:g_bound} to~\eqref{eq:two_terms} and recalling \cref{lem:mean_zero} proves the claim.

\end{proof}

\begin{proof}[Proof of \cref{lem:mean_zero}]
		Given any probability measure $\mu$ with finite first moment, denote its mean by $m(\mu)$.
	The Wasserstein distance satisfies the following translation equivariance property.
	Given a vector $v$, let $\tau^v(x) = x + v$ be the map that translates by $v$.
	Then we have
	\begin{equation}
		W_2^2(\mu, \nu) = W_2^2(\tau^{-m(\mu)}_\# \mu, \tau^{-m(\nu)}_\# \nu ) + \|m(\mu) - m(\nu)\|^2\,.
	\end{equation}
	In~\eqref{eq:empirical_estimate}, if we write $\widehat m_i := m(\widehat \nu_i)$ and $\bM := m(Q)$, then the objective function may therefore be  written as
	\begin{equation}
		\frac 1n \sum_{i=1}^n W_2^2(\tau^{-\widehat m_i}_\# \widehat \nu_i , (\tau^{-\bM}_\# Q)_{\bx_i}) + \|\widehat m_i - \bM^\top \bx_i\|^2\,.
	\end{equation}
	Reparametrizing in terms of the mean-zero measure $\tau^{-\bM}_\# Q$, we can therefore equivalently consider two separate minimization problems, the first optimizing over the ``shape'' of $Q$, and the second over its mean:
	\begin{align*}
		\widehat Q^0 & \in \argmin_{Q^0 \in \cP_2(\RR^{p \times d}): \E_{Q^0} \bB = 0} \frac 1n \sum_{i=1}^n W_2^2(\tau^{-\widehat m_i}_\# \widehat \nu_i, Q^0_{\bx_i}) \\
		\widehat {\bM} & \in \argmin_{\bM \in \RR^{p \times d}} \frac 1n \sum_{i=1}^n \|\widehat m_i - \bM^\top \bx_i\|^2 \\
		\widehat Q &:= \tau^{\widehat \bM}_\# \widehat Q^0\,.
	\end{align*}

	By the same logic, the in-sample error satisfies
	\begin{equation}
		\frac 1n \sum_{i=1}^n W_2^2(\widehat Q_{\bx_i}, Q^\star_{\bx_i}) = \frac 1n \sum_{i=1}^n W_2^2(\widehat Q^0_{\bx_i}, (\tau^{- \bM^\star}_\#  Q^\star)_{\bx_i}) + \|\widehat \bM^\top \bx_i - (\bM^\star)^\top \bx_i\|^2
	\end{equation}
	Applying the same decomposition to $\widetilde Q$ yields $\widetilde Q = \tau^{\widetilde \bM}_\# \widetilde Q^0$ and
	\begin{equation}\label{eq:tilde_insample}
		\frac 1n \sum_{i=1}^n W_2^2(\widetilde Q_{\bx_i}, Q^\star_{\bx_i}) = \frac 1n \sum_{i=1}^n W_2^2(\widetilde Q^0_{\bx_i}, (\tau^{-\bM^\star}_\# Q^\star)_{\bx_i}) + \|\widetilde \bM^\top \bx_i - (\bM^\star)^\top \bx_i\|^2\,.
	\end{equation}
	\textit{Shape errors coincide.}
	Since $\widehat \nu_i = \tau^{\bm{v}_i}_\# \widetilde \nu_i$, we have $\widehat m_i = \widetilde m_i + \bm{v}_i$ where $\widetilde m_i = m(\widetilde \nu_i)$, and therefore $\tau^{-\widehat m_i}_\# \widehat \nu_i = \tau^{-\widetilde m_i}_\# \widetilde \nu_i$.
	The centered data being identical implies $\widehat Q^0 = \widetilde Q^0$, so the shape errors in~\eqref{eq:tilde_insample} are equal to those in the preceding display.

	\textit{Mean errors.}
	The OLS solutions satisfy $\bX \widehat \bM = \bm{H} \widehat{\bm{m}}$ and $\bX \widetilde \bM = \bm{H} \widetilde{\bm{m}}$, where $\bm{H} = \bX(\bX^\top \bX)^+ \bX^\top$ is the hat matrix and where $\widehat{\bm{m}}$ and $\widetilde{\bm{m}}$ are $n \times d$ matrices with rows $\widehat m_i^\top$ and $\widetilde m_i^\top$, respectively.
	Since $\widehat{\bm{m}} = \widetilde{\bm{m}} + \bm{V}$ (where $\bm{V}$ is the $n \times d$ matrix with rows $\bm{v}_i^\top$), the inequality $\|a + b\|^2 \leq 2\|a\|^2 + 2\|b\|^2$ gives
	\begin{equation*}
		\frac 1n \|\bX(\widehat \bM - \bM^\star)\|^2
		= \frac 1n \|\bX(\widetilde \bM - \bM^\star) + \bm{H}\bm{V}\|^2
		\leq \frac 2n \|\bX(\widetilde \bM - \bM^\star)\|^2 + \frac 2n \|\bm{H}\bm{V}\|^2\,.
	\end{equation*}
	For the last term, independence of $\phi_i$ across $i$ and $\E[\bm{v}_i] = 0$ give
	$\E[\|\bm{H}\bm{V}\|_F^2] = \sum_{i=1}^n H_{ii}\, \E[\|\bm{v}_i\|^2] \leq d\sigma^2 \operatorname{Tr}(\bm{H}) \leq d\sigma^2 p$,
	where we used \cref{cond:convex} and $\operatorname{Tr}(\bm{H}) = \operatorname{rank}(\bX) \leq p$.
	Combining the equality of shape errors with the mean error bound gives
	\begin{equation*}
		\E\left[\frac 1n \sum_{i=1}^n W_2^2(\widehat Q_{\bx_i}, Q^\star_{\bx_i})\right] \leq 2\,\E\left[\frac 1n \sum_{i=1}^n W_2^2(\widetilde Q_{\bx_i}, Q^\star_{\bx_i})\right] + 2\sigma^2 \frac{pd}{n}\,.
	\end{equation*}
\end{proof}

\begin{lemma}\label{lem:stab}
  Under the conditions of \cref{thm:stat_main},
  \begin{multline}\label{eq:M-growth}
    \frac{\alpha}{2n}\sum_{i=1}^n W_2^2(\widehat{Q}_{\bx_i}, Q^\star_{\bx_i})
      \leq \int \left(\frac{1}{n}\sum_{i=1}^n \frac{\|\cdot^\top\bx_i\|^2}{2}
      - \phi_i(\cdot^\top\bx_i)\right) d(Q^\star - \widehat{Q})\\
      + (G(\widehat Q) - G(Q^\star))\,.
  \end{multline}
\end{lemma}
\begin{proof}
	Given $\alpha$-strong convexity of $\phi_i$ for each $i=1,...,n$ we have that
	\begin{equation*}
		\phi_i^*(x)+\phi_i(y)\geq \langle x, y \rangle + \frac{\alpha}{2}\|y-\nabla\phi^*_i(x)\|^2.
	\end{equation*}
	Expanding the square and rearranging
	\begin{equation*}
		\frac{\|x-y\|^2}{2} -\frac{\alpha}{2}\|y-\nabla\phi^*_i(x)\|^2 \geq \frac{\|x\|^2}{2}-\phi^*_i(x) + \frac{\|y\|^2}{2}-\phi_i(y).
	\end{equation*}
	Integrating the previous expression w.r.t $\gamma_i$, the optimal coupling  between $\nu_i$ and $\widehat{Q}_{\bx_i}$, we arrive at
	\begin{equation*}
		\frac 12 W^2_2(\nu_i,\widehat{Q}_{\bx_i})-\frac{\alpha}{2}\int\|y-\nabla\phi^*_i(x)\|^2d\gamma_i\geq \int(\frac{\|\cdot\|^2}{2}-\phi^*_i)d\nu_i + \int(\frac{\|\cdot\|^2}{2}-\phi_i )d\widehat{Q}_{\bx_i}.
	\end{equation*}
	Noting that $W^2_2(\widehat{Q}_{\bx_i},Q^\star_{\bx_i})\leq \int\|y-\nabla\phi^*_i(x)\|^2d\gamma_i$ and summing over $1\leq i\leq n$ gives us
	\begin{equation*}
		G(\widehat{Q}) - \frac{\alpha}{2n}\sum_{i=1}^n W_2^2(\widehat{Q}_{\bx_i}, Q^\star_{\bx_i}) \geq \frac{1}{n}\sum_{i=1}^n \int \left(\|\cdot\|_2^2/2 - \phi_i^*\right) d\nu_i + \int \left(\|\cdot\|_2^2/2 - \phi_i\right) d\widehat{Q}_{\bx_i}
	\end{equation*}
	So
	\begin{align*}
		G(\widehat{Q}) - \frac{\alpha}{2n}\sum_{i=1}^n W_2^2(\widehat{Q}_{\bx_i}, Q^\star_{\bx_i})
		&\geq \frac{1}{n}\sum_{i=1}^n \left[\int \left(\frac{\|\cdot\|^2}{2}  - \phi_i^*\right) d\nu_i\right.\\
		&\qquad\left.+ \int \left(\frac{\|\cdot\|^2}{2}  - \phi_i\right) d\widehat{Q}_{\bx_i}\right]\\
		&= G(Q^\star) + \frac{1}{n}\sum_{i=1}^n \int \left(\frac{\|\cdot\|^2}{2} - \phi_i\right) d(\widehat{Q}_{\bx_i} - Q^\star_{\bx_i})\\
		&= G(Q^\star) \\
		&\quad + \int \left(\frac{1}{n}\sum_{i=1}^n \frac{\|\cdot^\top\bx_i\|^2}{2} - \phi_i(\cdot^\top\bx_i)\right) d(\widehat{Q} - Q^\star).
	\end{align*}

	Re-ordering the previous expression yields~\eqref{eq:M-growth}.
\end{proof}
\begin{lemma}\label{lem:z_bound}
	Let $K = B_R^n \subseteq \RR^{n \times d}$, where $B_R = \{\by \in \RR^d: \|\by\| \leq R\}$.
	Then
	\begin{equation}
		\E \sup_{\bB : \bX \bB \in  K} |Z_{\bX \bB}| \lesssim \beta R^2 \sqrt{\frac{pd}{n}}\,.
	\end{equation}
\end{lemma}
\begin{proof}
	We first show that $Z_\bY$ has sub-gaussian increments:
	For any $\bY, \bY' \in K$, we will show
	\begin{equation}\label{eq:subg}
		\|Z_{\bY} - Z_{\bY'}\|_{\psi_2} \leq \frac{\sigma}{n} \|\bY - \bY'\|
	\end{equation}
	where $\sigma = \beta R$.
	Indeed, \Cref{cond:alphabeta} and the assumption that $\nabla \phi_i(0) = 0$ imply that the function
	\begin{equation*}
		\psi_i(\by) = \|\by\|^2 - 2 \phi_i(\by)
	\end{equation*}
	satisfies $\nabla \psi_i(0) = 0$ and $\|\nabla^2 \psi_i(\by)\|_{\mathrm{op}} \leq 2 \max\{1-\alpha, \beta - 1\} =: L$.
	Hence the gradient of $\psi_i$ is $L$ Lipschitz, and in particular $\|\nabla \psi_i(\by)\| \leq L\|\by\|$ for all $\by$.
	Therefore
	\begin{align*}
		\left| \|\by_i\|^2 - 2 \phi_i(\by_i) - ( \|\by'_i\|^2 - 2 \phi_i(\by'_i))\right| & = |\psi_i(\by_i) - \psi_i(\by'_i)| \\
		& \leq L \|\by_i - \by_i'\| \max\{\|\by_i\|, \|\by_i'\|\}
	\end{align*}
	We obtain that $Z_{\bY} - Z_{\bY'}$ is an average of $n$ independent, mean zero terms, and the $i$th term is bounded by $L \|\by_i - \by_i'\| \max\{\|\by_i\|, \|\by_i'\|\}$ almost surely.
	Hoeffding's lemma therefore implies that for any $\bY, \bY' \in K$, the increment $Z_{\bY} - Z_{\bY'}$ is
	\begin{equation}
		\frac{L^2}{4n^2} \max_{i \in [n]}\max\{\|\by_i\|^2, \|\by_i'\|^2\} \|\bY - \bY'\|^2
	\end{equation}
	subgaussian.
	\Cref{cond:alphabeta} implies that $\alpha \leq 1 \leq \beta$, so $L \leq 2 \beta$, and the definition of $K$ shows that any $\bY \in K$ satisfies $\max_{i \in [n]} \|\by_i\| \leq R$.
	This proves~\eqref{eq:subg}.

	The result now follows from Dudley's entropy bound~\cite[Theorem~8.1.6]{vershyninHighDimensionalProbability2018}.
	Equip the indexing set $T := \{\bB \in \RR^{p \times d} : \bX \bB \in K\}$ with the pseudo-metric $\rho(\bB, \bB') = \|\bX(\bB - \bB')\|$.
	The covering numbers of $(T, \rho)$ equal those of the image $\bX(T) = \mathrm{col}(\bX)^d \cap K$ under $\|\cdot\|$.
	This set has dimension at most $pd$ and Frobenius diameter at most $2R\sqrt{n}$, so standard volumetric estimates give
	\begin{equation}
		\log N(\eps, T, \rho) \leq pd \log\!\left(\frac{6R\sqrt{n}}{\eps}\right)\,.
	\end{equation}
	Applying Dudley's entropy bound to the sub-gaussian process $Z_{\bX\bB}$ with increment parameter $\beta R / n$ from~\eqref{eq:subg} yields
	\begin{align}
		\E \sup_{\bB \in T} |Z_{\bX\bB}|
		&\lesssim \frac{\beta R}{n} \int_0^{R\sqrt{n}} \sqrt{pd \log\!\left(\frac{6R\sqrt{n}}{\eps}\right)} \, d\eps
		\lesssim \frac{\beta R^2 \sqrt{pd}}{\sqrt{n}}\,,
	\end{align}
	where the last step uses the standard evaluation of the entropy integral for a $pd$-dimensional body of diameter $R\sqrt{n}$.
\end{proof}
\begin{lemma}\label{lem:compact}
	Assume that $\nabla \phi_i(0) = 0$ for each $i \in [n]$, that the covariates satisfy the incoherence condition
	\begin{equation}
		\bx_i^\top (\bX^\top \bX)^+ \bx_i \leq \frac{\mu p}{n}\,,
	\end{equation}
	that $\operatorname{supp}(Q^\star_{\bx_i}) \subset B_M$ for each $i \in [n]$, and that $\operatorname{supp}(\widehat \nu_i) \subset B_{\beta M}$ almost surely for each $i \in [n]$.
	Then, under both $\bB \sim \widehat Q$ and $\bB \sim Q^\star$, the support of $\bX \bB$ lies in $B_R^n$ with $R = M\max(1, \beta\sqrt{\mu p})$.
\end{lemma}
\begin{proof}
	We handle the two cases separately.

	\textbf{Under $\bB \sim Q^\star$:}
	By assumption, $\operatorname{supp}(Q^\star_{\bx_i}) \subset B_M$, so $\|\bB^\top \bx_i\| \leq M$ almost surely for each $i \in [n]$.

	\textbf{Under $\bB \sim \widehat Q$:}
	Following~\eqref{eq:LR-normal_equations}, we can write $\bB = (\bX^\top \bX)\bX^\top \bY$, for $(\by_1, \dots, \by_n) \sim P$ for some $P \in \Pi(\widehat \nu_1, \dots, \widehat \nu_n)$.
	Writing $\bm{H} = \bX (\bX^\top \bX)^+ \bX^\top$, we equivalently have $\bX \bB = \bm{H} \bY$, where $\bY$ is some random matrix satisfying $\|\bY\|^2 \leq \sum_{i=1}^n \sup\{\|\by_i\|^2: \by_i \in \operatorname{supp}(\widehat \nu_i)\} \leq n (\beta M)^2$.
	We obtain that the $i$th row of $\bX \bB$ has norm bounded by $\|\bm{h}_i\| \|\bY\| \leq \sqrt n\, \beta M \|\bm{h}_i\|$, where $\bm{h}_i$ is the $i$th row of $\bm{H}$.
	Finally, since $\bm{H}$ is an orthogonal projection, we have that
	\begin{equation}
		\|\bm{h}_i\|^2 = (\bm{H} \bm{H}^\top)_{ii} = \bm{H}_{ii} = \bx_i^\top (\bX^\top \bX)^+ \bx_i \leq \frac{\mu p}{n}\,.
	\end{equation}
	We obtain that the $i$th row of $\bX \bB$ has norm at most $\beta M \sqrt{\mu p}$.

	Taking $R = M\max(1, \beta\sqrt{\mu p})$ completes the proof.
\end{proof}
\begin{lemma}\label{lemma:exchange_sup_expcted}
	Let $\Omega = B_R \subseteq \RR^d$ be the ball of radius $R$ around the origin in $\RR^d$. If $\nu \in \cP(\Omega)$ and $\widehat{\nu}$ is an empirical version of $\nu$ obtained from $m$ i.i.d.\ samples, then
	\begin{equation}\label{eq:expected_sup_bound}
		\E\left[\sup_{\mu\in\mathcal{P}(\Omega)}|W^2_2(\mu,\nu)-W^2_2(\mu,\widehat{\nu})|\right]\lesssim R^2 \begin{cases}
			m^{-2/d} &\text{if $d>4$,}\\
			m^{-1/2}\log(m) &\text{if $d=4$,}\\
			m^{-1/2} &\text{if $d<4$,}
		\end{cases}
	\end{equation}
\end{lemma}
\begin{proof}
	Following the proof of Lemma 3 in \cite{chizat-peyre-estimation}, we can bound the left side of (\ref{eq:expected_sup_bound}) in terms of the semi-dual problem, yielding
	\begin{multline*}
		\E \left[\sup_{\mu\in\mathcal{P}_2(\Omega)}|W^2_2(\mu,\nu)-W^2_2(\mu,\widehat{\nu})|\right]\\
		\leq \E\left[\sup_{\mu\in\mathcal{P}_2(\Omega)}\left(\left|\int\|y\|^2\,d(\nu-\widehat{\nu})\right|+\sup_{\phi\in\mathcal{F}_{R}}\left|\int\phi \,d(\nu-\widehat{\nu})\right|\right)\right],
	\end{multline*}
	where $\mathcal{F}_{R}$ is the class of convex, $R$-Lipschitz functions defined in $B_R$. But the terms inside the supremum of the rhs do not depend on $\mu$, hence
	\begin{multline*}
		\E \left[\sup_{\mu\in\mathcal{P}_2(\Omega)}|W^2_2(\mu,\nu)-W^2_2(\mu,\widehat{\nu})|\right]\\
		\leq \E\left[\left|\int\|y\|^2\,d(\nu-\widehat{\nu})\right|+\sup_{\phi\in\mathcal{F}_{R}}\left|\int\phi \,d(\nu-\widehat{\nu})\right|\right].
	\end{multline*}
	Lemma 4 in \cite{chizat-peyre-estimation} yields the result.

\end{proof}

\FloatBarrier
\section{First order geometry Wasserstein least squares}\label{appendix:FOG}
In this section, we will describe the elements on $\mathcal{P}_{2}(\RR^{p\times d})$ that make the Wasserstein gradient of $G$, the Wasserstein regression functional
$$G(Q)=\frac{1}{n}\sum_{i=1}^nW^2_2(Q_{\bx_i},\nu_i),$$
 vanish.
Our goal, as is usual in first-order optimization, is to gain knowledge of the set of minimizers for a Wasserstein least squares problem.
\begin{proof}[Proof of \cref{lemma:W-gradient-of-G}]
\begin{sloppypar}
As a simplifying assumption, we will consider $(Q,\nu_i)\in\cP_{2,\text{ ac}}(\RR^{p\times d})\times\cP_{2,\text{ ac}}(\RR^d)$.
Then we can compute the $L_2-$first variation of $G$ in $Q$ by looking at the dual formulation of each $W^2_2(\nu_i,Q_{\bx_i})$ and noting that $\phi_i=\|\cdot\|-\varphi_i$, i.e.
\end{sloppypar}
\begin{equation*}
    \begin{aligned}
        \delta G(Q)&=-\delta\frac{1}{n}\sum^n_{i=1}(\int_{\mathbb{R}^d}(\varphi_i-\frac{\|\cdot\|^2}{2})dQ_{\bx_i} -\int_{\mathbb{R}^d}{\phi^*}d\nu_i)\\
        =&-\delta\frac{1}{n}\sum^n_{i=1}\int_{\mathbb{R}^d}(\varphi_i-\frac{\|\cdot\|^2}{2})dQ_{\bx_i},\\
    \end{aligned}
\end{equation*}
where we used the envelope theorem to eliminate the $\nu_i$ integrands.
Using the gluing lemma to produce a coupling $Q_X\in\mathcal{P}_{2,ac}((\mathbb{R}^d)^n)$ where every $Q_{\bx_i}$ is involved, we get
\begin{equation*}
    \begin{aligned}
        \delta G(Q)&= -\delta\int_{(\mathbb{R}^{d})^n}(\frac{1}{n}\sum^n_{i=1}(\varphi_i-\frac{\|\cdot\|^2}{2}))dQ_{X}\\
        &=-\delta \int_{\mathbb{R}^{p\times d}}((\frac{1}{n}\sum^n_{i=1}(\varphi_i-\frac{\|\cdot\|^2}{2}))\circ L_X )dQ
    \end{aligned}
\end{equation*}
where $L_X:\mathbb{R}^{p\times d}\to(\mathbb{R}^d)^n$ is given by $L_X(B)=(\bB^\top \bx_1,...,\bB^\top \bx_n)$.

Hence
\begin{equation}\label{eq:first_variation_of_G}
    \delta G(Q)= -(\frac{1}{n}\sum^n_{i=1}(\varphi_i-\frac{\|\cdot\|^2}{2}))\circ L_X .
\end{equation}

The $W_2-$gradient of $G$ with respect to $Q$ is the $\RR^{p\times d}-$gradient of $\delta G$, meaning
\begin{equation}\
\begin{aligned}
        \WGrad_Q G(Q)&=\nabla_{\bB}\delta G(Q)\\
        &=-\nabla_{\bB}(\frac{1}{n}\sum^n_{i=1}(\varphi_i(\bB^\top\bx_i)-\frac{\|\bB^\top\bx_i\|^2}{2}))\\
        &=-\frac{1}{n}\sum^n_{i=1}(\bx_i(\nabla\varphi_i(\bB^\top\bx_i))^\top -\bx_i\bx_i^\top \bB ).
\end{aligned}
\end{equation}
\end{proof}

The following observation appears in Theorem 7 of \cite{chewiGradientDescentAlgorithms2020c}.
\begin{lemma}\label{lemma:distance_smoothness}
    The function $F_{\nu_i}:\mathcal{P}_2(\RR^{d})\to\RR$ with assignment rule $$\mu \mapsto W^2_2(\mu,\nu_i)$$ is smooth along Wasserstein geodesics in $\mathcal{P}_{2,\text{ac}}(\mathbb{R}^d)$.
\end{lemma}
\begin{proof}
    Consider $\mu_s$ the constant speed geodesic joining $\mu_0$ and $\mu_1$.
     Letting $s\in(0,1]$, the NNC inequality reads as
     \begin{equation*}
        \frac{ W^2_2(\mu_s,\nu_i)-W^2_2(\mu_0,\nu_i)}{s}\geq [W^2_2(\mu_1,\nu_i)-W^2_2\mu_0,\nu_i)]-(1-s)W^2_2(\mu_0,\mu_1).
     \end{equation*}
     Taking $s\to 0_+$ and recalling Proposition 7.3.6 in \cite{ambrosioGradientFlowsMetric2008} yields the convergence of the left-hand side to
    \begin{equation*}
    W^2_2(\mu_0,\nu_i)|_{s=0_+}=-2\langle\nabla\varphi_{\mu_0\to\nu_i}-\Id,\log_{\mu_0}(\mu_1)\rangle_{\mu_0}=\langle\WGrad_{\mu_0} F_{\nu_i}, \log_{\mu_0}(\mu_1)\rangle_{\mu_0}.
    \end{equation*}

    Upon reordering, we achieved the desired inequality, meaning
    \begin{equation}
        F_{\nu_i}(\mu_1)\leq F_{\nu_i}(\mu_0)+ \langle\WGrad_{\mu_0} F_{\nu_i}, \log_{\mu_0}(\mu_1)\rangle_{\mu_0}+ W^2_2(\mu_0,\mu_1).
    \end{equation}
\end{proof}

The following theorem relates \eqref{eq:W-Flow-linear} and \eqref{eq:gradient_descent_G} in a practical way. It is a generalization of Theorem 7 in \cite{chewiGradientDescentAlgorithms2020c}.

\begin{theorem}\label{theorem:smoothness}

    Let $\eta=\frac{2}{n}\sum^n_{i=1}\|\bx_i\|^2_2$ and $\tau>0$. The Wasserstein least-squares functional $G$ is $\eta$-smooth along Wasserstein geodesics. As a consequence, if we define $Q^+:=\exp_Q\left(-\WGrad_Q G(Q)\right)$ for every $Q\in\cP_2(\mathbb{R}^{p\times d})$, then
    \begin{equation}\label{eq:descent_rate}
        G(Q^+)-G(Q)\leq (\frac{\eta}{2}-\tau)\left\|\WGrad_Q G\right\|^2_Q \quad \forall Q\in \cP_{2}(\mathbb{R}^{p\times d}).
    \end{equation}
\end{theorem}
\begin{proof}[Proof of Theorem \ref{theorem:smoothness}.]
    It is enough to verify that the function $G_i:\mathcal{P}_2(\RR^{p\times d})\to\RR$ with assignment rule $Q \mapsto W^2_2(Q_{\bx_i},\nu_i)$ is smooth for $1\leq i\leq n$.
    Let $Q^s$ be the constant speed geodesic joining $Q^0$ and $Q^1$ in $\mathcal{P}_{2,\text{ac}}(\RR^{p\times d})$.
    Due to Lemma \ref{lemma:distance_smoothness}
        \begin{equation}\label{eq:smoothness_on_Q_i}
        F_{\nu_i}(Q^1_{\bx_i})\leq F_{\nu_i}(Q^0_{\bx_i})+ \langle\WGrad_{Q^0_{\bx_i}} F_{\nu_i}, \frac{d}{ds}(Q^s)_i|_{s=0_+}\rangle_{Q^0_{\bx_i}}+ W^2_2(Q^0_{\bx_i},Q^1_{\bx_i}).
        \end{equation}
    But, note that $G_i=F_{\nu_i}\circ\pi_i$ where $\pi_i:\mathcal{P}_{2,\text{ac}}(\RR^{p\times d})\to\mathcal{P}_{2,\text{ac}}(\mathbb{R}^{d})$ is the $i$-th projection of $Q$, or $\pi_i(Q)=Q_{\bx_i}$. So, by the chain rule,
    \begin{equation*}
        \begin{aligned}
        \langle\WGrad_{Q^0_{\bx_i}} F_{\nu_i}, \frac{d}{ds}(Q^s)_{\bx_i}|_{s=0_+}\rangle_{Q^0_i}&=\langle\WGrad_{\pi_i\circ Q} F_{\nu_i}, \frac{d}{ds}(\pi_i\circ Q^s)|_{s=0_+}\rangle_{\pi_i\circ Q^0}\\
        &=\langle\WGrad_{Q^0}F_{\nu_i}\circ\pi_i,\log_{Q^0}(Q^1)\rangle_{Q^0}.
        \end{aligned}
    \end{equation*}
    Then we can rewrite equation \eqref{eq:smoothness_on_Q_i} as
    \begin{equation*}
        G_i(Q^1)\leq G_i(Q^0)+ \langle\WGrad_{Q^0}G_i,\log_{Q^0}(Q^1)\rangle_{Q^0}+ W^2_2(Q^0_{\bx_i},Q^1_{\bx_i}).
        \end{equation*}
    Furthermore, using the Cauchy-Schwarz inequality, we get that
    \begin{equation*}
        G_i(Q^1)\leq G_i(Q^0)+ \langle\WGrad_{Q^0}G_i,\log_{Q^0}(Q^1)\rangle_{Q^0}+ \|\bx_i\|^2_2W^2_2(Q^0,Q^1).
    \end{equation*}
    as long as $\bx_i$ is not a null vector.
    Multiplying each of the $n$ previous equations by $\frac{1}{n}$  and summing them up yields the smoothness:
    \begin{equation}\label{eq:G-smooth}
        G(Q^1)\leq G(Q^0)+ \langle\WGrad_{Q^0}G(Q_0),\log_{Q^0}(Q^1)\rangle_{Q^0}+ \frac{\eta}{2}W^2_2(Q^0,Q^1).
    \end{equation}
with $\eta=\frac{2}{n}\sum_{i=1}^n\|\bx_i\|^2_2$.

Letting $Q^0=Q$ and $Q^1=Q^+:=\exp_{Q}(\tau \WGrad G(Q))$ for any $Q\in\cP_2(\RR^{p\times d})$, and noting that
\begin{equation*}
    \langle\WGrad_{Q}G(Q),\log_{Q}(Q^+)\rangle_{Q^0}=\tau\langle\WGrad_{Q}G(Q),\WGrad_Q G(Q)\rangle_{Q^0}=\tau W^2_2(Q,Q^+)
\end{equation*}
concludes the proof upon rearrangement.
\end{proof}

Theorem \ref{theorem:smoothness} plays an important part in the convergence of \eqref{eq:gradient_descent_G} to a critical point of the Wasserstein least squares functional. For example, combining \cref{theorem:smoothness} with Proposition 4.7 in \cite{boumal2023intromanifolds} yields the following.

\begin{proposition}\label{prop:convergence_rate}
Choose  $\tau>\eta/2$. Then, for every $K\in\mathbb{N}$ there exists a $k$ in $0,\dots, K-1$ such that \begin{equation*}
        \|\WGrad G(Q^k)\|_{Q^k}\leq \sqrt{\frac{G(Q^0)}{K(\tau-\eta/2)}}
    \end{equation*}
    where the sequence $(Q^k)_{k>0}$ was produced according to \eqref{eq:gradient_descent_G}.
\end{proposition}

\FloatBarrier
\section{Wasserstein least squares: the Gaussian case}\label{appendix:geometry_of_gaussians}

In this section, we study the Wasserstein least squares problem under Gaussian responses. In conjunction with the benign properties derived in Section \ref{section:linear case}, the geometry of Gaussian measures equipped the Bures-Wasserstein metric will help us derive further connections with the Wasserstein barycenter problem via first-order conditions of the Wasserstein least squares functional.

Our analysis requires a careful treatment of $G:\mathcal{P}_{2}(\mathbb{R}^{p\times d})\to\RR$, the Wasserstein least squares functional for the data $(\bx_i,\nu_i)_{i=1}^n$, given by
\begin{equation}
    G(Q):=\sum_{i=1}^nW^2_2(\nu_i, Q_{\bx_i}).
\end{equation}

We begin by noting that, if each of the marginal measures is Gaussian, then Wasserstein least squares has a Gaussian solution.

\begin{proposition}\label{prop:gaussian_regression_sol_is_gaussian}
Suppose $\nu_1, \dots, \nu_n$ are Gaussian measures.
Then there exists a Gaussian measure $\widetilde Q$ on $\RR^{p \times d}$ such that
\begin{equation}\label{eq:min_wls_gauss}
    \widetilde{Q}\in\argmin_{Q\in\cP_{2}(\mathbb{R}^{p\times d})}\frac{1}{n}\sum_{i=1}^n W^2_2(\nu_i,Q_{\bx_i}).
\end{equation}
\end{proposition}
\begin{proof}
	We use the following fact~\cite{Gelbrich1990}:
    for every $\xi,\xi'\in\mathcal{P}_2(\mathbb{R}^d)$
\begin{equation}\label{eq:gelbrich}
    W^2_2(N_\xi,N_{\xi'})\leq W^2_2(\xi,\xi'),
\end{equation}
 where $N_\xi$ is the Gaussian distribution $\mathcal{N}(\mathbb{E}_{\xi}(X), \text{Cov}_{\xi}(X))$ whose first two moments match those of $\xi$.
Now, let $Q$ be an arbitrary minimizer of~\eqref{eq:min_wls_gauss}, and let $\widetilde Q = N_Q$ be the corresponding Gaussian.
Note that the projected measure $\widetilde Q_{\bx_i}$ is a Gaussian with the same mean and variance as $Q_{\bx_i}$, so~\eqref{eq:gelbrich} implies
\begin{equation*}
	W_2^2(\nu_i, \widetilde Q_{\bx_i}) \leq W_2^2(\nu_i,  Q_{\bx_i})\,,
\end{equation*}
so $\widetilde{Q}$ is also a minimizer in~\eqref{eq:min_wls_gauss}.
\end{proof}

\Cref{prop:gaussian_regression_sol_is_gaussian} implies that when solving the Wasserstein least squares problem with Gaussian responses we may restrict our attention to the subset of $\cP_2(\RR^{d \times p})$ consisting of Gaussian measures.
This finite-dimensional space equipped with the 2-Wasserstein metric is known as the \textit{Bures--Wasserstein} manifold~\cite{Bhatia2019}.
The objects introduced in \cref{sec:lifting} have simplified analogues on this space, which we now introduce.

For notational convenience, we work with vectorized versions of random variables on $\RR^{p \times d}$.
If  $Q\in\cP(\RR^{p\times d})$ is a Gaussian measure, then its representative in the Bures-Wasserstein manifold is given by $(m_Q,\Sigma_Q)\in\mathbb{R}^{dp}\times\mathbf{S}_{++}^{dp}$ where
    \begin{equation}
    \begin{cases}
        m_Q:=\EE[\vect{(\bB^\top)}],\\
        \Sigma_Q:=\E_Q[\text{vec}(\bB^\top-m_Q)\text{vec}(\bB^\top-m_Q)^\top]\\
     \bB\sim Q,
    \end{cases}
    \end{equation}
where $\vect(\cdot)$ maps matrices to vectors by stacking the columns: if $\bB^\top = [\bB^\top_1|...|\bB^\top_p] \in \RR^{d \times p}$, then
    \begin{equation*}
        \vect{(\bB^\top)}=\begin{bmatrix}
            \bB^\top_1\\
            \vdots\\
            \bB^\top_p
        \end{bmatrix}\in \mathbb{R}^{pd}.
    \end{equation*}

 With abuse of notation, we will not distinguish between $Q$, the measure supported in $\RR^{p\times d}$, and $Q$ as a measure in $\mathcal{P}_2(\RR^{dp})$ with vectorized mean $m_Q$ and covariance $\Sigma_Q$.
The following Lemma is a consequence of our parametrization.

\begin{lemma}\label{lemma:marginal_cov_gaussian}
Let $Q\in\mathcal{P}(\RR^{p\times d})$ be a Gaussian measure.
\begin{enumerate}
    \item If $\bB\sim Q$ and $\bx_i\in\RR^p$, then $Q_{\bx_i}:=\text{Law}(\bB^\top\bx_i)$ is Gaussian measure with mean and variance
\begin{equation}\label{prop:covariance_gaussian_marginals}
    \begin{aligned}
        m_{Q_{\bx_i}}&=(\bx_i^\top\otimes I_d) m_Q,\\
        \Sigma_{Q_{\bx_i}}&=(\bx^{\top}_i\otimes I_d)\Sigma_{Q}(\bx_i\otimes I_d).
    \end{aligned}
\end{equation}
\item A Brenier map, denoted by $\nabla \varphi_i$, between $Q_{\bx_i}$ and $\nu_i$ has the expression
\begin{equation}
    \nabla\varphi(\cdot)=m_{\nu_i}+ \Sigma_{Q_{\bx_i}}\#\Sigma_{\nu_i}(\cdot)-m_{Q_{\bx_i}}\,,
\end{equation}
where $\Sigma_1 \# \Sigma_2 :=\Sigma_1^{\frac12}(\Sigma^{-\frac{1}{2}}_1\Sigma_2\Sigma^{-\frac{1}{2}}_1)^{\frac 12}\Sigma_{1}^{\frac12}.$
\item \begin{sloppypar}The Bures-Wasserstein gradient of $G$ with mean-zero Gaussian data $(\{{\nu_i}\}_{i=1}^n, Q)$ specializes to\end{sloppypar}
\begin{equation}\label{eq:wgrad_gaussian}
    \WGrad_Q G(Q)= -\frac{1}{n}\sum^n_{i=1}\bx_i\bx_i^\top\bB(\Sigma_{Q_{\bx_i}}\#\Sigma_{\nu_i}- I_d).
\end{equation}
\end{enumerate}
\end{lemma}

\begin{proof}
Let $\bx\in\RR^{p}$. On one hand, if $Z\sim \mu:=\text{Law}({\bB^\top\bx})$, then
\begin{equation}\label{eq:integreation_on_Q}
\begin{aligned}
    m_\mu&=\E_\mu[Z]=\E_Q[
    \bB^\top\bx] \text{, and }\\
    \Sigma_{\mu}&=\mathbb{E}_{\mu}[Z^\top Z]=\mathbb{E}_{Q}[(\bB^\top\bx -m_\mu )(\bx^\top \bB-m_\mu^\top)].
\end{aligned}
\end{equation}
On the other hand, by properties of the $\vect{}$ operator
\begin{equation*}
     \bB^\top\bx=\vect(\bB^\top\bx)=\text{vec}(I_d\bB^\top\bx)= \bx^\top\otimes I_d\vect(\bB^\top).
\end{equation*}
Plugging the last equation back in (\ref{eq:integreation_on_Q}) yields
that
\begin{equation*}
    m_\mu=\E_{Q}[\bx^\top\otimes I_d\vect(\bB^\top)]=(\bx^\top\otimes I_d)m_Q.
\end{equation*}
Similarly, the covariance ends up being
\begin{equation*}
    \begin{aligned}
    \Sigma_{\mu}&=\mathbb{E}_{Q}[(\bx^{\top}\otimes I_d) \vect(\bB^\top-m_Q)\vect(\bB^\top-m_Q)^{\top}(\bx\otimes I_d) ]\\
    &=(\bx^{\top}\otimes I_d) \mathbb{E}_{Q}[\vect(\bB^\top-m_Q)\vect(\bB^\top-m_Q)^{\top}](\bx\otimes I_d)\\
    &=(\bx^{\top}\otimes I_d)\Sigma_{Q}(\bx\otimes I_d).\\
    \end{aligned}
\end{equation*}

Items 1 and 2 follow if:
\begin{enumerate}
    \item We recall that the Brenier map between two Gaussian measures $\xi,\xi'$ is given by
\begin{equation*}
    \nabla\varphi(\cdot)=m_{\xi'}+\Sigma_{\xi}\#\Sigma_{\xi'}(\cdot-m_\xi)=m_{\xi'}+\Sigma_{\xi}^{-\frac{1}{2}}(\Sigma_{\xi}^{\frac{1}{2}}\Sigma_{\xi'}\Sigma_{\xi}^{\frac{1}{2}})^{\frac{1}{2}}\Sigma_{\xi}^{-\frac{1}{2}}(\cdot-m_{\xi}).
\end{equation*}
    \item We apply Lemma \ref{lemma:W-gradient-of-G} to our data.
\end{enumerate}
\end{proof}

\begin{remark}\label{remark:centered-measures}
    The interesting case of study is when $\nu_1, \dots, \nu_n$ are centered. This reduction is reasonable since by Lemma \ref{lemma:marginal_cov_gaussian}, for general Gaussian data, the Wasserstein least squares functional decomposes to
\begin{equation}
    G(Q)=\sum_{i=1}^n\frac{1}{n}\|m_{\nu_i}-(\bx_i^\top\otimes I_d)m_{Q}\|^2_2+W^2_2(\widetilde{\nu_i},\widetilde{Q}_{\bx_i})
\end{equation}
where $(\{\widetilde{\nu_i}\}_{i=1}^n,\widetilde{Q})$ are zero-mean Gaussian measures with the same covariance of $(\{\nu_i\}_{i=1}^n,Q)$.
\end{remark}

The following statement is reminiscent of Theorem 6.1 in \cite{aguehBarycentersWassersteinSpace2011}.
\begin{proposition}\label{prop:optimal_dual_solution_gaussian}
    Let $S\in\mathbf{S}^{dp}_+$.
    Define by $S_i:=(\bx^\top_i\otimes I_d)S(\bx_i\otimes I_d)\in\mathbf{S}^{d}_+$.
    Consider $\nu_1,..., \nu_n$ $d$-dimensional Gaussian measures with null mean and covariances $\Sigma_{\nu_i}$. Assume $\Sigma_{\widetilde{Q}}$ is a solution of the normal matrix equation
    \begin{equation}\label{eq:first-order-condition}
        \frac{1}{n}\sum^n_{i=1}\bx_i\bx_i^\top\bB S_i\#\Sigma_{\nu_i}\ =\frac{1}{n}\sum^n_{i=1}\bx_i\bx_i^\top\bB
    \end{equation}
    for every $\bB\in\RR^{p\times d}$.
    Then $\widetilde{Q}=\mathcal{N}(0,\Sigma_{\widetilde{Q}})$ is an optimal dual solution of (\ref{eq:lin_wls}).
\end{proposition}

\begin{proof}
        By Theorem \ref{thm:opt_duals}, it is enough to verify that the mappings $(\varphi_1,\cdots,\varphi_n)$ given by
        \begin{equation*}
            \varphi_i:=\frac{1}{2}\langle \cdot\Sigma_{\widetilde{Q}_{\bx_i}}\#\Sigma_{\nu_i},\cdot\rangle.
        \end{equation*}
        are optimal dual solutions.
        The proposed maps are clearly the optimal Brenier potentials between $\widetilde{Q}_{\bx_i}$ and $\nu_i$.
        So we now verify that $$\frac{1}{n}\sum^n_{i=1}(\varphi_i(\bB^\top\bx_i)-\|\bB^\top\bx_i\|^2/2)=0 \quad \bB\text{-a.e}.$$
        Consequently we check that if $\bB\sim \widetilde{Q}$, then
        \begin{equation*}
            \begin{aligned}
                \frac{1}{n}\sum^n_{i=1}(\varphi_i(\bB^\top\bx_i)-\frac{\|\bB^\top\bx_i\|^2}{2})&=\sum_{i=1}^n\frac{1}{2n}(\langle \bx_i^\top\bB\Sigma_{\widetilde{Q}_{\bx_i}}\#\Sigma_{\nu_i}, \bx_i^\top\bB\rangle-\|\bB^\top\bx_i\|^2_2)\\
                &=\frac{1}{2}\langle \bB, \frac{1}{n}\sum^n_{i=1}\bx_i\bx_i^\top\bB\Sigma_{\widetilde{Q}_{\bx_i}}\#\Sigma_{\nu_i} \rangle_F+\sum_{i=1}^n\frac{1}{2n}\|\bB^\top\bx_i^\top\|^2\\
                &=\frac{1}{2}\langle \bB, \frac{1}{n}\sum^n_{i=1}\bx_i\bx_i^\top\bB \rangle_F+\sum_{i=1}^n\frac{1}{2n}\|\bB^\top\bx_i\|^2\\
                &=0.
            \end{aligned}
        \end{equation*}
    \end{proof}

 A further relationship between the Wasserstein barycenter problem and Wasserstein least squares is drawn in the following proposition.

\begin{proposition}\label{prop:covariates_first order_LD_case}
    Assume $\nu_i\sim \mathcal{N}_d(0,\Sigma_{\nu_i})$ and each $\bx_i=a_i \bx_1$ for a $a_i\neq0$. If $\widetilde{Q}$ is an optimal dual solution for \eqref{eq:lin_wls}, then $\Sigma_{\widetilde{Q}_{\bx_i}}$ solves the Barycenter equation for $S\in\mathbf{S}_+^{dp}$
    \begin{equation}\label{eq:almost_barycenter}
        \sum_{i=1}^n\widetilde{\lambda}_i(S^{1/2}\Sigma_{\nu_i}S^{1/2})^{1/2}=S
    \end{equation}
 Consequently, for this type of design matrix finding a solution of \eqref{eq:lin_wls} is equivalent to solving the Wasserstein barycenter problem
    \begin{equation}\label{eq:wasserstein_barycenter}
        \min_{b\in\mathcal{P}_{2,\text{ac.}}(\mathbb{R}^d)}\sum_{i=1}^n\widetilde\lambda_iW^2_2(b,\nu_i).
    \end{equation}
with $\widetilde{\lambda}_i=\frac{a_i^2}{\sum_{j=1}^n a_j^2}$.
\end{proposition}

\begin{proof}
   Note that by hypothesis
    \begin{equation*}
        \Sigma_{\widetilde{Q}_{\bx_i}}=a_i^2\mathbb{E}_{\bB\sim \widetilde{Q}}[\vect(\bB^\top \bx_1)\vect(\bB^\top \bx_1)^\top]=a_i^2\Sigma_{\widetilde{Q}_{\bx_1}}.
    \end{equation*}
    As a consequence
    \begin{equation*}
    \Sigma_{\widetilde{Q}_{\bx_i}}\#\Sigma_{\nu_i}=\Sigma_{\widetilde{Q}_{\bx_1}}\#\Sigma_{\nu_i} \quad\text{ for  }\quad1\leq i\leq n.
    \end{equation*}
      Specializing equation \eqref{eq:first-order-condition} to our case and multiplying it on the left side by $\bx^\top_1$ yields
      \begin{equation*}
       \sum^n_{i=1}a_i^2Z\Sigma_{Q_{\bx_1}}\#\Sigma_{\nu_i}=\sum^n_{i=1}a_i^2Z
    \end{equation*}
    where $Z=\bx^\top\bB(\|\bx_1\|_2^2 )$.
    Since $Z\in\mathbb{R}^d$ is arbitrary, the previous equations are equivalent to
    \begin{equation*}
    \sum^n_{i=1}a_i^2\Sigma_{\widetilde{Q}_{\bx_1}}\#\Sigma_{\nu_i}=\sum^n_{i=1}a_i^2I_d.
    \end{equation*}
    If we multiply this equation on both sides by $(\Sigma_{\widetilde{Q}_{\bx_1}})^{1/2}$ we get
        \begin{equation*}
        \sum^n_{i=1}a_i^2 (\Sigma_{\widetilde{Q}_{\bx_1}}^{1/2}\Sigma_{\nu_i}\Sigma_{\widetilde{Q}_{\bx_1}}^{1/2})^{1/2}=\sum^n_{i=1}a_i^2\Sigma_{\widetilde{Q}_{\bx_1}}
        \end{equation*}
\end{proof}

\begin{proof}[Proof of Proposition \ref{prop:full_gaussian_gf}.]

We proceed to treat each item separately.
\begin{enumerate}
    \item If we specialize this notion for a $G$ depending on Gaussian responses $\nu_i\in\cP_2(\RR^d)$ with mean $m_{\nu_i}$ and variance $\Sigma_{\nu_i}$, then the gradient flow equation reads as
\begin{equation*}
    \begin{cases}
    \dot{Q^t}=\frac{1}{n}\sum^n_{i=1}\bx_i\bx_i^\top(\cdot)(\Sigma_{Q^t_{\bx_i}}\#\Sigma_{\nu_i}- I_d)+\frac{1}{n}\sum_{i=1}^n \bx_i( m_{Q^t_{\bx_i}}^\top \Sigma_{Q^t_{\bx_i}}\#\Sigma_{\nu_i} - m_{\nu_i}^\top)\\
        Q^0= \mathcal{N}_{dp}(0,\Sigma_{Q^0}).
    \end{cases}
\end{equation*}

So, if we sample a random variable  $\bB^\top_t\sim Q^t$ at time $t>0$, then previous equation reads as
\begin{equation}\label{eq:full-particle}
    \dot{\bB}_t=v_t(\bB_t)=\frac{1}{n}\sum^n_{i=1}\bx_i\bx_i^\top\bB_t(\Sigma_{Q^t_{\bx_i}}\#\Sigma_{\nu_i}- I_d)+\frac{1}{n}\sum_{i=1}^n \bx_i( m_{Q^t_{\bx_i}}^\top \Sigma_{Q^t_{\bx_i}}\#\Sigma_{\nu_i} - m_{\nu_i}^\top)
\end{equation}
under the particle interpretation of the continuity equation.
Let's investigate the vectorized form of $\dot{\bB}^\top_t$.

Note the vectorized form of the transpose of the first summand of equation \eqref{eq:full-particle} simplifies to
\begin{equation}\label{eq:first_summand_vec}
    (\frac{1}{n}\sum^n_{i=1}\bx_i\bx_i^\top \otimes(\Sigma_{Q^t_{i}}\#\Sigma_{\nu_i}- I_d))\vect{(\bB^\top_t)}
\end{equation}
by using the vector identities $\vect{(ABC)}=(C^\top\otimes A)\vect{(B)}$.

But, since $m_{Q_{\bx_i}^t}=(\bx_i\otimes I_d)m_Q$,  the vectorized form of the transpose of the second summand of  \eqref{eq:full-particle} is
\begin{align*}
   \frac{1}{n} \left[ \left( \sum_{i=1}^n \bx_i \bx_i^\top \otimes \Sigma_{Q^t_{\bx_i}}\#\Sigma_{\nu_i} \right) m_{Q^t} - \sum_{i=1}^n \bx_i \otimes m_{\nu_i} \right],
\end{align*}
where we used the mixed-product property for Kronecker products $(A\otimes B)(C\otimes D)=(AC)\otimes (BD)$.

Putting these two expressions back together and taking expectation with  respect to $Q^t$ yields that
\begin{equation*}
    \dot{m}_{Q^t}=\sum^n_{i=1}\frac{1}{n}\left(\bx_i\bx_i^\top \otimes(2\Sigma_{Q^t_{i}}\#\Sigma_{\nu_i}- I_d)\right)m_{Q^t} -\sum_{i=1}^n \frac{1}{n}\bx_i \otimes m_{\nu_i}
\end{equation*}
for all $t>0$.

In order to derive an expression for the variance of $Q^t$, recall that $$\Sigma_{Q^t}=\E[\vect{(\bB^\top-m_{Q^t})}(\vect{(\bB^\top)-m_{Q^t}})^\top]$$ Taking the derivative of $\Sigma_{Q^t}$ over time gives
\begin{align*}
    \dot{\Sigma}_{Q^t}&=\partial_t\mathbb{E}[\vect{(\bB^\top-m_{Q^t})}(\vect{(\bB^\top)-m_{Q^t}})^\top]\\&=\mathbb{E}[\vect{(\dot{\bB}^\top_t-\dot{m}_{Q^t})}\vect{(\bB^\top_t-m_{Q^t})}^{\top}+\vect{(\bB^\top_t-m_{Q^t})}\vect{(\dot{\bB}^\top_t-\dot{m}_{Q^t})}^{\top}].
\end{align*}
But vectorizing $\dot{\bB}^\top_t-\dot m_{Q^t}$ yields
\begin{equation*}
    \vect{(\dot{\bB}^\top_t-\dot m_{Q^t})}=M_t\vect{(\bB^\top_t-m_{Q^t})}
\end{equation*}
where we define $M_t:=(\frac{1}{n}\sum^n_{i=1}\bx_i\bx_i^\top \otimes(\Sigma_{Q^t_{\bx_i}}\#\Sigma_{\nu_i}- I_d))$ for every $t>0$.
Putting the two previous equations together yields the Lyapunov equation
\begin{equation*}
    \dot{\Sigma}_{Q^t}=M_t\Sigma_{Q^t}+\Sigma_{Q^t}M_t.
\end{equation*}
\item By Theorem 5.22 in \cite{ChewiNilesWeedRigollet2025}, we just implicitly proved that for any $Q=\mathcal{N}(0,\Sigma_Q)$ in $\cP(\mathbb{R}^{p\times d})$
$$M_k:=-\E_Q[\nabla^2_\bB \delta G(Q)]=(\frac{1}{n}\sum^n_{i=1}\bx_i\bx_i^\top \otimes(\Sigma_{Q_{\bx_i}}\#\Sigma_{\nu_i}- I_d)).$$
Then the update equation  \eqref{eq:gradient_descent_G} reads
\begin{equation*}
    Q^{k+1}=\exp_{Q^k}(\tau\WGrad_{Q^k} G)=\mathcal{N}\left(0, (\tau M_k + I_{dp})\Sigma_{Q^k}(\tau M_k + I_{dp})\right).
\end{equation*}
\end{enumerate}
\end{proof}

\end{document}